\documentclass[preprint]{elsarticle}

\usepackage{amssymb}
\usepackage{indentfirst}
\usepackage{graphicx}
\usepackage{epstopdf}
\usepackage{epsfig}
\usepackage{overpic}
\usepackage{array}
\usepackage{color}
\usepackage{float}
\usepackage{amsmath,amssymb, amsthm}
\usepackage[mathscr]{euscript}
\usepackage{latexsym,bm}
\usepackage{booktabs}
\usepackage{bbm}
\usepackage{diagbox}
 \usepackage[all]{xy}
\usepackage{diagbox}
\usepackage{algorithm}  
\usepackage{algorithmicx} 
\usepackage{algpseudocode} 
\usepackage{mathtools} 

\usepackage{multirow}
\usepackage{subcaption}   % 推荐（不要用老的 subfigure）
\usepackage[colorlinks=true,
linkcolor=blue,
citecolor=blue,
urlcolor=blue]{hyperref}
\usepackage{empheq}

\usepackage{graphicx}
\usepackage{subcaption}

\newtheorem{theorem}{Theorem}
\newtheorem{lemma}{Lemma}

\newtheorem{remark}{Remark}
\newtheorem{example}{\textup{\textbf{Example}}}

\def\br{\bm{x}}

\def\br{\bm{r}}

\makeatletter
\@addtoreset{equation}{section}
\makeatother

\begin{document}

\newsavebox{\tablebox}

\begin{frontmatter}



\title{Nonlocal modeling of opinion alignment and environmental feedback: 
	Spatial aggregation and non-consensus patterns}

\author[label1]{Rui Wang}
\ead{rwang0913@mail.bnu.edu.cn}

\author{Yunfeng Xiong\corref{cor1}\fnref{label1,label2}}
\ead{yfxiong@bnu.edu.cn}

\author[label1,label2]{Zhengru Zhang}
\ead{zrzhang@bnu.edu.cn}

\author[label3]{Xiaofei Zhao}
\ead{matzhxf@whu.edu.cn}

\cortext[cor1]{To whom correspondence should be addressed}
           
\affiliation[label1]{organization={School of Mathematical Sciences, Beijing Normal University},	
%Department and Organization
 %           addressline={}, 
   %         city={},
    postcode={100875},
    state={Beijing},      country={P. R. China}
 }

\affiliation[label2]{organization={Laboratory of Mathematics and Complex Systems, Ministry of Education, Beijing Normal University},
postcode={100875},
state={Beijing},      country={P. R. China}}

\affiliation[label3]{organization={School of Mathematics and Statistics and Computational Sciences Hubei Key Laboratory, Wuhan University},
	city={Wuhan},
	postcode={430072},
	state={Hubei},      country={P. R. China}}

\begin{abstract}
The formation of public opinion in modern information environments is
shaped by the interplay between social conformity and information
exposure. While social interactions promote opinion alignment,
heterogeneous visibility and selective exposure may reinforce local
agreement, a mechanism commonly associated with the echo chamber effects.
To describe how such reinforcement influences spatially heterogeneous
opinion activity and non-consensus patterns, we propose a spatial opinion
dynamics model with attention-mediated feedback.
The model couples nonlocal alignment with an evolving attention field and
captures a self-reinforcing mechanism in which regions of high opinion
activity attract greater visibility. Starting from  agent-based jump mechanism inspired by bounded confidence interactions and biased random walks induced by environments,
we formally derive a nonlocal advection-cross-diffusion system, where
opinion transport is driven by nonlocal conformity and modulated by
attention-dependent redistribution.
We characterize the transition from spatially homogeneous consensus
states to non-consensus clustered regimes through linear stability
analysis of the homogeneous equilibrium. 
The results show that
attention-mediated feedback  has an explicit correction on the instability threshold and
enlarges the parameter regime in which clustering occurs, thereby
promoting persistent spatial heterogeneity and non-consensus
patterns.
Numerical simulations based on a structure-preserving IMEX
spectral method support the theoretical predictions and quantify the
resulting aggregation phenomena. In particular, both the selected wavelength and the cluster spacing scale as $2R$ in the weak diffusion regime, resembling the 2R-conjecture in the agent-based modeling. These findings provide a macroscopic
description of how nonlocal alignment and environmental feedback
jointly shape spatial signatures of non-consensus patterns.
\end{abstract}

\begin{keyword}
	Opinion dynamics \sep non-consensus \sep nonlocal aggregation \sep linear stability \sep structure-preserving schemes \sep phase transition

\MSC[2020]
35K55 \sep	%Linear waves in solid mechanics
35B36 \sep 	%Numerical computation using splines
35Q91 \sep 	%Numerical solution of discretized equations for initial value and initial-boundary value problems involving PDEs
65M70 \sep	%Parallel numerical computation
92D25	%Distributed algorithms

% or \MSC[2008] code \sep code (2000 is the default)

\end{keyword}

\end{frontmatter}

\section{Introduction}

Public opinion plays a central role in shaping collective decisions and 
social dynamics in modern societies. In many real situations, opinion formation exhibits different spatial patterns across multiple scales \cite{BalsaBarreiro2022SocialSpace}.
This is illustrated, for example, by activities of public opinion polls and elections,
which show pronounced spatial variation in voting preferences. Local
communities may display relatively coherent support for a candidate,
while different regions exhibit distinct patterns of support across space
\cite{BalsaBarreiro2022SocialSpace,Bujalski2018Consensus}.
A fundamental question in opinion
formation is how information is shared and reinforced through social
interactions, and how the structure of these interactions shapes collective
outcomes \cite{LiScaglioneSwamiZhao2013,Peralta2025OpinionDynamics}.

A wide range of mathematical frameworks has been developed to describe
opinion and information dynamics across different modeling scales.
At the microscopic level, agent-based models prescribe interaction rules
for individual opinion updates. Representative examples include the
DeGroot consensus model \cite{DeGroot1974}, the Friedkin-Johnsen model
\cite{FriedkinJohnsen1990}, and bounded-confidence models such as the
Hegselmann-Krause (HK) and Deffuant-Weisbuch (DW) models
\cite{Deffuant2000,HegselmannKrause2002,Weisbuch2002}. Related
discrete-opinion models, including voter-type models
\cite{CliffordSudbury1973} and Sznajd dynamics \cite{Sznajd2000}, have
also been studied in the context of consensus formation and multi-cluster
dynamics. These models capture consensus and clustered opinion
configurations, with extensive developments in convergence, stability,
state-dependent connectivity, and generalized bounded-confidence
dynamics \cite{vanAlebeekCator2025,Blondel2009,ProskurnikovTempo2017}.
Recent extensions of agent-based opinion models have incorporated
spatial structure and heterogeneous interaction rules, showing that
spatial effects can lead to localized opinion clusters and spatially
separated opinion groups
\cite{BaumgaertnerTysonKrone2016,BaumgaertnerFetrosKroneTyson2018,PasimeniWadeAlkemade2025}.
Mobile-agent systems have also been shown to generate metapopulations
and echo chambers \cite{StarniniFrascaBaronchelli2016}.
At the mesoscopic scale, kinetic descriptions provide a statistical framework for opinion distributions. Starting from binary interaction rules, one obtains Boltzmann-type equations and their Fokker-Planck
limits \cite{During2009,PareschiToscani2013,Toscani2006}, thereby
bridging microscopic interaction laws and macroscopic descriptions.
Spatially inhomogeneous kinetic models have also incorporated spatial
position into the opinion distribution, with applications to the
spatial clustering of like-minded groups motivated by ``The Big Sort''
\cite{DuringWolfram2015OpinionDynamics}.
At the macroscopic level, continuum PDE models provide a useful framework for describing the spatiotemporal redistribution of activity densities and aggregation phenomena in collective systems
\cite{BellomoLiaoQuainiRussoSiettos2023}. 
Nonlocal transport and aggregation equations have been widely used to study concentration and self-organization
\cite{CarrilloFornasierToscani2010,CarrilloMcCannVillani2003}.
Related reaction-diffusion and cross-diffusion models have also been used to describe hotspot formation in crime models
\cite{ShortDOsognaPasourTitaBrantinghamBertozziChayes2008,ShortBertozziBrantingham2010}, as well as aggregation phenomena in
biological systems \cite{TopazBertozziLewis2006}. Such macroscopic
approaches are particularly suitable when the state variable represents a spatially distributed level of activity rather than a single scalar opinion carried by each agent.

Despite these developments, many existing models primarily focus on conformity-driven alignment.  In modern information environments, however, regions with stronger opinion activity may receive more attention. Such spatial heterogeneity suggests that collective opinions are shaped
not only by interactions within a local community, but also by external
influences, or ``outside voices'', including media exposure,
geographically distant social ties, and other forms of information
exchange across regions \cite{Bujalski2018Consensus,Peralta2025OpinionDynamics}.The resulting feedback can reinforce local activity and feed back into the spatial distribution of opinion density.
It is therefore important to
understand how it affects stability, clustering, and the
transition from spatially homogeneous states to non-consensus regimes.

%Despite these developments, the dynamic feedback between opinion activity
%and the surrounding information environment remains less understood in
%spatially extended models. When attention is included, it is often treated as a static or externally
%prescribed factor, rather than as an evolving field coupled to opinion
%dynamics. and alter the balance between nonlocal alignment and
%diffusion. 

Motivated by these observations, we develop a spatial opinion dynamics
model that captures the interaction between local opinion activity and
the information environment. The model is built upon two key
mechanisms. The first is \textbf{social conformity}, through which opinion activity tends to align with that in a surrounding interaction neighborhood,
promoting local agreement and, under suitable conditions, large-scale
consensus \cite{BellomoHaKimLiao2026,Deffuant2000,DeGroot1974,FriedkinJohnsen1990,HegselmannKrause2002}.
Depending on the interaction structure, such dynamics may also give rise
to multiple coexisting opinion groups and clustered patterns
\cite{BenNaim2005,CastellanoFortunatoLoreto2009,Lorenz2007,Weisbuch2002}. 
The second is \textbf{attention-mediated information exposure}, through which regions with
stronger opinion activity attract more attention and thereby reinforce
the spatial concentration of opinion activity, leading to an uneven
distribution of attention and influence across space
\cite{AlbiPareschiZanella2014,BakshyMessingAdamic2015,CraneSornette2008,WuHuberman2007}.
The interplay between these two mechanisms can generate spatially
localized opinion patterns and persistent non-consensus states, in which
several regions of high opinion activity coexist rather than merging into
a spatially homogeneous consensus.
In modern information environments, this feedback is closely related to
selective exposure and \textbf{echo chamber effects}, where repeated
exposure to reinforcing information can strengthen local agreement,
further concentrate opinion activity, and reduce effective mixing across different groups
\cite{BakshyMessingAdamic2015,CinelliMoralesGaleazziQuattrociocchiStarnini2021,HolmeNewman2006,RaoufiHamannRomanczuk2025,Vosoughi2018}.

%To capture the dynamic feedback between opinion activity and attention and to investigate its impact on clustering and non-consensus behavior, 
%we introduce and analyze a spatial opinion dynamics model with attention reinforcement, 
%designed to describe the emergence of non-consensus states induced by the interplay between social conformity and echo chamber effects. 

Starting from a agent-based model involving both jump mechanisms inspired by bounded confidence interactions and biased random walks \cite{Short2010Crime,ShortDOsognaPasourTitaBrantinghamBertozziChayes2008} induced by environments, we take the continuum limit and derive a macroscopic PDE system that takes the form of a nonlocal advection-cross-diffusion system with attention-dependent transport, 
referred to as the \emph{NODAR (Nonlocal Opinion Dynamics with Attention Reinforcement)} system. Specifically, the model couples a conserved opinion density $\rho(x,t)$,
which represents the spatial distribution of opinion activity or support
on a given topic, with an evolving attention field $S(x,t)$,
and reads
\begin{equation}\label{eq:intro_rhoS}
	\left\{
	\begin{aligned}
		\partial_t \rho
		&=-\kappa_{\mathrm{conf}}\nabla \cdot (\rho V[\rho])
		+ D_\rho \nabla \cdot \left(
		\nabla \rho
		-
		\frac{2\rho}{A}\nabla A
		\right), \\
		\partial_t S
		&=
		D_S \Delta S
		-
		\omega S
		+
		\theta \rho,
	\end{aligned}
	\right.
\end{equation}
where the effective attention field is given by
$A(x,t)=A^0(x)+S(x,t)$,
with $A^0(x)$ denoting a baseline attention level and $S(x,t)$ an activity-driven component.
The nonlocal alignment velocity is defined as
\begin{equation}\label{eq:intro_V}
	V[\rho](x,t)
	=
	\frac{\int_{\Omega} K_R(y-x)(y-x)\rho(y,t)\,dy}
	{\int_{\Omega} K_R(y-x)\rho(y,t)\,dy + \varepsilon}.
\end{equation}
Here $K_R$ is an interaction kernel with interaction radius $R>0$, and $\varepsilon>0$ is a regularization parameter. 
The coefficient $\kappa_{\mathrm{conf}}>0$ denotes the strength of consensus-driven interactions. 
The parameters $D_\rho$ and $D_S$ are the diffusion coefficients of the opinion density and the attention field, respectively, 
while $\omega>0$ represents the decay rate of attention and $\theta>0$ quantifies its reinforcement by local opinion activity. 
A complete list of variables and parameters is summarized in Table~\ref{tab:variables}.

The NODAR system captures the interplay between nonlocal alignment and attention-mediated feedback. 
Nonlocal interactions drive opinion alignment, while attention reinforcement induces a feedback loop in which regions of high opinion activity attract greater visibility and further enhance aggregation. The analysis shows that attention-mediated feedback enters the instability
threshold through an explicit mode-dependent correction factor \(\Gamma(k,D_S) < 1\) and might potentially enlarge the instability regime.
This mechanism promotes spatial clustering, reduces effective mixing, and may prevent convergence to a homogeneous consensus state, leading to persistent non-consensus behavior.
%To capture the dynamic feedback between opinion activity and attention and to investigate its impact on clustering and non-consensus behavior, 
%we introduce and analyze a spatial opinion dynamics model with attention reinforcement, 
%designed to describe the emergence of non-consensus states induced by the interplay between social conformity and echo chamber effects. 
Another related issue concerns the characteristic scale of the resulting
patterns. While bounded-confidence models predict cluster formation and, in some
discrete-agent settings, the classical 2R-type separation phenomenon
\cite{BlondelTsitsiklis2007,WangLiEChazelle2017}, 
its analogue in
continuum models remains less understood. Our numerical simulations reveal that both  the selected wavelength and cluster spacing are close to \(2R\) in the weak diffusion regime, providing some evidence on the 2R-conjecture in a general setting.

%{\cf
	%In the NODAR system, the linear dispersion relation provides a continuum
	%wavelength-selection mechanism. In the weak-diffusion,
	%aggregation-dominated regime, the selected wavelength is close to the
	%\(2R\) scale. Our numerical simulations via a structure-preserving IMEX spectral method reveal that the late-time cluster spacing remains
	%proportional to \(R\) and is close to the \(2R\) scale. These results
	%are consistent with the \(2R\) scale associated with
	%bounded-confidence particle models.
	%}
%Our numerical simulations via a structure-preserving
%IMEX spectral method reveal that the PDE system might also admit an intrinsic length scale for pattern
%formation,  depending on the interaction radius in the
%presence of attention-mediated feedback.

%In particular, it remains
%unclear whether such systems admit an intrinsic length scale for pattern
%formation and how this scale depends on the interaction radius in the
%presence of attention-mediated feedback.
%its analogue in spatially extended
%continuum models is still not well understood. In particular, it remains
%unclear whether such systems admit an intrinsic length scale for pattern
%formation and how this scale depends on the interaction radius in the
%presence of attention-mediated feedback.

To summarize, the main contributions of this work are given as follows:
\begin{itemize}
	
	\item[(1)] 
	We propose a spatial opinion dynamics model with attention reinforcement and derive its continuum formulation as a nonlocal advection-cross-diffusion system. 
	The model incorporates the feedback between opinion activity and collective attention, providing a framework for describing the interaction between nonlocal alignment and environmental feedback.
	
	\item[(2)] 
	We establish the threshold conditions for the onset of aggregation.
	%	The stability analysis shows that attention-mediated feedback modifies the instability threshold
	%	through an explicit mode-dependent correction factor \(\Gamma(k,D_S)\). When
	%	\(0<\Gamma(k,D_S)<1\), 
	The attention-mediated feedback enlarges the instability regime and promotes clustering, and induces a transition from spatially homogeneous states
	to non-consensus regimes.
	Numerical simulations further provide some evidence on the 2R-conjecture in the PDE level.

	%	\item[(2)] 
	%	The resulting system is analyzed to establish threshold conditions for the onset of aggregation.
	%	The stability analysis shows that attention-mediated feedback modifies the instability threshold,
	%	promotes clustering, and induces a transition from spatially homogeneous states
	%	to non-consensus regimes. Furthermore, numerical simulations provide some evidence on the 2R-conjecture in the PDE level.

	%	{\cf In addition, the dispersion relation
		%		identifies a characteristic wavelength, and numerical simulations show
		%		that both the selected wavelength and the nonlinear cluster spacing scale
		%		linearly with the interaction radius, approaching the \(2R\) scale in the
		%		weak-diffusion aggregation regime.}
	%	In addition, a characteristic pattern scale is identified, with an approximately
	%	linear dependence on the interaction radius that is consistent with numerical observations.
	
\end{itemize}

%\iffalse
%The NODAR system captures the interplay between nonlocal alignment and
%attention-mediated feedback.
%Social conformity is modeled by nonlocal alignment, while echo chamber
%effects arise from the reinforcement of attention and the induced
%feedback between opinion activity and collective attention.
%This feedback mechanism promotes the formation of spatially localized
%high-density regions, reduces effective mixing, and may prevent the
%system from reaching a spatially homogeneous consensus state.
%Within this framework, fragmentation is characterized not merely by the
%coexistence of different opinions, but by persistent spatial
%heterogeneity and clustered opinion activity.
%We investigate how the interaction between nonlocal alignment,
%diffusion, and attention reinforcement leads to aggregation phenomena
%and fragmented pattern formation.
%The threshold conditions for the onset of aggregation, which also mark the
%loss of global consensus, are given via linear stability analysis of the spatially homogeneous state.
%The analysis shows that attention reinforcement can lower the instability
%threshold and accelerate the emergence of clustering.
%Numerical simulations further confirm these findings and demonstrate
%that attention feedback leads to more localized and concentrated
%structures, where opinion activity is confined to a small number of
%high-density regions.
%Furthermore, the dynamics exhibit a characteristic pattern scale
%that depends on the interaction radius, suggesting a possible
%continuum analogue of the classical $2R$ conjecture.
%\fi

The rest is organized as follows.
Section~\ref{sec:modeling} introduces the modeling framework and
derives the macroscopic PDE system.
Section~\ref{sec:phase_stability} analyzes phase transition via the linear
stability, and derives threshold conditions for aggregation.
Section~\ref{sec:simulation} presents the numerical scheme and
simulation results, including phase diagrams and clustering behavior.
Finally, conclusions are given in Section~\ref{sec:conclusion}.

\section{Modeling} \label{sec:modeling}

%We begin by introducing the notion of topic visibility, 

For simplicity, we consider a
two-dimensional periodic spatial domain
\(\Omega=\mathbb T^2\). The domain is discretized by a uniform square
lattice with mesh size \(\ell>0\). We denote the grid nodes by
\(x_s=x_{i,j}\), where \(s=(i,j)\) and \(i,j=0,\ldots,N-1\).
%Time is discretized as $t_k = k\Delta t$,
%$k \in \mathbb{N}$.
The state variables and model parameters are summarized in
Table~\ref{tab:variables}.

\begin{table}[h!]
	\centering
	\caption{State variables and parameters of the agent-based model.}
	\label{tab:variables}
	\renewcommand{\arraystretch}{1.20}
	\begin{tabular}{lp{0.7\textwidth}}
		\hline\hline
		Symbol & Description \\
		\hline
		
		$n_{i,j}(t)$  & Opinion level at node $(i,j)$ at time $t$ \\
		
		$\rho_{i,j}(t) = n_{i,j}/\ell^2$
		& Opinion density at node $(i,j)$ at time $t$ \\
		
		$S_{i,j}(t)$
		& Activity-driven attention  at node $(i,j)$ at time $t$ \\
		
		$A^0_{i,j}$
		& Static baseline attention at node $(i,j)$ \\
		
		$A_{i,j}(t) $
		& Effective attention at node $(i,j)$ at time $t$ \\
		
		\hline
		$R>0$ & Interaction range of nonlocal conformity\\
		
		$\kappa_{\mathrm{conf}}>0$ & Strength of nonlocal conformity alignment\\
		
		$D_\rho>0$
		& Diffusion coefficient of the opinion density \\
		
		$D_S>0$
		& Diffusion coefficient of the dynamic attention field \\
		
		$\omega>0$
		& Decay rate of the dynamic attention field \\
		
		$\theta>0$
		& Reinforcement strength of attention by opinion activity \\
		\hline\hline
	\end{tabular}
\end{table}

\subsection{State variables}
\label{subsec:domain-variables}

%In this section, we derive the NODAR system \eqref{eq:intro_rhoS}. 

Topic visibility plays a central role in the formation and spatial redistribution of public opinion. 
Even in the absence of ongoing activity, a topic may retain a nonzero level of exposure due to persistent user interest, platform recommendation mechanisms, or structural heterogeneity. 
This observation motivates a decomposition of the effective visibility into a baseline component and an activity-driven component, which together define the effective attention level variable at $(i, j)$.
\begin{equation}\label{eq:A-decomposition}
	A_{i, j}(t) = A^0_{i, j} + S_{i, j}(t),
\end{equation}
where $A^0_{i, j}$ denotes a time-independent baseline visibility that may
vary spatially, and $S_{i, j}(t)$ represents the endogenous attention
generated by the evolving opinion activity. The attention field $S$ is
reinforced in regions of high opinion density, spreads through local
mixing, and decays over time. For technical convenience, it is assumed
that $A^0_{i, j}$ is bounded away from zero so that the effective attention
$A_{i, j}(t)$ remains strictly positive
\cite{ShortBertozziBrantingham2010,Short2010Crime,ShortDOsognaPasourTitaBrantinghamBertozziChayes2008}.

\begin{figure}[h!]
	\centering
	\includegraphics[width=0.6\textwidth]{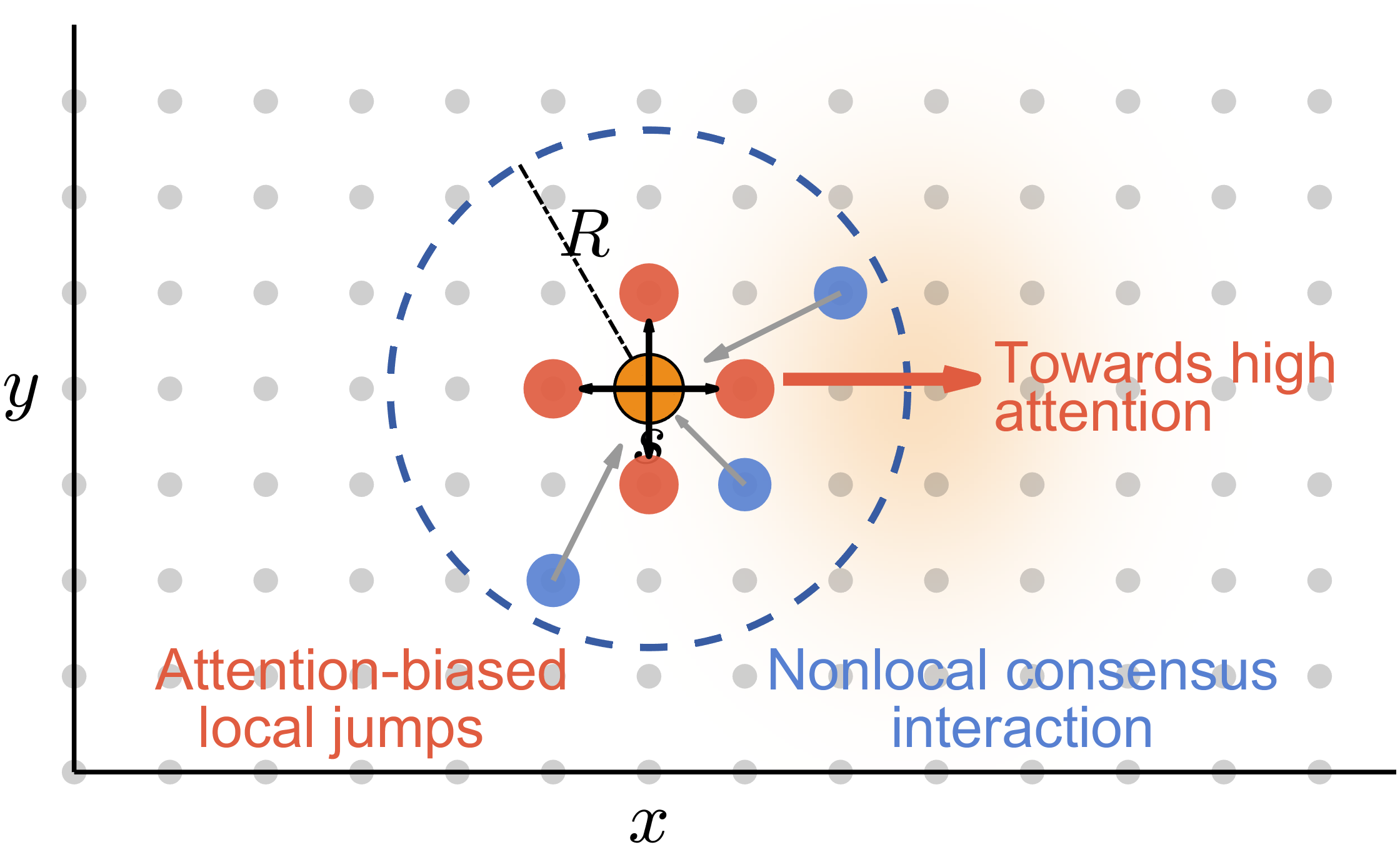}
	\caption{
		Schematic illustration of the microscopic mechanisms in the model.
		The opinion level at site $s$ evolves through two components:
		(i) nearest-neighbour jumps biased toward regions of higher attention (red nodes), and
		(ii) nonlocal conformity interactions within a radius $R$ (blue nodes).
	}
	\label{fig:schematic_model}
	\vspace{-10pt}
\end{figure}

To describe the underlying mechanisms at the microscopic level, we
consider a agent-based representation of the system, as illustrated
in Fig.~\ref{fig:schematic_model}. 
Each site $(i,j)$ of the two-dimensional lattice carries an
opinion-activity level $n_{i,j}$ associated with the topic and an
effective attention level $A_{i,j}$ \cite{BalsaBarreiro2022SocialSpace}. 
At the individual level, opinion interactions may lead to several basic responses: (1) keep, where the opinion
of a person remains unchanged; (2) adopt, where the subject copies a reference opinion; (3) compromise, where the individual approaches the reference opinion \cite{ChacomaZanette2015,BalsaBarreiro2022SocialSpace}.
%	In the present lattice model, these microscopic responses are not
%	distinguished separately; instead, their collective effect is incorporated
%	through the redistribution of the opinion-activity level $n_{i,j}$.	

The redistribution of the opinion-activity level $n_{i,j}$ is driven by two mechanisms:  Biased random walk 
toward the nearest neighborhoods and nonlocal conformity interactions within a finite interaction radius $R$.
%The opinion level $n_{i,j}$ at site $(i,j)$ represents the aggregate
%level of support or opinion activity associated with a given topic at that location
Based on this microscopic picture, the agent-based model is constructed
in three steps:
\begin{enumerate}
	\item[(1)] The physical or platform space is represented as a regular
	two-dimensional lattice, where each cell carries an opinion level and
	an attention level.
	
	\item[(2)] The redistribution of opinion level between lattice sites is
	modeled by a continuous-time jump process 
	%with a conservative master equation,
	incorporating both attention bias and nonlocal conformity alignment.
	
	\item[(3)] A dynamic attention field is introduced on the lattice, whose
	evolution includes local averaging, temporal decay, and feedback from the
	opinion level.
	
\end{enumerate}

A continuum limit is derived by a transport approximation for the nonlocal interactions, leading to a coupled nonlocal advection-cross-diffusion PDE system.

\subsection{agent-based microscopic modeling}
%\label{subsec:rho-derivation}

The redistribution of opinion level between lattice sites is modeled as a
continuous-time jump process with two distinct mechanisms, corresponding
to nonlocal conformity and attention-driven local motion. During a short
time interval, the opinion activity located at node $(i,j)$ may jump either to a
nearest-neighbour site or to another site within a wider interaction
neighborhood.

At the microscopic level, the master equation of the opinion level $n_{i, j}$ reads
\begin{equation}\label{eq:master-n-general}
	\frac{d n_{i,j}}{dt}
	=
	\sum_{\br}
	\Big(
	\lambda_{\br \to(i,j)}(t)\,n_{\br}(t)
	-
	\lambda_{(i,j)\to \br}(t)\,n_{i,j}(t)
	\Big),
\end{equation}
where the sum runs over all sites $\br$ that can be reached from $(i,j)$
in a single jump, and $\lambda_{(i,j)\to \br}(t)\ge0$ denotes the
corresponding jump intensity. Since the right-hand side of
\eqref{eq:master-n-general} is written in terms of pairwise gain-loss
fluxes, the total opinion level
$N(t):=\sum_{i,j} n_{i,j}(t)$
is conserved for all $t$.

We consider two types of jump mechanisms, providing a microscopic
realization of the two key driving effects: Social
conformity and attention-mediated feedback arising from the echo
chamber effect.
\begin{enumerate}
	\item[(1)] Conformity-driven nonlocal jumps within an interaction
	neighborhood $\mathcal{N}_R(i,j)$ of radius $R$, with intensities
	$\lambda^{\mathrm{conf}}_{(i,j)\to \br}(t)$;
	
	\item[(2)] Attention-driven local jumps to
	$\br\sim(i,j)$, where $\br\sim(i,j)$ denotes the four nearest
	neighbors of $(i,j)$, reflecting attention-induced local mobility,
	with intensities $\lambda^{\mathrm{att}}_{(i,j)\to \br}(t)$.
\end{enumerate}

Accordingly, the total jump intensity is decomposed as
\[
\lambda_{(i,j)\to \br}(t)
=
\lambda^{\mathrm{conf}}_{(i,j)\to \br}(t)
+
\lambda^{\mathrm{att}}_{(i,j)\to \br}(t),
\]
and the microscopic dynamics becomes
\begin{equation}\label{eq:master-n-unified}
	\begin{aligned}
		\frac{d n_{i,j}}{dt}
		&=
		\sum_{\br\in\mathcal{N}_R(i,j)}
		\Big(
		\lambda^{\mathrm{conf}}_{\br \to(i,j)}(t)\,n_{\br}(t)
		-
		\lambda^{\mathrm{conf}}_{(i,j)\to \br}(t)\,n_{i,j}(t)
		\Big)
		\\
		&\quad
		+
		\sum_{\br\sim(i,j)}
		\Big(
		\lambda^{\mathrm{att}}_{\br \to(i,j)}(t)\,n_{\br}(t)
		-
		\lambda^{\mathrm{att}}_{(i,j)\to \br}(t)\,n_{i,j}(t)
		\Big).
	\end{aligned}
\end{equation}

\subsubsection{Conformity-driven nonlocal jumps}
%We first consider the contribution of nonlocal conformity-driven jumps, 
%\begin{equation}\label{eq:J-discrete}
%	\mathcal J^{\ell}_{i,j}
%	:=
%	\frac{1}{\ell^2}
%	\sum_{(p,q)\in\mathcal N_R(i,j)}
%	\Big(
%	\lambda^{\mathrm{conf}}_{(p,q)\to(i,j)}\,n_{p,q}
%	-
%	\lambda^{\mathrm{conf}}_{(i,j)\to(p,q)}\,n_{i,j}
%	\Big),
%\end{equation}
%where, by symmetry of the interaction neighborhood, both sums are taken over the same set $\mathcal N_R(i,j)$.
Motivated by bounded-confidence interactions in opinion dynamics
\cite{Deffuant2000,HegselmannKrause2002}, we introduce conformity-driven nonlocal jumps on the lattice.
The associated jump intensities are specified by decomposing them
into an overall interaction frequency and a regularized transition
weight:
\begin{equation}\label{eq:lambda-HK}
	\lambda^{\mathrm{conf}}_{(i,j)\to(p,q)}(t)
	=
	\kappa_{\mathrm{conf}}\,
	\pi_{(i,j)\to(p,q)}(t),
\end{equation}
where $\kappa_{\mathrm{conf}}>0$ denotes the overall jump rate, and
\begin{equation}\label{eq:pi-HK}
	\pi_{(i,j)\to(p,q)}(t)
	=
	\frac{
		K_R(x_{p,q}-x_{i,j})\,\rho_{p,q}(t)\,\ell^2}
	{\displaystyle
		\sum_{(m,n)\in\mathcal{N}_R(i,j)}
		K_R(x_{m,n}-x_{i,j})\,\rho_{m,n}(t)\,\ell^2+\varepsilon}
\end{equation}
defines a regularized jump preference from $(i,j)$ to $(p,q)$.
Here $K_R$ is an isotropic interaction kernel supported in $\|z\|\le R$,
and $\varepsilon>0$ is a small regularization parameter preventing
degeneracy of the normalization. Due to the presence of $\varepsilon$,
the weights \eqref{eq:pi-HK} are not required to sum to one exactly.

\subsubsection{Attention-driven biased random walk}
In the spirit of the attractiveness-based biased random-walk mechanism used in crime-hotspot models \cite{ShortDOsognaPasourTitaBrantinghamBertozziChayes2008}, we model the attention field as an environmental bias that directs local jumps toward neighboring sites with higher attention.
For the attention-driven biased random walk, we assume that each unit of opinion level
located at site $\bm{s}=(i,j)$ undergoes jump attempts according to an
independent Poisson process with rate $\alpha_S>0$.
Over a short time interval $[t,t+\Delta t)$, the probability of a jump
event at site $\bm{s}$ is
\begin{equation*}
	p_{\bm{s}}(t)
	=
	1-e^{-\alpha_S\Delta t}
	=
	\alpha_S\Delta t + \mathcal{O}((\Delta t)^2),
	\quad (\Delta t \ll 1).
\end{equation*}
Conditioned on a jump event at site $\bm{s}$, the destination node $\bm{n}\sim \bm{s}$
is selected according to the probability distribution
\[
\mathbb{P}(\bm{s} \to \bm{n})
=
\frac{A_{\bm{n}}(t)}{T_{\bm{s}}(t)},
\quad
T_{\bm{s}}(t):=\sum_{\bm{n'}\sim \bm{s}} A_{\bm{n'}}(t).
\]
Therefore, for each nearest neighbour $\bm{n}\sim \bm{s}$, the transition probability
over the interval $[t,t+\Delta t)$ satisfies
\[
\mathbb{P}(X_{t+\Delta t}=\bm{n} \mid X_t=\bm{s}) = p_{\bm{s}}(t)\mathbb{P}(\bm{s} \to \bm{n})
=
\alpha_S\,\frac{A_{\bm{n}}(t)}{T_{\bm{s}}(t)}\,\Delta t
+ \mathcal{O}((\Delta t)^2).
\]

Accordingly, the corresponding attention-driven jump intensities are
\begin{equation*}
	\lambda^{\mathrm{att}}_{\bm{s}\to \bm{n}}(t)
	=
	\alpha_S\,\frac{A_{\bm{n}}(t)}{T_{\bm{s}}(t)},
	\quad \bm{n}\sim \bm{s}.
\end{equation*}

\subsubsection{Spread, decay and reinforcement of the attention level}

Unlike the stochastic transport of opinion level described above,
the attention level is modeled phenomenologically through a local
update rule on the lattice.
This rule captures spatial averaging of attention, temporal decay, and
reinforcement by local opinion density.
For a time step $\Delta t$, we prescribe
\begin{equation}\label{eq:S-lattice}
	S_{i,j}^{k+1}
	=
	\Big[(1-\eta) S_{i,j}^k
	+ \frac{\eta}{4}
	\sum_{(p,q)\sim(i,j)} S_{p,q}^k\Big]
	(1-\omega\Delta t)
	+
	\theta\rho_{i,j}^k\Delta t,
\end{equation}
where $\eta\in[0,1]$ controls the strength of spatial mixing.

We introduce the discrete Laplacian
\[
\Delta_{\mathrm{disc}} S_{i,j}^k
=
\frac{1}{\ell^2}
\Big(
\sum_{(p,q)\sim(i,j)} S_{p,q}^k
-
4S_{i,j}^k
\Big),
\]
so that \eqref{eq:S-lattice} can be written as
\begin{equation}\label{eq:S-agent-based}
	S_{i,j}^{k+1}
	=
	\left(
	S_{i,j}^k
	+
	\frac{\eta\ell^2}{4}
	\Delta_{\mathrm{disc}}S_{i,j}^k
	\right)
	(1-\omega\Delta t)
	+
	\theta\rho_{i,j}^k\Delta t.
\end{equation}

\subsection{Derivation of the macroscopic PDE system}
\label{subsec:rho-derivation}

%Now we assume that $\rho$ and $S$ vary smoothly on spatial and temporal
%scales larger than $\ell$ and $\Delta t$, 

Now we introduce the density variable
$\rho_{i,j}(t)=\frac{n_{i,j}(t)}{\ell^2}$ and rewrite the update in
incremental form, divide by $\Delta t$, and pass to the limit
$\ell,\Delta t\to 0$ with $D_S=\frac{\eta\ell^2}{4\Delta t}$
fixed. Since the mixed term involving both spatial averaging and temporal
decay is of higher order and vanishes in the continuum limit, Eq.~\eqref{eq:S-agent-based} becomes
\begin{equation}\label{eq:S-PDE}
	\partial_t S
	=
	D_S\Delta S
	-
	\omega S
	+
	\theta\rho,
\end{equation}
where $\rho$ and $S$ are continuum limit of $\rho_{i, j}$ and $S_{i, j}$, respectively. It remains to identify the continuum limit of the two jump
mechanisms. 

\subsubsection{Conformity-driven nonlocal advection}

Define the operator
$\mathcal J^\ell_{i,j}$ as a discrete redistribution of mass
induced by jumps of displacement $x_{p,q}-x_{i,j}$. 
\begin{equation}\label{eq:J-discrete}
	\mathcal J^{\ell}_{i,j}
	:=
	\frac{1}{\ell^2}
	\sum_{(p,q)\in\mathcal N_R(i,j)}
	\Big(
	\lambda^{\mathrm{conf}}_{(p,q)\to(i,j)}\,n_{p,q}
	-
	\lambda^{\mathrm{conf}}_{(i,j)\to(p,q)}\,n_{i,j}
	\Big),
\end{equation}
where, by symmetry of the interaction neighborhood, both sums are taken over the same set $\mathcal N_R(i,j)$.

%Although the conformity-driven and attention-driven parts are
%treated differently at the discrete level, both contribute to a
%conservative macroscopic evolution law for the density.

%Using the relation $n_{i,j}\sim \rho(x_{i,j},t)\ell^2$, 

To derive the continuum limit, we start from the weak form of Eq.~\eqref{eq:J-discrete}. For a smooth test function $\varphi$, it yields
\begin{equation*}
	\sum_{i,j}\mathcal J^\ell_{i,j}\,\varphi_{i,j}\,\ell^2
	=
	\kappa_{\mathrm{conf}}
	\sum_{i,j}
	\rho_{i,j}
	\sum_{(p,q)\in\mathcal N_R(i,j)}
	\pi_{(i,j)\to(p,q)}
	\bigl(\varphi_{p,q}-\varphi_{i,j}\bigr)\ell^2.
\end{equation*}
Expanding $\varphi_{p,q}$ around $x_{i,j}$ gives
\begin{equation}\label{expansion}
	\varphi_{p,q}-\varphi_{i,j}
	=
	(x_{p,q}-x_{i,j})\cdot \nabla\varphi(x_{i,j})
	+
	\frac12
	(x_{p,q}-x_{i,j})^\top
	D^2\varphi(x_{i,j})
	(x_{p,q}-x_{i,j})
	+\cdots.
\end{equation}

Retaining only the first term in Eq.~\eqref{expansion} (i.e., under a first-moment hydrodynamic closure) leads to the discrete mean
displacement
\begin{equation}\label{eq:intro_V-node}
	V_{i,j}
	=
	\sum_{(p,q)\in\mathcal{N}_R(i,j)}
	(x_{p,q}-x_{i,j})\,\pi_{(i,j)\to(p,q)},
\end{equation}
which characterizes the directional bias induced by the nonlocal jump
preference. Substituting \eqref{eq:pi-HK} into \eqref{eq:intro_V-node}, we
obtain
\begin{equation*}
	V_{i,j}
	=
	\frac{\displaystyle
		\sum_{(p,q)\in\mathcal{N}_R(i,j)}
		K_R(x_{p,q}-x_{i,j})
		(x_{p,q}-x_{i,j})\,\rho_{p,q}\,\ell^2}
	{\displaystyle
		\sum_{(m,n)\in\mathcal{N}_R(i,j)}
		K_R(x_{m,n}-x_{i,j})\,\rho_{m,n}\,\ell^2+\varepsilon}.
\end{equation*}
and consequently,
\begin{equation*}
	\sum_{i,j}\mathcal J^\ell_{i,j}\,\varphi_{i,j}\,\ell^2
	=
	\kappa_{\mathrm{conf}}
	\sum_{i,j}
	\rho_{i,j}
	V_{i,j}\cdot \nabla\varphi(x_{i,j})\,\ell^2 + o(\ell^2).
\end{equation*}

By taking the formal limit $\ell \to 0$, we arrive at
\[
\sum_{i,j}\mathcal J^\ell_{i,j}\,\varphi_{i,j}\,\ell^2 \to \int_\Omega
\kappa_{\mathrm{conf}}\rho\,V[\rho]\cdot \nabla\varphi\,\mathrm dx = \int_\Omega \mathcal J[\rho]\,\varphi\,\mathrm dx, 
\]
where the nonlocal velocity
field $V[\rho]$ is defined in Eq.~\eqref{eq:intro_V}. In the strong form, it arrives at
\begin{equation}\label{eq:rho-adv-part}
	\mathcal J[\rho]
	=
	-\kappa_{\mathrm{conf}}\nabla\cdot\bigl(\rho V[\rho]\bigr).
\end{equation}

%Under suitable regularity assumptions on $\rho$ and $K_R$, these Riemann
%sums formally converge, as $\ell\to0$, to the nonlocal velocity
%field $V[\rho]$ in Eq.~\eqref{eq:intro_V}.
%\begin{equation}\label{eq:intro_V-continuum}
%	V[\rho](x,t)
%	=
%	\frac{\displaystyle
	%		\int_{\Omega} K_R(y-x)\,(y-x)\,\rho(y,t)\,\mathrm{d}y}
%	{\displaystyle
	%		\int_{\Omega} K_R(y-x)\,\rho(y,t)\,\mathrm{d}y+\varepsilon}.
%\end{equation}

%Now we approximate the operator by retaining only its first moment, thereby capturing its drift contribution

%which, in the continuum description, gives
%\begin{equation*}
%	\int_\Omega \mathcal J[\rho]\,\varphi\,\mathrm dx
%	\approx
%	\int_\Omega
%	\kappa_{\mathrm{conf}}\rho\,V[\rho]\cdot \nabla\varphi\,\mathrm dx.
%\end{equation*}
%This corresponds, in weak form, to the drift approximation
%\begin{equation}\label{eq:rho-adv-part}
%	\mathcal J[\rho]
%	\approx
%	-\kappa_{\mathrm{conf}}\nabla\cdot\bigl(\rho V[\rho]\bigr).
%\end{equation}

\subsubsection{Attention-driven chemotactic movement}

We now derive the chemotactic movement induced by the attention-driven
biased random walk. Since these jumps are restricted to the nearest
neighbors, the corresponding contribution can be written in terms of
discrete interface fluxes. The attention-driven flux across the
interface $(i+\frac12,j)$ is
\[
J^{x,\mathrm{att}}_{i+\frac12,j}
=
\alpha_S
\left(
\frac{A_{i+1,j}}{T_{i,j}}\,n_{i,j}
-
\frac{A_{i,j}}{T_{i+1,j}}\,n_{i+1,j}
\right),
\]
with an analogous expression for
$J^{y,\mathrm{att}}_{i,j+\frac12}$.

By introducing the density variable
$\rho_{i,j}:=\frac{n_{i,j}}{\ell^2}$, the $x$-flux can be written as
\[
J^{x,\mathrm{att}}_{i+\frac12,j}
=
\alpha_S\ell^2
\left(
\frac{A_{i+1,j}}{T_{i,j}}\rho_{i,j}
-
\frac{A_{i,j}}{T_{i+1,j}}\rho_{i+1,j}
\right),
\]
and similarly in the $y$ direction.

Define the discrete attention-driven operator by
\[
\mathcal D^\ell_{i,j}
:=
\frac{1}{\ell^2}
\left(
J^{x,\mathrm{att}}_{i-\frac12,j}
-
J^{x,\mathrm{att}}_{i+\frac12,j}
+
J^{y,\mathrm{att}}_{i,j-\frac12}
-
J^{y,\mathrm{att}}_{i,j+\frac12}
\right).
\]
Thus $\mathcal D^\ell_{i,j}$ represents the attention-driven contribution
to $\partial_t\rho_{i,j}$.

To pass to the continuum limit, we use a weak formulation. Let
$\varphi\in C^\infty_{\rm per}(\Omega)$ be a smooth periodic test
function and set $\varphi_{i,j}=\varphi(x_{i,j})$. By periodicity and
discrete summation by parts, we obtain
\[
\sum_{i,j}
\mathcal D^\ell_{i,j}\varphi_{i,j}\ell^2
=
\sum_{i,j}
J^{x,\mathrm{att}}_{i+\frac12,j}
\left(\varphi_{i+1,j}-\varphi_{i,j}\right)
+
\sum_{i,j}
J^{y,\mathrm{att}}_{i,j+\frac12}
\left(\varphi_{i,j+1}-\varphi_{i,j}\right).
\]
Substituting the fluxes gives
\[
\begin{aligned}
	\sum_{i,j}
	\mathcal D^\ell_{i,j}\varphi_{i,j}\ell^2
	&=
	\frac{\alpha_S\ell^2}{4}
	\sum_{i,j}
	\frac{4}{\ell}
	\left(
	\frac{A_{i+1,j}}{T_{i,j}}\rho_{i,j}
	-
	\frac{A_{i,j}}{T_{i+1,j}}\rho_{i+1,j}
	\right)
	\frac{\varphi_{i+1,j}-\varphi_{i,j}}{\ell}
	\ell^2
	\\
	&\quad+
	\frac{\alpha_S\ell^2}{4}
	\sum_{i,j}
	\frac{4}{\ell}
	\left(
	\frac{A_{i,j+1}}{T_{i,j}}\rho_{i,j}
	-
	\frac{A_{i,j}}{T_{i,j+1}}\rho_{i,j+1}
	\right)
	\frac{\varphi_{i,j+1}-\varphi_{i,j}}{\ell}
	\ell^2 .
\end{aligned}
\]

We impose the diffusive scaling for the nearest-neighbour jump rate:
\[
D_\rho^\ell
:=
\frac{\alpha_S\ell^2}{4}
\to
D_\rho>0,
\quad \ell\to0.
\]
Assuming that $\rho_{i,j}$ and $A_{i,j}$ have smooth continuum limits
$\rho$ and $A$ with $A(x,t)\ge a^\ast>0$, the weak formulation gives
\[
\sum_{i,j}
\mathcal D^\ell_{i,j}\varphi_{i,j}\ell^2
\to
-
D_\rho
\int_\Omega
\left(
\nabla\rho
-
\frac{2\rho}{A}\nabla A
\right)
\cdot
\nabla\varphi\,dx ,
\quad \ell\to0.
\]
and
\begin{equation}\label{eq:rho-diff-part}
	-D_\rho
	\int_\Omega
	\left(
	\nabla\rho
	-
	\frac{2\rho}{A}\nabla A
	\right)
	\cdot
	\nabla\varphi\,dx
	=
	D_\rho
	\int_\Omega
	\nabla\cdot
	\left(
	\nabla\rho
	-
	\frac{2\rho}{A}\nabla A
	\right) \varphi \,dx.
\end{equation}

%Hence the limiting attention-driven operator $\mathcal D[\rho,A]$ is
%characterized weakly by
%\[
%\int_\Omega
%\mathcal D[\rho,A]\varphi\,dx
%=
%-
%D_\rho
%\int_\Omega
%\left(
%\nabla\rho
%-
%\frac{2\rho}{A}\nabla A
%\right)
%\cdot
%\nabla\varphi\,dx .
%\]
%Equivalently,
%\begin{equation}\label{eq:rho-diff-part}
%	\mathcal D[\rho,A]
%	=
%	D_\rho\nabla\cdot
%	\left(
%	\nabla\rho
%	-
%	\frac{2\rho}{A}\nabla A
%	\right).
%\end{equation}

When the effective attention is bounded away from zero, 
\[
A(x,t)=A^0(x)+S(x,t),
\quad
A(x,t)\ge a^\ast>0,
\]
we may equivalently write $\nabla A/A=\nabla\ln A$.
The first term in \eqref{eq:rho-diff-part} represents unbiased diffusion
of the opinion density, while the second describes drift toward
regions of higher attention.

Combining the transport approximation \eqref{eq:rho-adv-part} with the
attention-driven contribution \eqref{eq:rho-diff-part}, we obtain
\[
\partial_t\rho
=
-\kappa_{\mathrm{conf}}\nabla\cdot(\rho V[\rho])
+
D_\rho\nabla\cdot
\left(
\nabla\rho
-
\frac{2\rho}{A}\nabla A
\right),
\]
which is the macroscopic equation for the opinion density in
Eq.~\eqref{eq:intro_rhoS}.

%following
%macroscopic equation for the opinion density:
%\begin{equation}\label{eq:macro-rho}
%	\partial_t\rho
%	=
%	-\kappa_{\mathrm{conf}}\nabla\!\cdot\!\bigl(\rho V[\rho]\bigr)
%	+
%	D_\rho\,\nabla\!\cdot\!
%	\left(
%	\nabla\rho-\frac{2\rho}{A}\nabla A
%	\right),
%\end{equation}
%where the nonlocal velocity field $V[\rho]$
%is defined by \eqref{eq:intro_V-continuum}.

%Combining Eqs.~\eqref{eq:rho-adv-part} and \eqref{eq:rho-diff-part}, we finally obtain the full NODAR system~\eqref{eq:intro_rhoS}.

%\subsection{Macroscopic modeling of the attention field $S$}
%\label{subsec:attention-model}

\subsection{Scaling analysis of the NODAR system}
\label{subsec:notation-dim}

For analytical clarity, we perform a scaling analysis.
%Let $x \in \Omega \subset \mathbb{R}^2$ and $t \ge 0$ denote the
%space and time variables. 
%The opinion density $\rho(x,t)$ is normalized such that 
%$\int_{\Omega}\rho(x,t)\,dx = 1$, and thus has dimension $[\rho]=L^{-2}$, 
The opinion density $\rho(x,t)$ satisfies $\int_{\Omega}\rho(x,t)\,dx = 1$, and thus has scaling $[\rho]=L^{-2}$,  while the attention variables $A$ and $S$ are dimensionless.

The macroscopic parameters in \eqref{eq:intro_rhoS} satisfy
\[
[D_\rho]=[D_S]=L^2T^{-1},\quad
[\kappa_{\mathrm{conf}}]=T^{-1},\quad
[\omega]=T^{-1},\quad
[\theta]=L^2T^{-1}.
\]
From the lattice derivation, the transport coefficients obey the
diffusive scalings
$D_S = \eta\ell^2/(4\Delta t)$ and
$D_\rho = (\alpha_S\Delta t) \ell^2/(4\Delta t)$.

We introduce the characteristic time scale
$T_0 = \kappa_{\mathrm{conf}}^{-1}$  and define
\[
\tilde{x} = \frac{x}{L}, \quad
\tilde{t} = \frac{t}{T_0}, \quad
\tilde{\rho} = L^2 \rho,
\]
together with the dimensionless parameters
\[
\tilde{\omega} = \frac{\omega}{\kappa_{\mathrm{conf}}}, \quad
\tilde{\theta} = \frac{\theta}{L^2 \kappa_{\mathrm{conf}}}, \quad
\tilde{D}_\rho = \frac{D_\rho}{L^2 \kappa_{\mathrm{conf}}}, \quad
\tilde{D}_S = \frac{D_S}{L^2 \kappa_{\mathrm{conf}}},\quad \tilde{R}=R/L.
\]
Substituting these scalings into the system and rescaling the domain to
$\Omega=[0,1]^2$, we obtain a dimensionless formulation. 
Hereafter we omit tildes for brevity.
%For simplicity, we omit the tildes hereafter.

The dimensionless NODAR system reads
\begin{subequations}
	\label{eq:HK-S-scaled}
	\begin{empheq}[left=\empheqlbrace]{align}
		\partial_t \rho &= -\nabla \cdot (\rho V[\rho]) + D_\rho \nabla \cdot \left( \nabla \rho - \frac{2\rho}{A} \nabla A \right), \label{eq:HK-S-scaled-rho} \\[6pt]
		\partial_t S &= D_S \Delta S - \omega S + \theta \rho. \label{eq:HK-S-scaled-S}
	\end{empheq}
\end{subequations}

%In the numerical experiments, we set $L=1$, so that
%$D_\rho$ and $D_S$ measure diffusion relative to the alignment time scale.

If the effective attention is spatially homogeneous, $\nabla A \equiv 0$, then Eq.~\eqref{eq:HK-S-scaled-rho} in the NODAR system reduces to
\begin{equation}\label{eq:HK_nonlocal}
	\partial_t \rho
	=
	-\nabla\cdot(\rho V[\rho])
	+
	D_\rho \Delta \rho,
\end{equation}
which extends the nonlinear nonlocal Fokker-Planck equation
for the noisy Hegselmann-Krause model
\cite{ChazelleJiuLiWang2017,WangLiEChazelle2017} to multidimensional cases.

\subsection{Basic properties}
\label{subsec:properties}

Let $\Omega=\mathbb T^2$ be the two-dimensional periodic domain.
The initial data satisfy
\[
\rho_0\in H^3(\Omega),\quad
S_0\in H^4(\Omega),
\quad
\rho_0\ge0,\quad S_0\ge0
\quad\text{a.e. in }\Omega .
\]
The static baseline attention satisfies
\[
A^0\in H^4(\Omega),
\quad
A^0(x)\ge a_0>0
\quad\text{a.e. in }\Omega .
\]
The nonlocal velocity \eqref{eq:intro_V} can be written as $V[\rho](x)=\frac{\mathcal N[\rho](x)}{\mathcal K[\rho](x)}$,
%\begin{equation}\label{eq:intro_V}
%	V[\rho](x)=\frac{\mathcal N[\rho](x)}{\mathcal K[\rho](x)},
%\end{equation}
where
\[
\mathcal N[\rho](x)
:=
\int_\Omega G_R(y-x)\rho(y)\,dy,
\qquad
\mathcal K[\rho](x)
:=
\int_\Omega K_R(y-x)\rho(y)\,dy+\varepsilon.
\]
Here \(\varepsilon>0\) is fixed and
$K_R(z)=\mathbf 1_{\{|z|\le R\}}$, $G_R(z)=zK_R(z)$,
so that
\[
K_R\in L^1(\Omega),\quad
G_R\in L^1(\Omega;\mathbb R^2),\quad
K_R\ge0 .
\]

We first establish coefficient
estimates for nonlinear drift terms $V[\rho]$ and $\nabla A/A$.
\begin{lemma}\label{lem:coeff-estimates}
	Let \(\Omega=\mathbb T^2\), and \(V\) be defined by \eqref{eq:intro_V}.
	In addition, let \(A^0\in H^4(\Omega)\), and suppose that
	\[
	A_i=A^0+S_i,\quad
	A_i\ge a_0>0,\quad
	S_i\in H^4(\Omega),\quad
	\|S_i\|_{H^4}\le M,\quad i=1,2.
	\]	
	Then for any \(u,v\in H^3(\Omega)\) satisfying $\|u\|_{H^3(\Omega)}\le M, ~\|v\|_{H^3(\Omega)}\le M,~ u \ge 0,~ v \ge 0$, there exists a constant \({C}>0\) such that
	\[
	\|V[u]\|_{H^3(\Omega;\mathbb R^2)}\le C, \quad
	\|V[u]-V[v]\|_{H^3(\Omega;\mathbb R^2)} 
	\le C\|u-v\|_{H^3(\Omega)},
	\]
	and
	\[
	\left\|\frac{\nabla A_i}{A_i}\right\|_{H^3}
	\le C,
	\quad
	\left\|
	\frac{\nabla A_1}{A_1}
	-
	\frac{\nabla A_2}{A_2}
	\right\|_{H^3}
	\le C\|S_1-S_2\|_{H^4}.
	\]
	Here \(C\) depends only on
	$M,\ a_0,\ \varepsilon, \ \|A^0\|_{H^4},\
	\|K_R\|_{L^1(\Omega)},\
	\|G_R\|_{L^1(\Omega;\mathbb R^2)},\
	\Omega $.
\end{lemma}

These estimates provide the Lipschitz conditions needed to
derive the following basic structural properties. The proof is given in \ref{app:proof-positivity}.
\begin{theorem}[Basic structural properties]
	\label{thm:basic-properties}
	Under the setting in Lemma \ref{lem:coeff-estimates}, there exists $T>0$ and a
	solution $(\rho,S)$ of \eqref{eq:HK-S-scaled} on
	$\Omega\times[0,T]$ such that
	\[
	\rho\in C([0,T];H^3(\Omega))\cap L^2([0,T];H^4(\Omega)),
	\quad
	S\in C([0,T];H^4(\Omega))\cap L^2([0,T];H^5(\Omega)).
	\]
	Moreover, the following properties hold.
	
	\begin{itemize}
		
		\item[(i)] \textbf{Mass conservation of $\rho$.}
		The total mass of $\rho$ is conserved:
		\[
		\int_\Omega \rho(x,t)\,dx
		=
		\int_\Omega \rho_0(x)\,dx,
		\quad 0\le t\le T.
		\]
		
		\item[(ii)] \textbf{non-negativity and non-degeneracy.}
		The solution preserves non-negativity:
		\[
		\rho(x,t)\ge0,\quad S(x,t)\ge0
		\quad
		\text{a.e. in } \Omega\times(0,T).
		\]
		Consequently, the effective attention field
		\[
		A(x,t):=A^0(x)+S(x,t)
		\]
		satisfies
		\[
		A(x,t)\ge a_0>0
		\quad
		\text{a.e. in } \Omega\times(0,T).
		\]
	\end{itemize}
\end{theorem}

\section{Phase transition of spatial patterns and mathematical analysis}
\label{sec:phase_stability}

Assume that the baseline attention is
spatially uniform, $A^0(x)\equiv A^0>0$.
Under these assumptions, the system admits the spatially homogeneous equilibrium
\begin{equation}\label{eq:equilibrium}
	\rho(x,t)\equiv \rho_0>0,\quad
	S(x,t)\equiv S_0>0,\quad
	A(x,t)\equiv \bar A:=A^0+S_0,
\end{equation}
where
\begin{equation}
	S_0=\frac{\theta}{\omega} \rho_0, \quad \bar A=A^0+\frac{\theta}{\omega} \rho_0.
\end{equation}
One is interested in the phase patterns of NODAR after adding some perturbations on such  equilibrium state.

\begin{figure}[htbp]
	\centering
	\includegraphics[width=1\textwidth]{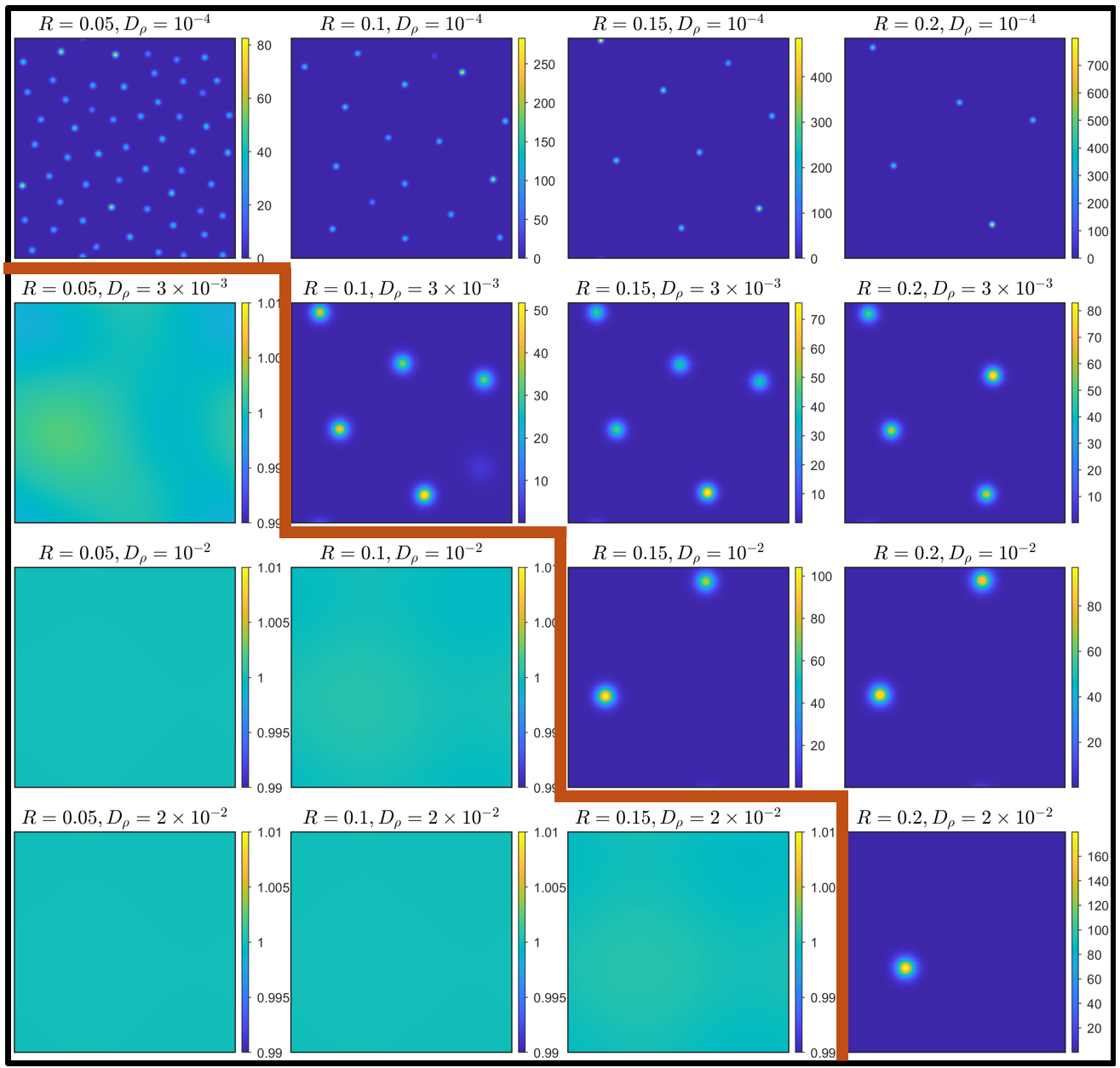}
	\caption{
		Spatial patterns of $\rho(x,t)$ at $T=100$ for different values of
		$(R,D_\rho)$, obtained from small perturbations of the homogeneous
		equilibrium. Increasing $R$ or decreasing $D_\rho$ drives a transition
		from homogeneous states to clustered patterns.
	}
	\label{fig:phase-diagram}
\end{figure}

We first provide numerical evidence of a phase transition between 
dispersed and clustered regimes. 
Fig.~\ref{fig:phase-diagram} illustrates this transition in the 
parameter plane $(R,D_\rho)$. 
For fast diffusion or small interaction radius, spatial perturbations 
decay and the solution relaxes toward a homogeneous equilibrium, 
corresponding to a consensus state. 
In contrast, when the interaction radius is sufficiently large or the 
diffusion is weak, persistent spatial clusters emerge, indicating 
non-consensus on localized opinion groups. 
The boundary separating these regimes suggests 
the existence of a critical threshold.

\subsection{Main results}

To explain the phase transition observed in the numerical simulations, 
we analyze the linear stability of the spatially homogeneous equilibrium.
This provides a theoretical characterization of the conditions under 
which aggregation emerges from small perturbations, thereby 
distinguishing consensus and non-consensus regimes.

\subsubsection{Linear stability of the homogeneous equilibrium}

%To investigate the stability of this equilibrium, 
we introduce small perturbations on the equilibrium,
\[
\rho=\rho_0+\eta r,\quad
S=S_0+\eta s,\quad 0<\eta\ll 1,
\]
and retain terms of order $\mathcal O(\eta)$.

%By symmetry of the interaction kernel, the nonlocal velocity 
%vanishes at the homogeneous state, i.e. $V[\rho_0]\equiv 0$. 
%Moreover, since the numerator of $V[\rho]$ vanishes at 
%$\rho=\rho_0$, variations of the denominator do not contribute 
%at order $\mathcal O(\eta)$. Consequently, the transport term can be linearized as
%\[
%-\nabla\!\cdot(\rho V[\rho])\approx -\rho_0\nabla\!\cdot V_1[r],
%\]
%while the attention-driven cross-diffusion introduces a coupling
%between $r$ and $s$ through the effective attention level $\bar A$.

By symmetry of the interaction kernel, the nonlocal velocity vanishes
at the homogeneous state, i.e. \(V[\rho_0]\equiv 0\). Moreover, since
the numerator of \(V[\rho]\) vanishes at \(\rho=\rho_0\), variations of
the denominator do not contribute at order \(\mathcal O(\eta)\). More precisely,
$V[\rho_0+\eta r]
=
\eta V_1[r]+o(\eta)$,
where
\[
V_1[r](x)
:=
\frac{1}{Z_R}
\int_\Omega K_R(y-x)(y-x)r(y)\,dy,
\qquad
Z_R=\rho_0\int_\Omega K_R(z)\,dz+\varepsilon .
\]
Consequently, $-\nabla\!\cdot(\rho V[\rho])
=
-\eta\rho_0\nabla\!\cdot V_1[r]+o(\eta).$

Now we seek normal modes
\[
r(x,t)=\widehat r\,e^{\lambda t+i k\cdot x},\quad
s(x,t)=\widehat s\,e^{\lambda t+i k\cdot x},\quad
k\in 2\pi\mathbb{Z}^2.
\]
By omitting $o(\eta)$ terms, it arrives at the $2\times2$ eigenvalue problem
\begin{equation}\label{eq:accd-eigsyst-nd}
	\left\{
	\begin{aligned}
		\lambda \widehat r
		&=
		\Big(\frac{\rho_0}{Z_R}\,m(k)-D_\rho |k|^2\Big)\widehat r
		+\Big(2D_\rho\,\frac{\rho_0}{\bar A}\,|k|^2\Big)\widehat s,
		\\[1mm]
		\lambda \widehat s
		&=
		-\big(D_S |k|^2+\omega\big)\widehat s
		+\theta\,\widehat r,
	\end{aligned}
	\right.
\end{equation}
where
\[
m(k)=-k\cdot\nabla_k\widehat K_R(k),\quad \widehat K_R(k) = \int_{\mathbb{R}^2}K_R(z) e^{-i k z}\,dz.
\]

Denote by
\begin{equation}\label{def.abc}
	a(k)=\frac{\rho_0}{Z_R}m(k)-D_\rho|k|^2, \quad  
	b(k)=D_S|k|^2+\omega, \quad 
	c(k)=2D_\rho\,\frac{\rho_0}{\bar A}|k|^2.
\end{equation}
From the determinant of Eq.~\eqref{eq:accd-eigsyst-nd}, we obtain the two-branch dispersion relation
\begin{equation}\label{eq:dispersion-accd-nd}
	\lambda_\pm(k)
	=
	\frac12\Big(
	a(k)-b(k)
	\pm
	\sqrt{(a(k)+b(k))^2+4c(k)\theta}
	\Big).
\end{equation}

Linear stability of the homogeneous equilibrium is determined by the
dominant branch $\lambda_{\pm}(k)$:
if $\lambda_\pm(k)<0$ for all $k\neq0$, then $(\rho_0,S_0)$ (equivalently, $(\rho_0,\bar A)$)
is linearly stable; whereas the existence of $k^\ast\neq0$ with $\lambda_{\pm}(k^\ast)>0$
signals a finite-wavelength instability and the onset of aggregation.

%The dispersion relation $\lambda_\pm(k)$ provides a complete
%characterization of the linear stability of the homogeneous state.
In particular, for the disk kernel $K_R(z)=\mathbf{1}_{\{|z|\le R\}}$, using
rotational symmetry, it yields
\begin{equation}\label{eq:mk}
	m(k)
	=
	2\pi R^2\Big(\frac{2J_1(|k|R)}{|k|R} - 
	J_0(|k|R)
	\Big),
	\quad 
	Z_R=\rho_0 \pi R^2+\varepsilon.
\end{equation}
where $J_0$ and $J_1$ are Bessel functions of the first kind.

The following lemma characterizes the asymptotic behaviors
in the long-wave and high-frequency regimes, indicating that 
any unstable modes (if they exist) are confined to
a bounded range of wavenumbers.

\begin{lemma}[Long-wave behavior and high-frequency damping]
	\label{lem:accd-small-large-k}
	Suppose $K_R(z)=\mathbf{1}_{\{|z|\le R\}}$ and let $\lambda_\pm(k)$ be given by \eqref{eq:dispersion-accd-nd}. Then the following results are obtained.
	
	\begin{enumerate}
		\item[(i)] For zeroth mode $k=0$, 
		\[
		\lambda_+(0)=0,
		\quad
		\lambda_-(0)=-\omega,
		\]
		where the zero eigenvalue corresponds to mass conservation.
		
		\item[(ii)] As $|k|\to0$,
		\[
		\lambda_-(k)=-\omega+\mathcal O(|k|^2),
		\quad
		\lambda_+(k)=\mu\,|k|^2+\mathcal O(|k|^4),
		\]
		where for the disk kernel $K_R=\mathbf{1}_{\{|z|\le R\}}$,
		\begin{equation}\label{eq:mu-accd}
			\mu
			=
			-D_\rho
			+\rho_0 \left(\frac{\pi R^4}{4Z_R}
			+\frac{2D_\rho\theta}{\bar{A}\omega}\right).
		\end{equation}
		The sign of $\mu$ determines whether long-wave
		perturbations are damped $(\mu<0)$ or amplified $(\mu>0)$.
		
		\item[(iii)] As $|k|\to\infty$,
		\[
		\lambda_+(k)\sim -\min\{D_\rho,D_S\}|k|^2,
		\quad
		\lambda_-(k)\sim -\max\{D_\rho,D_S\}|k|^2,
		\]
		so that high-frequency modes are damped.
	\end{enumerate}
\end{lemma}

\begin{proof}
	At $k=0$, we have $m(0)=0$ and $c(0)=0$, so that $a(0)=0$, $b(0)=\omega$ and
	\[
	\lambda_+(0)=0,
	\quad
	\lambda_-(0)=-\omega.
	\]
	%where the zero eigenvalue corresponds to the constant mode associated with mass conservation.
	
	As \(|k|\to0\), we have
	$a(k)=\mathcal O(|k|^2)$, $b(k)=\omega+\mathcal O(|k|^2)$, $c(k)=\mathcal O(|k|^2)$.
	In fact, for the disk kernel $K_R=\mathbf{1}_{\{|z|\le R\}}$,
	\[
	m(k)
	=
	2\pi R^2
	\left(\frac{2J_1(q)}{q} - 
	J_0(q)
	\right),
	\quad
	q=|k|R.
	\]
	Using the asymptotic expansion of the Bessel functions near
	$q=0$,
	\[
	\frac{2J_1(q)}{q}-
	J_0(q)
	=
	\frac{q^2}{8}+\mathcal O(q^4),
	\]
	we obtain
	\[
	m(k)
	=
	\frac{\pi R^4}{4}|k|^2
	+
	\mathcal O(|k|^4),
	\quad |k|\to0.
	\]
	Therefore,
	\begin{equation}\label{eq.1}
		a(k)
		=
		-D_\rho|k|^2
		+
		\frac{\rho_0}{Z_R}m(k)
		=
		\left(
		-D_\rho
		+
		\frac{\rho_0}{Z_R}\frac{\pi R^4}{4}
		\right)|k|^2
		+
		\mathcal O(|k|^4).
	\end{equation}
	In addition, using $(\omega + D_S|k|^2)^{-1} \sim \omega^{-1}(1 - \frac{D_S}{\omega}|k|^2 + \mathcal{O}(|k|^4))$, it yields 
	\begin{equation}\label{eq.2}
		\frac{c(k)\theta}{b(k)}
		=
		\frac{2D_\rho\rho_0\theta}{\bar A\omega}|k|^2
		+
		\mathcal O(|k|^4).
	\end{equation}
	
	Since
	\[
	(a(k)+b(k))^2+4c(k)\theta=	  \omega^2+\mathcal O(|k|^2)
	\]				
	%	    Since
	%	     $(a(k)+b(k))^2+4c(k)\theta=
	%	    \omega^2+\mathcal O(|k|^2)$.
	it further has
	$ \sqrt{(a(k)+b(k))^2+4c(k)\theta}
	=
	\omega+\mathcal O(|k|^2)$.
	Substituting it into the dispersion relation \eqref{eq:dispersion-accd-nd} gives
	$\lambda_-(k)
	=
	-\omega+\mathcal O(|k|^2)$.
	
	It remains to expand the branch \(\lambda_+(k)\). Write
	\[
	(a(k)+b(k))^2+4c(k)\theta
	=
	b^2(k)\left(
	1+
	\frac{2a(k)b(k)+a^2(k)+4c(k)\theta}{b^2(k)}
	\right).
	\]
	Using the Taylor expansion,
	we obtain
	\[
	\sqrt{(a(k)+b^2(k))+4c(k)\theta}
	=
	b(k)+
	\frac{2a(k)b(k)+a^2(k)+4c(k)\theta}{2b(k)}
	+
	\mathcal O(|k|^4).
	\]
	Substituting this expansion into \(\lambda_+(k)\) gives
	\[
	\lambda_+(k)
	=
	a(k)+\frac{c(k)\theta}{b(k)}
	+\frac{a^2(k)}{4b(k)}
	+\mathcal O(|k|^4).
	\]
	
	%	 Moreover,
	%		\[
	%		\frac{2a(k)b(k)+a^2(k)+4c(k)\theta}{b^2(k)}
	%		=
	%		\mathcal O(|k|^2).
	%		\]		
	Since \(a^2(k)/(4b(k))=\mathcal O(|k|^4)\), this term can be absorbed
	into the remainder, therefore
	\[
	\lambda_+(k)
	=
	a(k)+\frac{c(k)\theta}{b(k)}
	+
	\mathcal O(|k|^4).
	\]

	Combining with Eqs.~\eqref{eq.1} and \eqref{eq.2}, it arrives at
	\[
	\lambda_+(k)
	=
	\left(
	-D_\rho
	+
	\frac{\rho_0}{Z_R}\frac{\pi R^4}{4}
	+
	\frac{2D_\rho\rho_0\theta}{\bar A\omega}
	\right)|k|^2
	+
	\mathcal O(|k|^4),
	\]
	which implies $\lambda_+(k)=\mu |k|^2+\mathcal O(|k|^4)$ with $\mu$ given by \eqref{eq:mu-accd}.

	Finally, we consider the high-frequency limit $|k|\to\infty$. For the
	disk kernel, using the Bessel asymptotic formula
	\[
	J_\nu(q)
	=
	\sqrt{\frac{2}{\pi q}}
	\cos\left(q-\frac{\nu\pi}{2}-\frac{\pi}{4}\right)
	+
	\mathcal O(q^{-3/2}),
	\quad q\to\infty,
	\]
	we obtain
	$J_0(q)=\mathcal O(q^{-1/2})$, and $J_1(q)/q=\mathcal O(q^{-3/2})$.
	Hence
	\[
	m(k)
	=
	2\pi R^2
	\left(\frac{2J_1(q)}{q} - 
	J_0(q)
	\right)
	=
	\mathcal O(q^{-1/2})
	=
	\mathcal O(|k|^{-1/2}).
	\]
	Thus
	\[
	a(k)
	=
	-D_\rho|k|^2
	+
	\frac{\rho_0}{Z_R}m(k)
	=
	-D_\rho|k|^2
	+
	o(|k|^2),
	\quad
	b(k)=D_S|k|^2+\omega.
	%=
	%D_S|k|^2+\mathcal O(1).
	\]
	Substituting it into \eqref{eq:dispersion-accd-nd} yields
	\[
	\lambda_+(k)\sim -\min\{D_\rho,D_S\}|k|^2,
	\quad
	\lambda_-(k)\sim -\max\{D_\rho,D_S\}|k|^2,
	\]
	which shows that all high-frequency modes are damped.
	
\end{proof}

\begin{figure}[h!]
	\centering
	\includegraphics[width=0.5\textwidth]{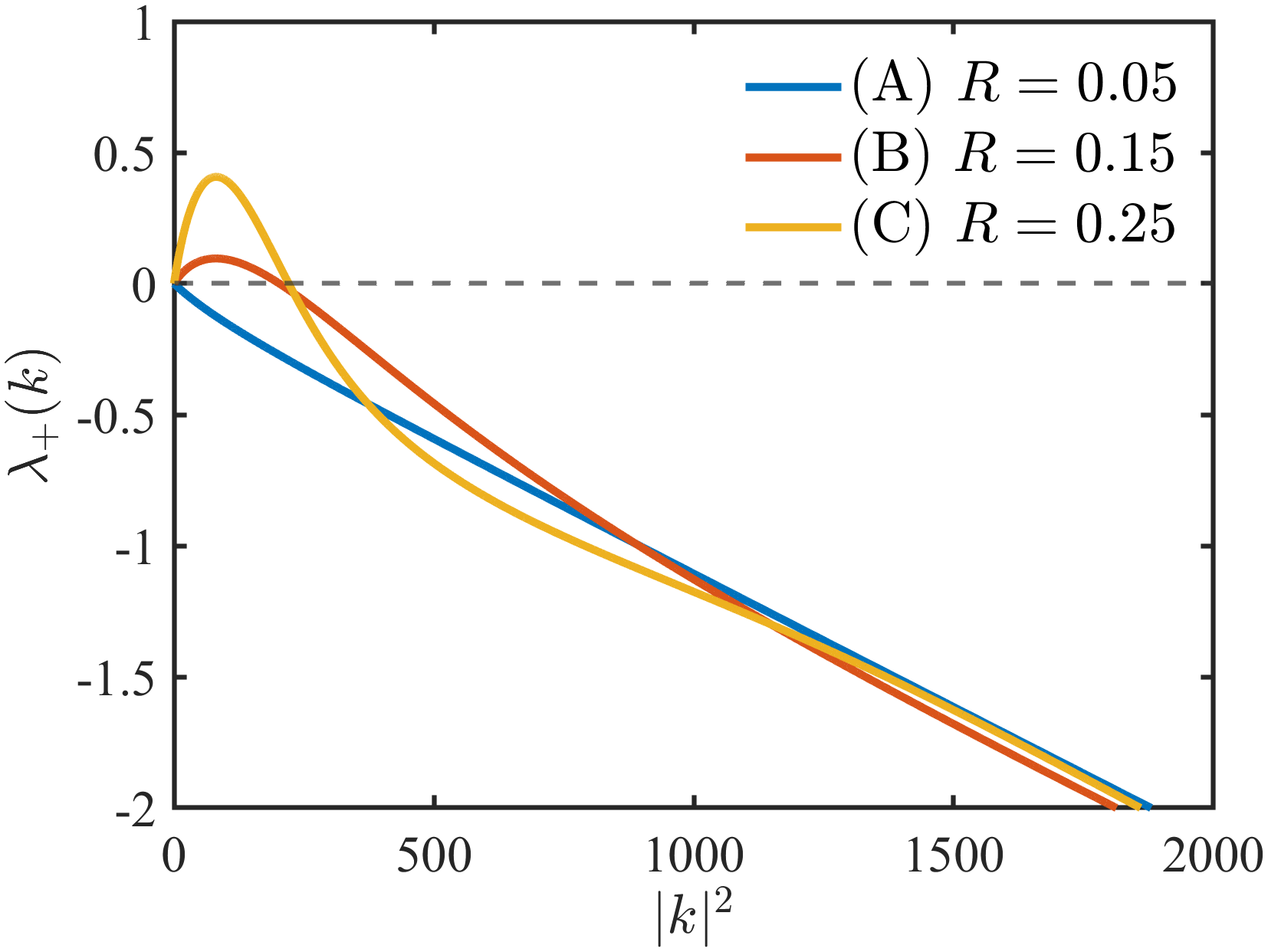}
	\vspace{-5pt}
	\caption{
		Dispersion curves $\lambda_+(k)$ for three representative values of the
		interaction radius $R$.
		A finite band of unstable wavenumbers appears as $R$ increases. Parameters are $A^0=1/30$, $D_\rho =0.01$, $D_S=10^{-3}$, $\omega=1$, and $\theta=0.02$.
	}
	\label{fig:dispersion_example}
	\vspace{-5pt}
\end{figure}

To illustrate the dispersion relation predicted by
Lemma~\ref{lem:accd-small-large-k}, we display in
Fig.~\ref{fig:dispersion_example} the dispersion curves
$\lambda_+(k)$ for several representative values of the interaction
radius $R$.
When $R$ is sufficiently small (curve (A)), the dispersion relation
remains negative for all wavenumbers, indicating that the homogeneous
state is linearly stable.
As $R$ increases (curve (B)), a band of intermediate wavenumbers
appears for which $\lambda_+(k)>0$, signalling the onset of a
finite-band instability.
For larger values of $R$ (curve (C)), the unstable band becomes more
pronounced and the maximal growth rate increases.

These observations are consistent with the asymptotic behaviour
established in Lemma~\ref{lem:accd-small-large-k}.
In particular, long-wave modes are governed by the coefficient $\mu$,
while high-frequency modes are stabilized by diffusion.
Consequently, instability occurs within a bounded interval of
intermediate wavenumbers.
Moreover, Fig.~\ref{fig:dispersion_example} shows that both the
existence and the location of unstable modes depend sensitively on the
interaction radius $R$.
This dependence suggests that the onset of instability is determined by
a threshold relation between the interaction radius $R$ and the
diffusion coefficient $D_\rho$.

Now we fix $D_S$ and $R$ and show that the instability  threshold for $D_{\rho}$ is closely related to a quantity for $k \ne 0$,
\begin{equation}\label{eq:Dk-accd}
	D_{\rho,\mathrm{crit}}(k,R,D_S)
	:=
	\underbrace{\frac{\rho_0}{Z_R}\frac{m(k)}{|k|^2}}_{\textup{opinion alignment}} \times
	\underbrace{\frac{1}{\Gamma(k,D_S)}}_{\textup{environmental feedback}}
\end{equation}
with $Z_R=\rho_0\pi R^2+\varepsilon$. The first term corresponds to the critical point in the opinion alignment (see Eq.~\eqref{eq:Dk-base} below). The second term is given by
\begin{equation}\label{eq:gamma-accd}
	\Gamma(k,D_S)
	:=
	1-\frac{2\rho_0\theta}{(A^0 + \frac{\theta}{\omega} \rho_0)\big(D_S|k|^2+\omega\big)} ,
\end{equation}
which presents an explicit correction by the environmental feedback. 

We can divide the parameter regime into two parts. 
\begin{equation}\label{eq:gamma-negative}
	\textup{Case I}:\quad A^0 < \frac{\rho_0 \theta}{\omega} ~\Rightarrow~
	\Gamma(k,D_S) < 0 ~~\textup{for}~~ |k|^2 < \frac{1}{D_S} \left(\frac{2\rho_0 \theta}{A^0 + \frac{\rho_0 \theta}{\omega}} - \omega\right),
\end{equation}
and
\begin{equation}\label{eq:gamma-positive}
	\textup{Case II}:\quad 	A^0 > \frac{\rho_0 \theta}{\omega} ~\Rightarrow~
	0 < \Gamma(k,D_S) < 1 ~~\textup{for all}~k.
\end{equation}

\begin{theorem}[Instability criterion for fixed \(D_S\)]
	\label{thm:accd-DR-threshold}
	Define a global threshold for fixed \(D_S\),
	\begin{equation}\label{eq:Dcrit-global}
		D_{\mathrm{crit}}^{\mathrm{NODAR}}(R,D_S)
		:=
		\sup_{\substack{k\neq0\\ m(k)>0}}
		D_{\rho,\mathrm{crit}}(k,R,D_S).
	\end{equation}
	%with the convention that
	%\(D_{\mathrm{crit}}^{\mathrm{NODAR}}(R,D_S)=0\) if
	%\(\{k\neq0:m(k)>0\}=\emptyset\).
	Then the following statements hold.
	\begin{enumerate}
		\item[(1)]
		With assumption \eqref{eq:gamma-negative}, the homogeneous
		equilibrium is linearly unstable.
		
		\item[(2)] With assumption \eqref{eq:gamma-positive}, if
		\(D_\rho>D_{\mathrm{crit}}^{\mathrm{NODAR}}(R,D_S)\), then the homogeneous
		equilibrium is linearly stable.
		
		\item[(3)] With assumption \eqref{eq:gamma-positive}, if
		\(0<D_\rho<D_{\mathrm{crit}}^{\mathrm{NODAR}}(R,D_S)\), then the homogeneous
		equilibrium is linearly unstable. Moreover, the unstable modes are contained
		in a bounded set of wavenumbers, and \(\lambda_+\) attains its maximum at
		some
		\[
		k^\ast\in\arg\max_{k\neq0}\lambda_+(k).
		\]
		
	\end{enumerate}
\end{theorem}

\begin{proof}
	Fix \(k\neq0\). Since \(b(k)=D_S|k|^2+\omega>0\), we first have
	\(\lambda_-(k)<0\). Indeed,
	\[
	\sqrt{(a(k)+b(k))^2+4c(k)\theta}\ge |a(k)+b(k)|,
	\]
	and therefore
	\[
	\lambda_-(k)
	=
	\frac12\Big(a(k)-b(k)
	-\sqrt{(a(k)+b(k))^2+4c(k)\theta}\Big)<0 .
	\]
	Hence linear instability can only occur through the branch \(\lambda_+(k)\).
	
	It is observed that \(\lambda_+(k)\lambda_-(k)=-F_k(D_\rho)\), 
	\begin{equation*}
		\begin{aligned}
			F_k(D_\rho):=a(k)b(k)+c(k)\theta 
			&=
			\left(\frac{\rho_0}{Z_R}m(k)-D_\rho |k|^2\right)b(k)
			+
			2D_\rho\frac{\rho_0\theta}{\bar A}|k|^2   \\
			&=
			\left(
			\frac{\rho_0}{Z_R} \frac{m(k)}{|k|^2} \frac{1}{\Gamma(k,D_S)}
			-
			D_\rho
			\right)\Gamma(k,D_S)|k|^2 b(k).
		\end{aligned}
	\end{equation*}
	As \(\lambda_-(k)<0\), we have
	\[
	\lambda_+(k)>0
	\quad\Longleftrightarrow\quad
	F_k(D_\rho)>0 .
	\]

	For Case I, $\Gamma(k,D_S) < 0$ for small $|k|$. Since 
	\[
	m(k)
	=
	\frac{\pi R^4}{4}|k|^2
	+
	\mathcal O(|k|^4),
	\]	
	$m(k)$ is positive for small $|k|$ and consequently 
	\[
	\frac{\rho_0}{Z_R}m(k)-D_\rho\Gamma(k,D_S)|k|^2 > 0	\Rightarrow F_k(D_\rho) > 0,
	\]
	implying the homogeneous equilibrium is linearly unstable.

	For Case II, $0 < \Gamma(k,D_S) < 1$. We now distinguish two cases.
	
	\begin{enumerate}
		\item[(i)] For $k \in \{m(k)\le0\}$, 
		\(F_k(D_\rho)<0\), and hence
		\(\lambda_+(k)<0\). Thus this mode is stable.
		
		\item[(ii)] For $k \in \{m(k)>0\}$,  \(F_k(D_\rho)\) is strictly decreasing in \(D_\rho\). Moreover, \(F_k(0)>0\), and the unique
		positive root of \(F_k(D_\rho)=0\) is exactly
		\(D_{\rho,\mathrm{crit}}(k,R,D_S)\) defined in \eqref{eq:Dk-accd}. Therefore,
		for this mode,
		\[
		\lambda_+(k)>0
		\quad\Longleftrightarrow\quad
		0<D_\rho<D_{\rho,\mathrm{crit}}(k,R,D_S).
		\]
	\end{enumerate}
	
	Now assume \(D_\rho>D_{\mathrm{crit}}^{\mathrm{NODAR}}(R,D_S)\). If
	\(m(k)\le0\), the preceding argument gives \(\lambda_+(k)<0\). If
	\(m(k)>0\), then by the definition \eqref{eq:Dcrit-global},
	$D_\rho>D_{\rho,\mathrm{crit}}(k,R,D_S)$,
	and hence \(\lambda_+(k)<0\). Thus \(\lambda_+(k)<0\) for every
	\(k\neq0\), which proves linear stability.
	
	Conversely, if
	\(0<D_\rho<D_{\mathrm{crit}}^{\mathrm{NODAR}}(R,D_S)\), then by
	\eqref{eq:Dcrit-global} there exists a mode \(k\neq0\) with \(m(k)>0\)
	such that
	$D_\rho<D_{\rho,\mathrm{crit}}(k,R,D_S)$.
	For this mode, \(\lambda_+(k)>0\), and the homogeneous equilibrium is
	linearly unstable.
	
	Finally, by Lemma~\ref{lem:accd-small-large-k}, high-frequency modes are
	damped. Hence the unstable modes are contained in a bounded set of
	wavenumbers. Since the admissible wavevectors form a discrete set, the
	maximum of \(\lambda_+(k)\) is attained at some
	\(k^\ast\in\arg\max_{k\neq0}\lambda_+(k)\).
\end{proof}

%The diffusion-threshold criterion below is formulated in the regime
%\begin{equation}\label{eq:gamma-positive}
%	\Gamma(k,D_S)>0,\qquad k\neq0, 
%\end{equation}
%Equivalently,
%\[
%D_S|k|^2+\omega>
%\frac{2\rho_0\theta}{\bar A},
%\qquad k\neq0,
%\]
%so that \(D_\rho\) remains a stabilizing diffusion parameter in the
%linearized mode-wise criterion. In this regime, for each mode \(k\neq0\)
%with \(m(k)>0\), define

Apart from the instability threshold, the dispersion relation
also selects a dominant wavenumber associated with the fastest-growing
mode, which provides a characteristic spatial scale for the emerging
patterns in the linear regime.

\begin{remark}[Characteristic wavelength]\label{remark:L}
	Whenever the homogeneous equilibrium is unstable, the dominant branch of the dispersion relation admits a maximizer 
	$k^\ast \in \arg\max_{k\neq 0}\lambda_+(k)$
	(see Theorem~\ref{thm:accd-DR-threshold}). 
	This wavenumber determines a characteristic wavelength
	\begin{equation}
		\label{eq:lstar_def}
		\ell^\ast(R)=\frac{2\pi}{|k^\ast|},
	\end{equation}
	which represents the spatial scale of the fastest-growing mode selected by the linear instability.
\end{remark}

The analytical characterization provided by Theorem~\ref{thm:accd-DR-threshold} 
further enables a phase-diagram representation in the $(R,D_\rho)$-plane.
For fixed model parameters, we compute the dominant growth rate
\[
\lambda_+^{\max}(R,D_\rho)
:= \max_{k\neq 0} \lambda_+(k,R,D_\rho),
\]
and evaluate it over the parameter plane.
The resulting phase diagram, shown in
Fig.~\ref{fig:lambda_phase_diagrams}, visualizes the stability
structure predicted by the dispersion relation.
In particular, the global diffusion threshold
$D_{\mathrm{crit}}^{\mathrm{NODAR}}(R,D_S)$ separates the linearly stable
and unstable regimes. This boundary marks a transition between a spatially homogeneous distribution, 
which can be interpreted as a consensus state, and regimes where spatial clustering 
and non-consensus emerge.

\begin{figure}[h!]
	\centering
	\vspace{-7pt}
	\begin{subfigure}{0.5\textwidth}
		\centering
		\includegraphics[width=\linewidth]{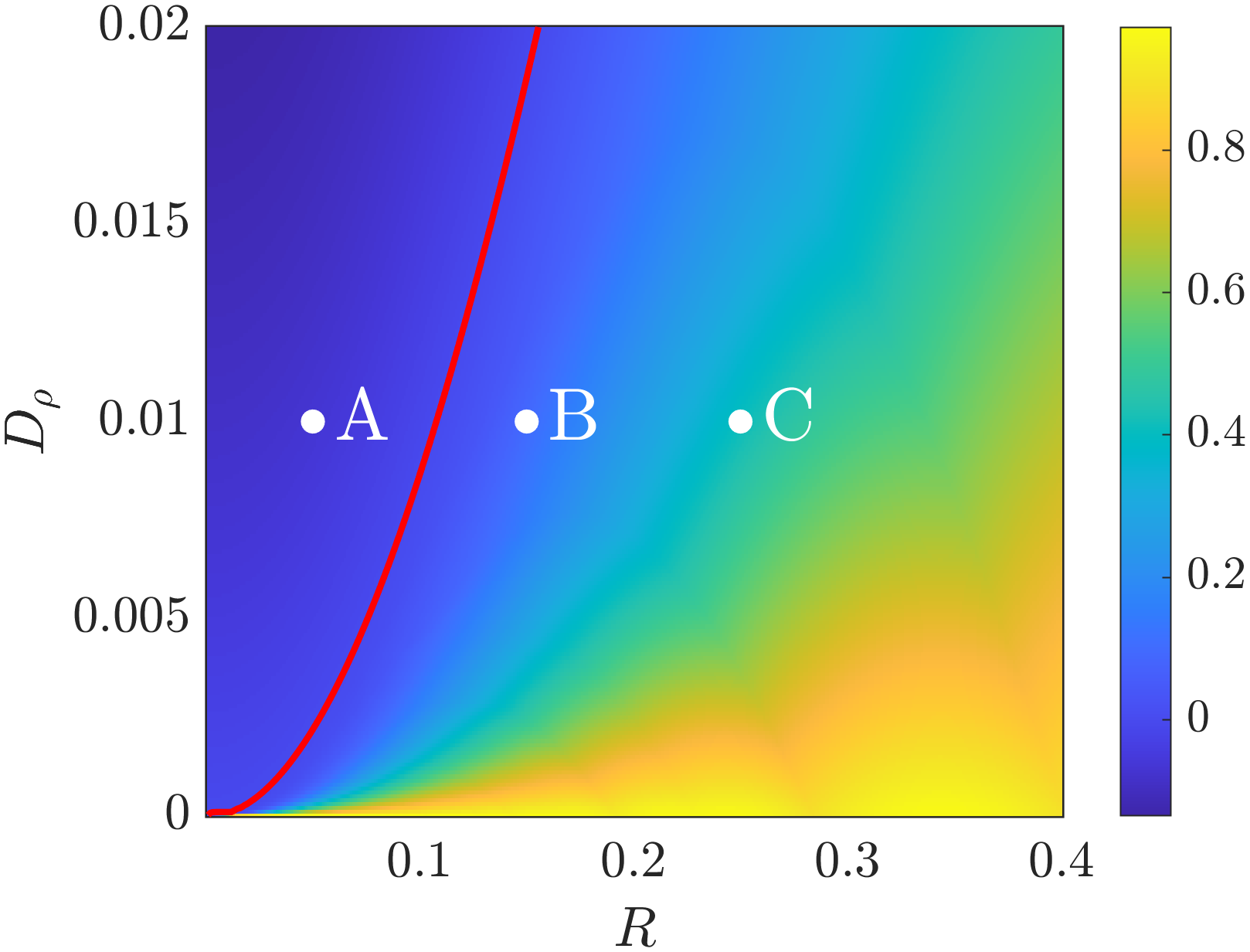}
		\label{fig:lambda_I0}
	\end{subfigure}
	\vspace{-15pt}
	\caption{
		Phase diagram of $\lambda_+^{\max}$ in the $(R,D_\rho)$-plane. The red curve shows the threshold $D_{\mathrm{crit}}^{\mathrm{NODAR}}(R,D_S)$ given by Theorem~\ref{thm:accd-DR-threshold}. Other parameters follow Fig.~\ref{fig:dispersion_example}.
	}
	\label{fig:lambda_phase_diagrams}
	\vspace{-5pt}
\end{figure}

\subsubsection{Effect of the attention diffusion coefficient $D_S$}
%The above analysis characterizes the instability threshold and the
%associated pattern scale for a fixed value of the attention diffusion
%coefficient $D_S$, as reflected in the phase diagram in the
%$(R,D_\rho)$-plane.
%We now investigate how this is modified when $D_S$ varies,
%and in particular how $D_S$ influences the instability
%threshold.
We now investigate how $D_S$ influences the instability threshold in Theorem~\ref{thm:accd-DR-threshold}, depending on the attention diffusion coefficient $D_S$ through the factor
$\Gamma(k,D_S)$ defined in \eqref{eq:gamma-accd}.
%\[
%\Gamma(k,D_S)
%=
%1-\frac{2\rho_0\theta}{\bar A(D_S|k|^2+\omega)}.
%\]
For each fixed mode \(k\neq 0\) with \(m(k)>0\), and within the regime
\(\Gamma(k,D_S)>0\), the quantity $\Gamma(k,D_S)$ is
increasing in $D_S$, and hence the threshold
$D_{\rho,\mathrm{crit}}(k,R,D_S)$ decreases for larger $D_S$, shrinking the unstable region.
%This implies that the effective instability threshold
%$D_{\mathrm{crit}}^{\mathrm{NODAR}}(R,D_S)$ decreases for larger $D_S$,
%implying that the unstable region becomes smaller.

This coincides with our intuition:
Increasing $D_S$ enhances diffusion in the attention equation,
which smooths spatial variations of the attention field and weakens
the density-attention coupling, thereby suppressing localized
aggregation. In contrast, for small $D_S$, attention remains more
localized, leading to stronger coupling and a larger unstable region.

%In the large-$D_S$ regime, $D_S|k|^2+\omega \sim D_S|k|^2$ as
%$D_S\to\infty$, so that the dependence of $\Gamma(k,D_S)$ on $D_S$
%becomes weak and the critical threshold approaches an asymptotic
%regime.

\begin{figure}[h!]
	\centering
	\includegraphics[width=0.6\textwidth]{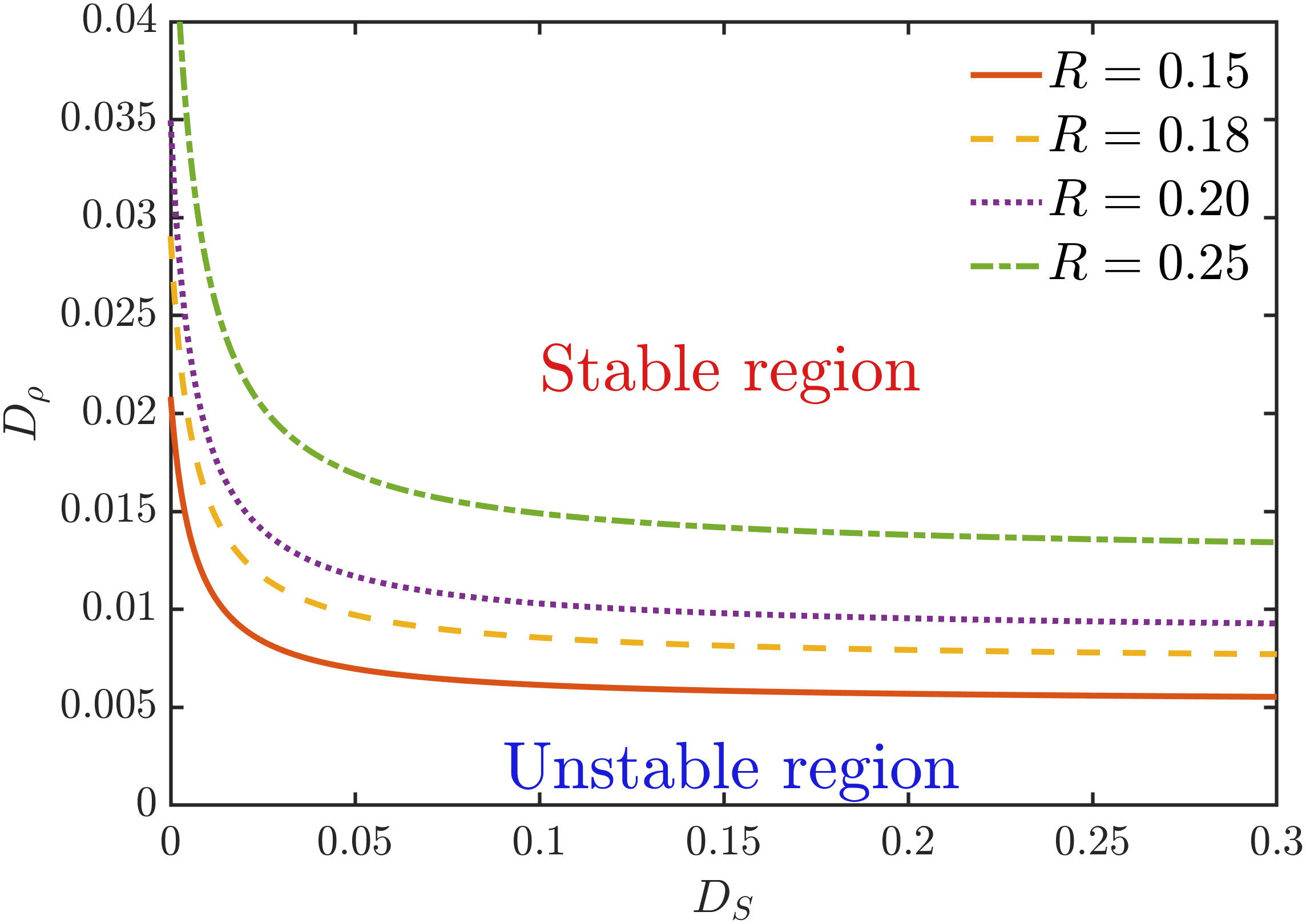}
	\caption{
		Critical curves $\lambda_+^{\max}(D_S,D_\rho;R)=0$ in the
		$(D_S,D_\rho)$-plane for several values of $R$.
		The region above each curve is stable, while the region below is
		unstable.
	}
	\label{fig:lambda_phase_DS}
\end{figure}

To further illustrate these effects, Fig.~\ref{fig:lambda_phase_DS} shows the
critical curves in the $(D_S,D_\rho)$-plane for several representative
values of the interaction radius $R$. For each fixed $R$, the region
above or below the corresponding curve is linearly stable or unstable, respectively.
%corresponding to a
%spatially homogeneous distribution, which can be interpreted as a
%consensus state. 
%In contrast, the region below the curve is unstable,
%where spatial clustering and non-consensus emerge.
The figure shows that, for all displayed values of $R$,
the critical value of $D_\rho$ lifts up as $D_S$ increases
from a small level.
%, and then gradually levels off as $D_S$ becomes
%larger.
This confirms that stronger attention diffusion has a
stabilizing effect on the system.

Fig.~\ref{fig:lambda_phase_DS} also reveals the role of the interaction radius $R$.
For larger $R$, the critical curves lie higher in the
$(D_S,D_\rho)$-plane, indicating that a larger diffusion strength
$D_\rho$ is required to suppress instability. Equivalently, stronger
nonlocal interaction tends to enlarge the unstable regime. Thus,
although increasing $D_S$ stabilizes the system by dispersing attention,
this effect competes with the destabilizing influence of larger
interaction radius. In particular, when both $R$ is large and $D_S$ is
small, the system is most prone to clustering.

\subsubsection{When attention feedback is absent}

%The above results demonstrate how environmental feedback modifies the
%instability threshold. To isolate its effect, we now examine a
%baseline model in which this mechanism is absent.

When $A$ is spatially homogeneous, namely, $\nabla A = 0$,  the NODAR system
reduces to the nonlinear nonlocal Fokker-Planck equation
\eqref{eq:HK_nonlocal} for the opinion density $\rho$.
By the linearization near the homogeneous equilibrium $\rho\equiv\rho_0$, we obtain the dispersion relation
\begin{equation}\label{eq:dispersion-nd}
	\lambda(k)=\frac{\rho_0}{Z_R}m(k)-D_\rho|k|^2,
\end{equation}
so that the critical point reads
\begin{equation}\label{eq:Dk-base}
	D_{\rho,\mathrm{crit}}(k,R)
	:=
	\frac{\rho_0}{Z_R}\frac{m(k)}{|k|^2}.
\end{equation}

%where $Z_R=\rho_0\pi R^2+\varepsilon$ and
%$m(k)=2\pi R^2\left(\frac{2J_1(|k|R)}{|k|R} - J_0(|k|R) \right)$. 

The corresponding instability threshold is characterized as follows.

\begin{theorem}[Instability threshold without attention]
	\label{thm:ncd-threshold}
	Define the global threshold
	\begin{equation}\label{eq:Dk-HK}
		D_{\mathrm{crit}}(R)
		:=
		\sup_{\substack{k\neq0\\ m(k)>0}}
		\frac{\rho_0}{Z_R}\frac{m(k)}{|k|^2}.
	\end{equation}
	%Let $\lambda(k)$ be given by \eqref{eq:dispersion-nd}. 
	The homogeneous equilibrium is linearly stable for
	$D_\rho>D_{\mathrm{crit}}(R)$ and unstable for
	$0<D_\rho<D_{\mathrm{crit}}(R)$.
\end{theorem}

The dispersion relation and the corresponding phase structure are
shown in Fig.~\ref{fig:ncdf_dispersion_and_threshold}. A finite band of
unstable modes appears when $D_\rho<D_{\mathrm{crit}}(R)$, leading to
aggregation. The phase diagram in the $(R,D_\rho)$-plane shows that
the curve $D_{\mathrm{crit}}(R)$ separates the stable regime from the
finite-wavelength instability regime.

\begin{figure}[htbp]
	\centering
	\begin{subfigure}{0.47\textwidth}
		\centering
		\includegraphics[width=\linewidth]{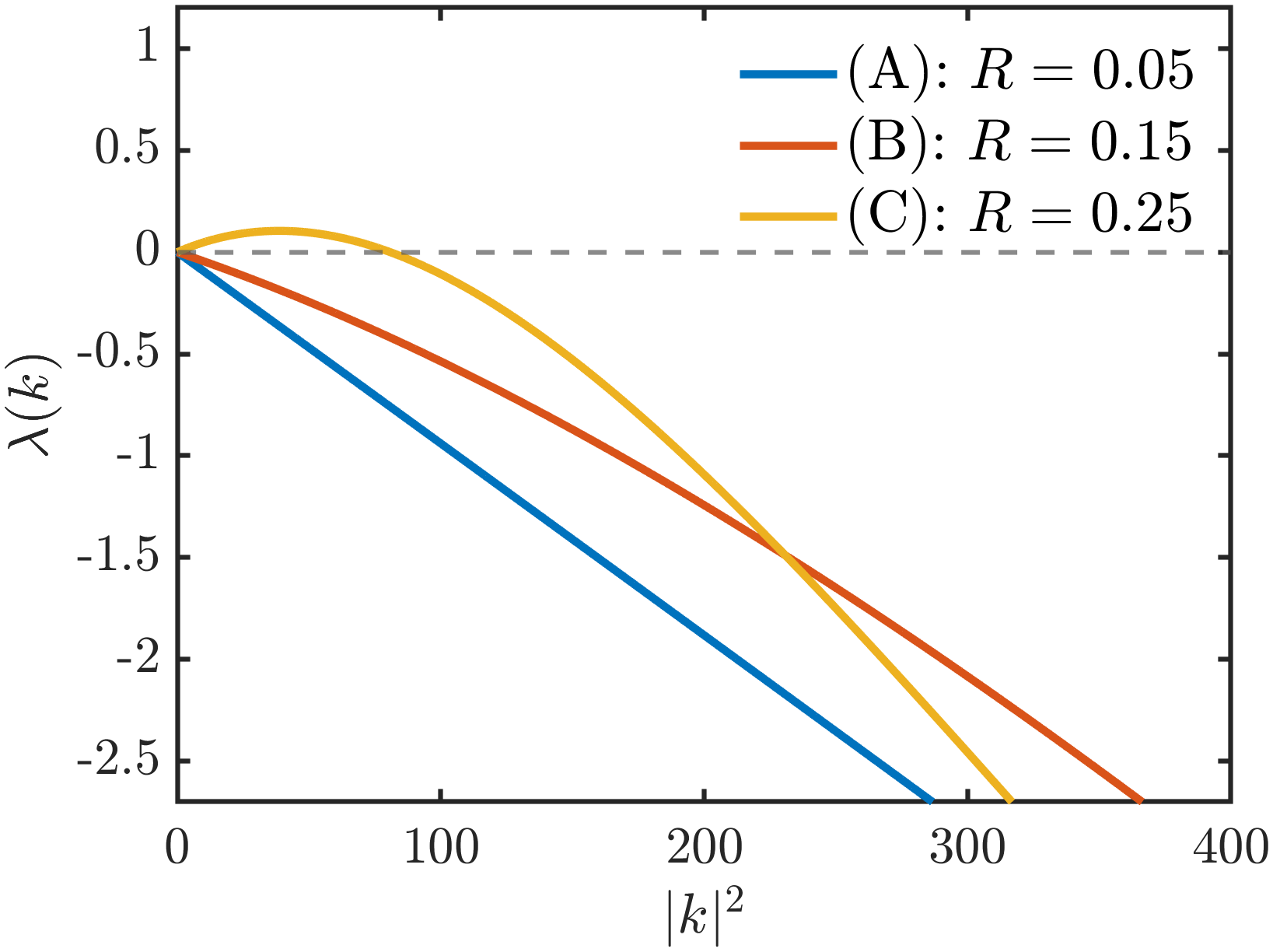}
	\end{subfigure}
	\hfill
	\begin{subfigure}{0.47\textwidth}
		\centering
		\includegraphics[width=\linewidth]{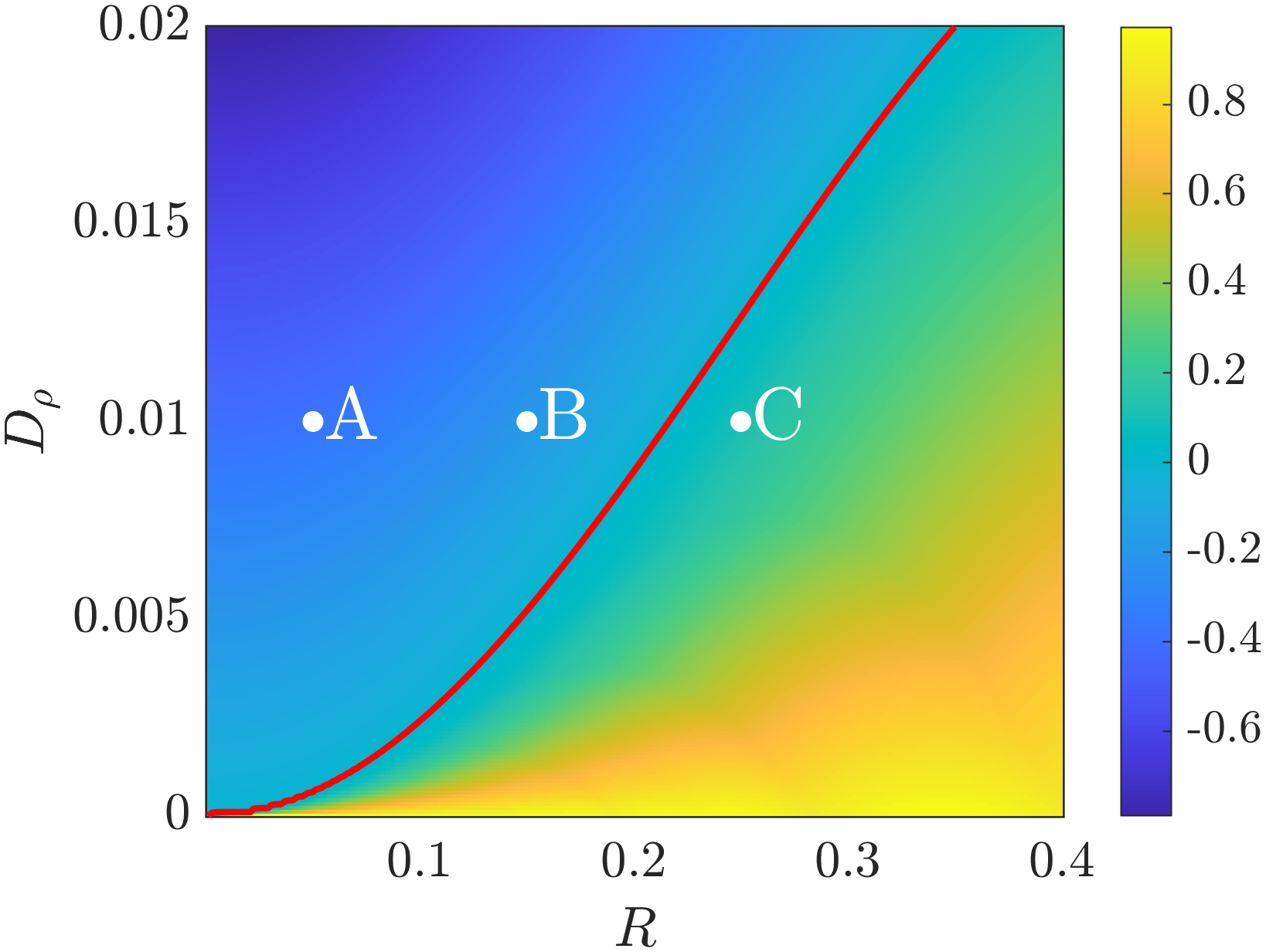}
	\end{subfigure}
	\caption{
		Left: dispersion relation $\lambda(k)$.
		Right: phase diagram of $\lambda_{\max}(R,D_\rho)$ in the
		$(R,D_\rho)$-plane. The red curve indicates the threshold
		$D_{\mathrm{crit}}(R)$.
	}
	\label{fig:ncdf_dispersion_and_threshold}
\end{figure}

This result shows that, even without attention feedback, instability is
driven by the competition between nonlocal alignment and diffusion.
Comparing Eq.~\eqref{eq:Dk-base} with Eq.~\eqref{eq:Dk-accd}, it is seen that the environmental feedback enlarges the instability threshold in the opinion alignment, thereby promoting aggregation.

\subsection{Interpretation of phase transition}

The instability of the homogeneous equilibrium leads to the amplification
of spatial inhomogeneities and the formation of clustered opinion
patterns. This provides a natural interpretation of the emergence of
non-consensus states. In this regime, the opinion density does not relax
toward a spatially homogeneous distribution, but remains heterogeneous in
space, with several localized regions of high opinion activity emerging
in the patterned state. In this sense, the homogeneous equilibrium
represents a consensus state at the macroscopic level, whereas the
unstable regime corresponds to non-consensus state. To interpret the underlying mechanisms, we return to two interacting
effects: Social conformity and attention-mediated
feedback.

In the present model, the
alignment term represents the conformity effect and drives opinion
density toward regions where nearby opinion activity is already
concentrated, whereas diffusion promotes spatial mixing and
homogenization. The interaction radius $R$ characterizes the spatial range over
which conformity acts. The diffusion coefficient $D_\rho$
measures the strength of spatial mixing in the opinion density and can
be interpreted as an effective level of exposure to diverse information.

In the absence of attention feedback, instability is governed by the
competition between nonlocal alignment and diffusion (see Theorem~\ref{thm:ncd-threshold}).  As
illustrated in Fig.~\ref{fig:ncdf_dispersion_and_threshold},
sufficiently large $D_\rho$ or sufficiently small $R$ stabilizes the
homogeneous equilibrium, while $D_\rho$ below the corresponding critical
threshold amplifies spatial perturbations and leads to clustered
opinion patterns. This is consistent with recent empirical observations of socially driven motion,
where distance-dependent alignment interactions produce phase-like
transitions between different collective configurations
\cite{SarkerZhangPerryMessingerSong2026}. 

%accounts for the role of
%alignment as an organizing mechanism in collective social behavior is

%A larger $R$ means that opinion activity is
%influenced by a broader neighborhood, whereas a smaller $R$ corresponds
%to more localized interactions. 

%The associated diffusion threshold marks the point at
%which alignment overcomes diffusion and spatial aggregation begins. 

%Spatial heterogeneity is also well documented in empirical studies and
%data-driven models of public opinion and voting behavior, where
%geographic context and spatially embedded social interactions help
%explain regional differences in political attitudes and electoral
%outcomes
%\cite{BalsaBarreiro2022SocialSpace,ChuDongesRobertsonPopEleches2021SpatialPolarization,JohnstonManleyJones2016SpatialPolarization}.

The attention-mediated feedback is incorporated through the attention
field $S$. Regions of high opinion density generate stronger attention,
and this increased attention feeds back into the evolution of the opinion
density through the attention-dependent cross-diffusion term. This
creates a positive feedback loop between opinion activity and attention:
High-density regions attract more attention, and the resulting attention
reinforcement further enhances local concentration. The attention diffusion coefficient $D_S$ controls the spatial spread of
attention. For  small $D_S$, attention remains localized and the feedback is strong. Consequently, it raises the level
of diffusion required to suppress aggregation,  enlarges the
unstable regime and makes non-consensus states more likely to appear. This may account for the echo chamber effects,
where repeated exposure and local reinforcement can amplify the
opinion activity
\cite{CinelliMoralesGaleazziQuattrociocchiStarnini2021,RaoufiHamannRomanczuk2025}.

\section{Computer simulations}\label{sec:simulation}

To further support the theoretical predictions and to illustrate
the pattern formation induced by the interplay between nonlocal
aggregation, diffusion, and attention feedback,  numerical simulations of the NODAR system 
\eqref{eq:HK-S-scaled} are performed.

All the simulations are carried out on the periodic domain
$\Omega=[0,1]^2$ using a Fourier pseudo-spectral discretization in
space combined with an IMEX Runge-Kutta time integrator.
We first illustrate our numerical scheme, and then report the phase
transition behavior and the role of nonlocal aggregation.
%-------------
\subsection{Structure-preserving IMEX spectral method}

The coupled system can be rewritten in a compact form
\[
U = (\rho, S)^\top,
\quad
\partial_t U = \mathcal L U + \mathcal N(U,t),
\]
where
\[
\mathcal L =
\begin{pmatrix}
	D_\rho \Delta & 0 \\[1mm]
	0 & D_S \Delta - \omega
\end{pmatrix},
\quad
\mathcal N(U,t) =
\begin{pmatrix}
	-\nabla\!\cdot(\rho V[\rho])
	-
	D_\rho \nabla\!\cdot\!\big( \tfrac{2\rho}{A}\nabla A \big)
	\\[2mm]
	\theta \rho 
\end{pmatrix}.
\]

%\iffalse
%On the periodic domain, the Fourier pseudo-spectral method is employed
%for spatial discretisation.
%The linear operator $\mathcal L$ is diagonal in Fourier space and can
%therefore be evaluated exactly at each Fourier mode.
%All nonlinear terms in $\mathcal N$ are computed in physical space, with
%FFT transforms used to switch between representations.
%\fi

The principal computational difficulty lies in the evaluation of the
nonlocal drift $V[\rho]$. On the periodic domain, we adopt a spectral
approach and exploit the convolution structure of the disk kernel.
For $K_R(z)=\mathbf 1_{\{|z|\le R\}}$, its Fourier coefficients admit an
explicit representation in terms of Bessel functions. Denoting
$a=|k|R$, 
the corresponding spectral multipliers are $\widehat K_R(k)=2\pi R^2\frac{J_1(a)}{a}$,
Given $\widehat\rho(k)$, the denominator is evaluated by
$\widehat{K_R * \rho}(k)=\widehat K_R(k)\widehat\rho(k)$.
For the numerator in \(V[\rho]\), the Fourier transform is computed as
\[
\widehat{z_jK_R}(-k)\widehat\rho(k)
=
i\,k_j\,2\pi R^4
\frac{2J_1(a)-aJ_0(a)}{a^3}\widehat\rho(k),
\quad j=1,2.
\]
The corresponding quantities are then obtained by inverse FFT.
%The transport term $-\nabla\!\cdot(\rho V[\rho])$ is then computed by
%forming the flux $\rho V[\rho]$ in physical space and evaluating its
%divergence spectrally. 
For small $a$, Taylor expansions are used to
avoid cancellation errors. This reduces the computational cost to $\mathcal O(N^2\log N)$ and avoids
direct evaluation of the double integral.

Time integration is carried out using the third-order IMEX Runge-Kutta
scheme ARS$(2,3,3)$ \cite{AscherRuuthSpiteri1997}, in which the diffusive
terms are treated implicitly, while the transport and nonlocal interaction
terms are handled explicitly. 

Let $\gamma=\frac{3+\sqrt{3}}{6}$ and denote by
$U^n\approx U(t^n)$ the numerical solution at time $t^n$.
\[
\begin{cases}
	U^{(1)} = U^n + \gamma\,\Delta t\,\mathcal N(U^n,t^n),\\[2mm]
	(I-\gamma\Delta t\,\mathcal L)\,U^{(2)} = U^{(1)},\\[2mm]
	U^{(3)}
	=
	\frac{1}{\gamma} U^n
	+
	\frac{3\gamma - 2}{\gamma} U^{(1)}
	+
	\frac{1 - 2\gamma}{\gamma} U^{(2)}
	+
	2(1-\gamma)\Delta t\,\mathcal N(U^{(2)}, t^n + \gamma \Delta t),\\[2mm]
	(I-\gamma\Delta t\,\mathcal L)\,U^{(4)}=U^{(3)},\\[2mm]
	U^{n+1}
	=
	-\tfrac12 U^n
	-\tfrac{3\gamma}{2} U^{(1)}
	+\tfrac32 U^{(2)}
	+\tfrac{3(3\gamma-2)}{2} U^{(3)}
	+\tfrac{1}{2\gamma} U^{(4)}
	+
	\tfrac{\Delta t}{2}\,
	\mathcal N(U^{(4)}, t^{n+1}).
\end{cases}
\]

%Since the linear operator $\mathcal L$ is diagonal in Fourier space,
%all implicit stages reduce to scalar divisions at each Fourier mode.
%%No iterative linear solver is required.
%%The diffusive part of the scheme is $L$-stable,
%The restriction on time step is determined solely by
%the explicit nonlocal transport and cross-diffusion terms.

Because the nonlinear transport, cross-diffusion and source terms
are treated explicitly, the intermediate solution
$(\tilde\rho_h^{\,k+1}, \tilde S_h^{\,k+1})$
may violate non-negativity and the exact mass conservation.
To restore the admissible structure,
we apply a prediction-correction step based on a
Lagrange multiplier projection \cite{cheng2022new,cheng2022bound,TongFenghua2024Positivity,van2019positivity,WangXiongZhang2025} onto the constraint set
after each timestep. 
The corrected solution
$(\rho_h^{k+1}, S_h^{k+1})$
is obtained as the minimizer of the constrained optimization problem
\[
\frac12\|\rho_h^{k+1}-\tilde\rho_h^{\,k+1}\|_{L^2}^2
+
\frac12\|S_h^{k+1}-\tilde S_h^{\,k+1}\|_{H^1}^2
\]
subject to
\[
\rho_h^{k+1}\ge0,
\quad
S_h^{k+1}\ge0,
\quad
\langle \rho_h^{k+1},1\rangle
=
\langle \rho_h^{0},1\rangle.
\]
The $L^2$-projection enforces positivity and mass conservation for
$\rho$, while the $H^1$ projection guarantees positivity for $S$.
Strict convexity ensures existence and uniqueness of the minimizer. 
%The optimality conditions for the
%$L^2$-$H^1$ projection are given as follows. Given the predicted solution
%$(\tilde\rho_h^{\,k+1}, \tilde S_h^{\,k+1})$,
%the corrected state
%$(\rho_h^{k+1}, S_h^{k+1})$
%is defined as the solution of the convex minimisation problem
%\begin{equation*}
%	\begin{aligned}
	%		\min_{\rho_h^{k+1},\,S_h^{k+1}\in X}
	%		\;&
	%		\frac12\Big(
	%		\|\rho_h^{k+1}-\tilde\rho_h^{\,k+1}\|_{L^2}^2
	%		+
	%		\|S_h^{k+1}-\tilde S_h^{\,k+1}\|_{H^1}^2
	%		\Big),
	%		\\
	%		\text{s.t.}\;&
	%		\rho_h^{k+1}\ge0,
	%		\quad
	%		S_h^{k+1}\ge0,
	%		\quad
	%		\langle \rho_h^{k+1},1\rangle
	%		=
	%		\langle \rho_h^{0},1\rangle .
	%	\end{aligned}
%\end{equation*}

%Since the objective functional is strictly convex
%and the admissible set is closed and convex,
%a unique minimiser exists.

The Karush-Kuhn-Tucker conditions
characterizing the projection are as follows.
There exist nonnegative grid functions
$\lambda_h^{k+1}\ge0$,
$\mu_h^{k+1}\ge0$
and a scalar multiplier
$\xi^{k+1}\in\mathbb{R}$ such that
\begin{equation*}
	\begin{aligned}
		&\rho_h^{k+1}
		=
		\tilde\rho_h^{\,k+1}
		+
		\lambda_h^{k+1}
		-
		\xi^{k+1},
		\quad (I-\Delta)\,S_h^{k+1}
		=
		(I-\Delta)\,\tilde S_h^{\,k+1}
		+
		\mu_h^{k+1},
		\\
		&\lambda_h^{k+1}\,\rho_h^{k+1} = 0,
		\quad
		\mu_h^{k+1}\,S_h^{k+1} = 0,
		\quad
		\lambda_h^{k+1}\ge0,
		\quad
		\mu_h^{k+1}\ge0,
		\\
		&\langle \rho_h^{k+1},1\rangle
		=
		\langle \rho_h^{0},1\rangle .
	\end{aligned}
\end{equation*}
Here $\lambda_h^{k+1}$ and $\mu_h^{k+1}$
enforce pointwise non-negativity,
while $\xi^{k+1}$ guarantees exact mass conservation.
The operator $(I-\Delta)$ arises from
the Euler-Lagrange equation associated with the $H^1$ inner product.

%The associated optimality conditions are summarized in
%\ref{app:projection}. The proposed scheme has been validated in \ref{app:verification} through convergence tests and shown to preserve key structural properties, including non-negativity and mass conservation.

%We now investigate the dynamical behaviour of the NODAR system using the established numerical discretization. 

%We therefore focus on qualitative dynamics and on the mechanisms underlying phase transitions predicted by linear stability analysis.

\subsection{Pattern formation and wavelength selection}

Now we investigate the phase transition and the associated
pattern selection predicted by the linear stability analysis.
Numerical simulations are conducted to validate the phase diagram in
Fig.~\ref{fig:lambda_phase_diagrams} and to illustrate the emergence of
spatial patterns across different dynamical regimes, thereby confirming
the theoretical predictions.

\begin{example}\label{exa:random}
	The initial condition is constructed as a small perturbation of the
	homogeneous equilibrium \eqref{eq:equilibrium},
	\[
	\rho(x,y,0)=\rho_0+\delta(x,y), \quad
	S(x,y,0)=S_0+\delta(x,y),
	\]
	where $\delta(x,y)$ is a random field given by a superposition of
	$30$ Gaussian functions with randomly chosen centres, amplitudes,
	and widths. The computations are performed on the periodic domain $\Omega=[0,1]^2$ with spatial step size $h=1/512$ and time step $\Delta t_{max}=10^{-3}$.
\end{example}

Fig.~\ref{fig:rho_snapshots} shows snapshots of the opinion density
$\rho(x,y)$ at time $T=50$ for three representative values of the
interaction radius $R$, corresponding to the points
$\mathrm{A}$, $\mathrm{B}$, and $\mathrm{C}$ in
Fig.~\ref{fig:lambda_phase_diagrams}.
The simulations are performed with diffusion coefficient
$D_\rho = 0.01$, while the remaining parameters are the same as in
Fig.~\ref{fig:lambda_phase_diagrams}.
In the stable regime (panel A), spatial perturbations decay and the
solution remains close to the homogeneous equilibrium.
Near the threshold (panel B), small perturbations begin to grow,
indicating the onset of instability.
In the unstable regime (panel C), the solution develops clearly localized spatial clusters, reflecting the dominance of aggregation over
diffusion.
These observations are in good qualitative agreement with the theoretical phase diagram, suggesting that the linear stability analysis captures the transition threshold and provides insight into the initial pattern formation mechanism.

\begin{figure}[h!]
	\centering
	\vspace{-5pt}
	\begin{subfigure}{0.32\textwidth}
		\centering
		\includegraphics[width=\linewidth]{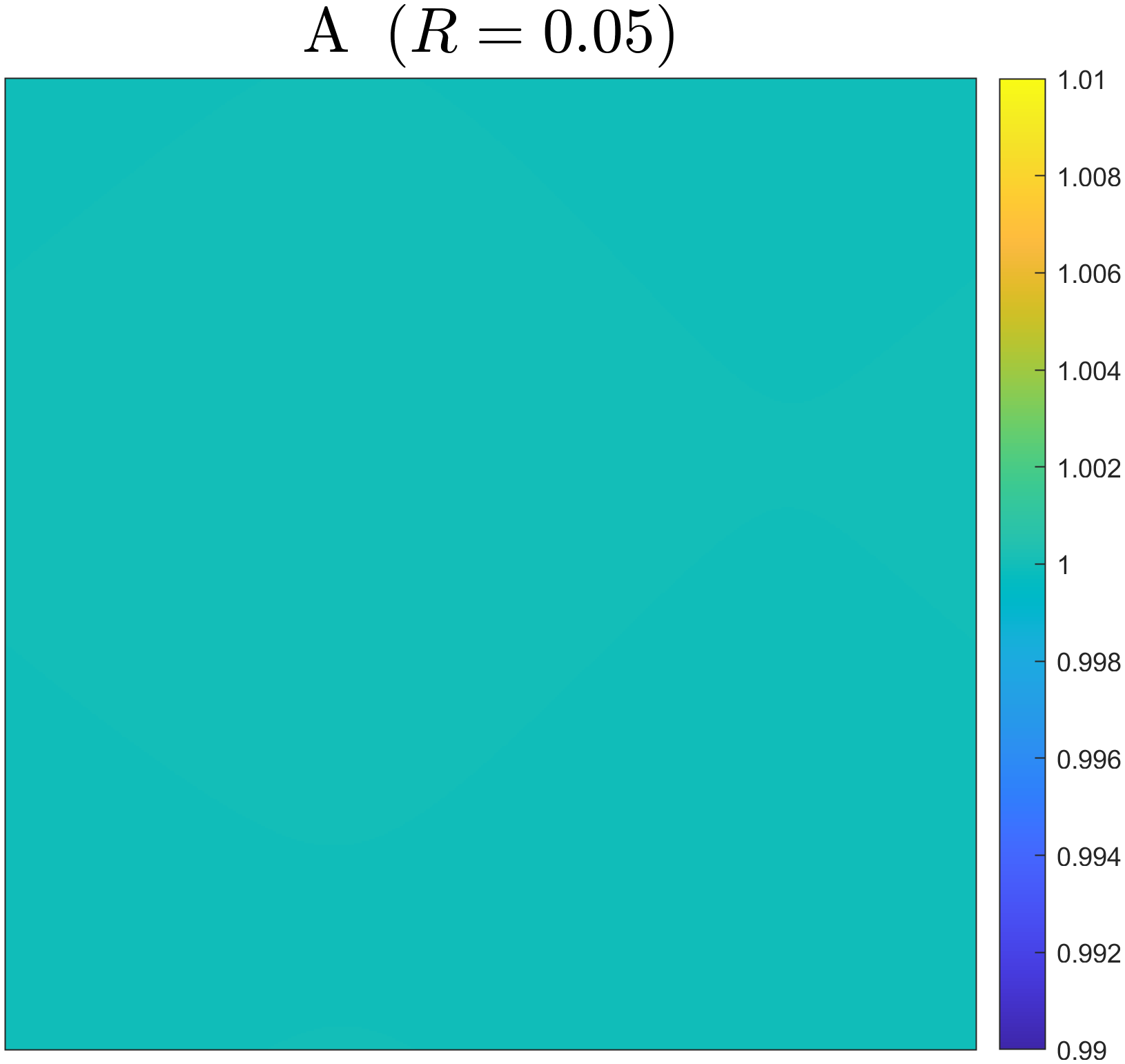}
		\label{fig:lambda_I}
	\end{subfigure}
	\hfill
	\begin{subfigure}{0.32\textwidth}
		\centering
		\includegraphics[width=\linewidth]{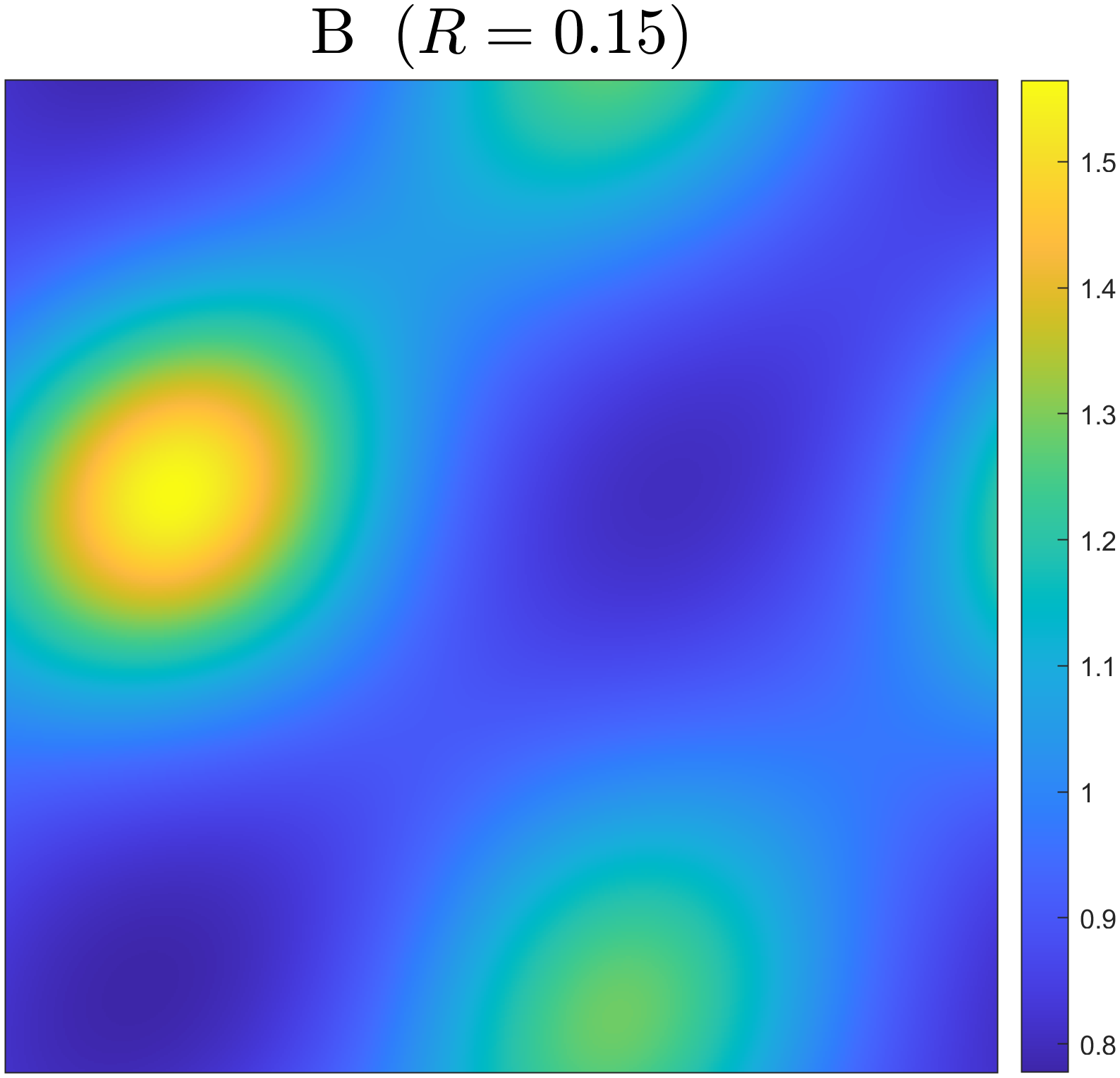}
		\label{fig:lambda_I1}
	\end{subfigure}
	\hfill
	\begin{subfigure}{0.32\textwidth}
		\centering
		\includegraphics[width=\linewidth]{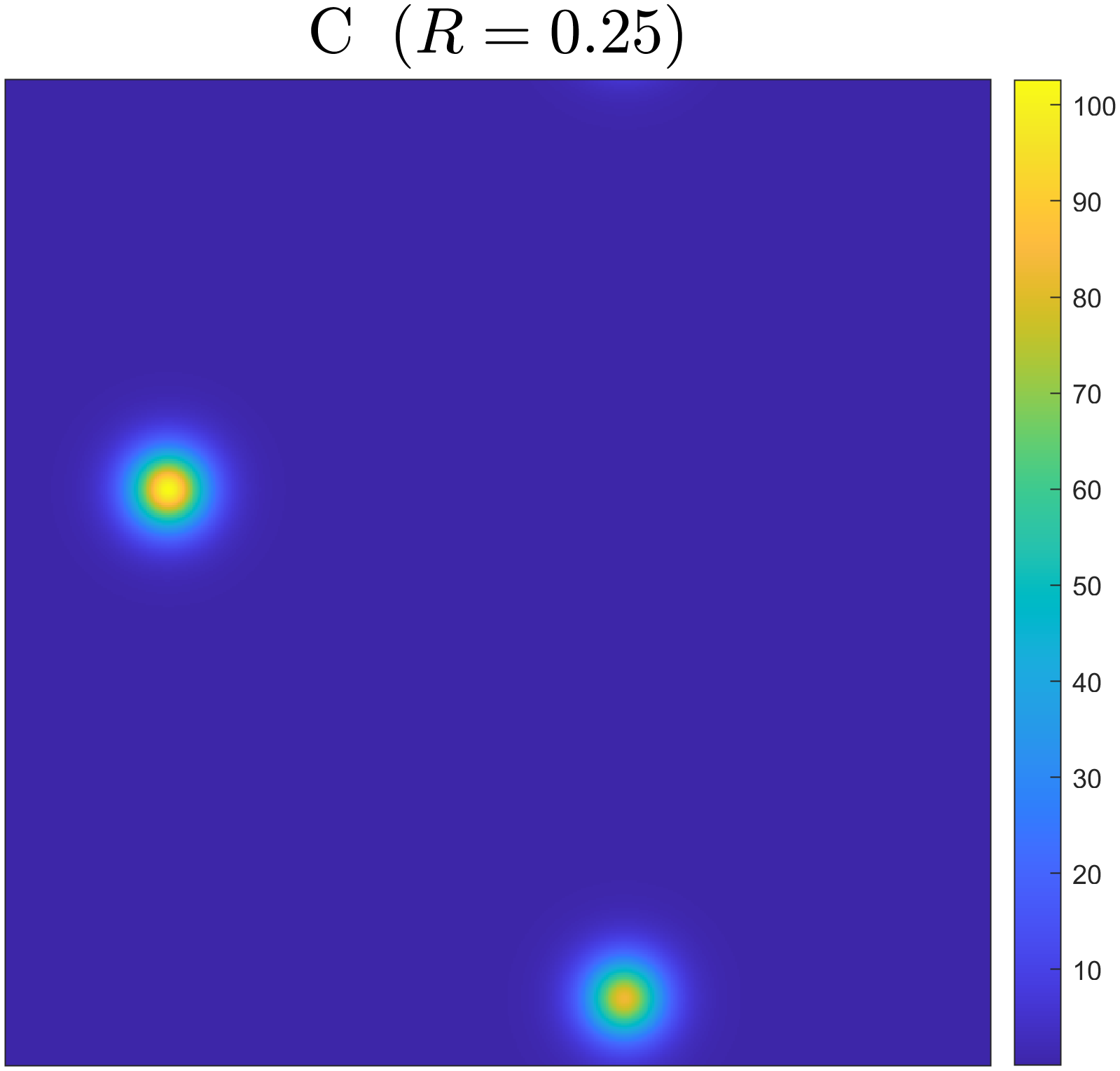}
		\label{fig:lambda_I2}
	\end{subfigure}
	\vspace{-15pt}
	\caption{
		Snapshots of the opinion density $\rho(x,y)$ at $T=50$
		for three representative interaction radii $R=0.05$, $0.15$, and $0.25$,
		corresponding to points $\mathrm{A}$-$\mathrm{C}$ in
		Fig.~\ref{fig:lambda_phase_diagrams}.
	}
	\label{fig:rho_snapshots}
	\vspace{-10pt}
\end{figure}

We next examine the dynamical evolution of the system for different interaction radii. 
Fig.~\ref{fig:rho_evolution} shows representative time snapshots for $R=0.06,0.1,0.15$ with $D_\rho = 10^{-4}$, corresponding to different dynamical regimes.
Starting from small random perturbations around the homogeneous state, weak spatial inhomogeneities emerge and are progressively amplified, leading to the formation of localized high-density regions. 
As time evolves, these regions become more pronounced and eventually organize into clusters.
We observe that the interaction radius $R$ has a significant influence on the spatial organization of the patterns: Larger values of $R$ lead to more widely spaced clusters. 
This behavior is consistent with the wavelength selection mechanism predicted by the linear stability analysis.
From a macroscopic perspective, the clustering process can be interpreted as the formation of locally coherent regions, where the opinion density becomes highly concentrated around certain spatial locations. These regions correspond to local consensus, while the coexistence of multiple clusters reflects a global non-consensus state.

\begin{figure}[h!]
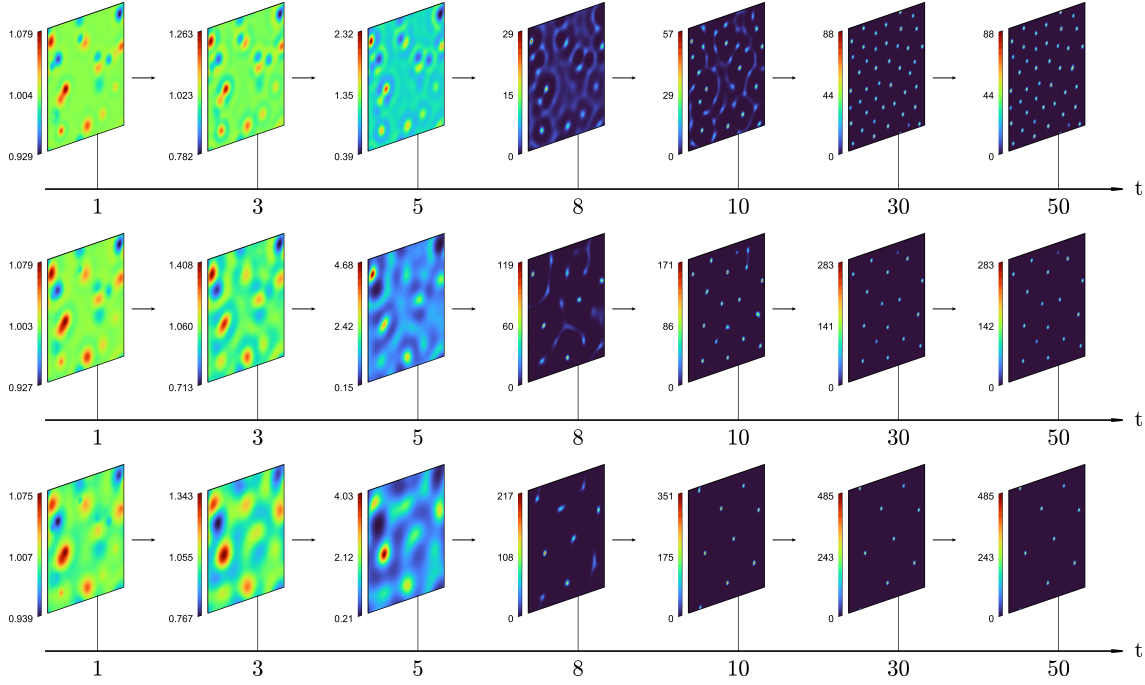

	\centering
	
	\begin{subfigure}{1.0\textwidth}
		\centering
		\includegraphics[width=\linewidth]{figures/rho_R006_D1e-4.png}
	\end{subfigure}\\[0.2cm]
	\begin{subfigure}{1.0\textwidth}
		\centering
		\includegraphics[width=\linewidth]{figures/rho_R010_D1e-4.png}
	\end{subfigure}\\[0.2cm]
	\begin{subfigure}{1.0\textwidth}
		\centering
		\includegraphics[width=\linewidth]{figures/rho_R015_D1e-4.png}
	\end{subfigure}
	\caption{
		Time evolution of the opinion density $\rho(x,y,t)$ for different interaction radii 
		$R=0.06$ (top), $R=0.10$ (middle), and $R=0.15$ (bottom). 
		Snapshots are shown at $t=1,3,5,8,10,30,50$. 
		The system evolves from small perturbations of the homogeneous state to localized clusters, 
		with larger $R$ producing wider inter-cluster spacing.
	}
	\label{fig:rho_evolution}
\end{figure}

To characterize the pattern scale, we consider the wavelength $\ell^\ast(R)$ associated with the fastest-growing mode of the dispersion relation (see Remark~\ref{remark:L}).

For the disk kernel $K_R(z)=\mathbf{1}_{\{|z|\le R\}}$, introducing the rescaled variable $q=|k|R$, we obtain
$m(k)=2\pi R^2 M(q)$, $M(q)=\frac{2J_1(q)}{q}-J_0(q)$.
It follows that
\[
\ell^\ast(R)=\frac{2\pi}{|k^\ast(R)|}
=\frac{2\pi R}{q^\ast(R)},
\quad
q^\ast(R)=R|k^\ast(R)|.
\]

In the weak diffusion regime $D_\rho \ll 1$,  the dominant wavenumber is mainly governed
by the kernel-dependent function \(M(q)\), while the diffusive contribution
$-D_\rho |k|^2=-D_\rho q^2/R^2$
only weakly shifts the maximizer of the full dispersion relation.
Numerical evaluation shows that the resulting maximizer \(q^\ast(R)\)
depends only mildly on \(R\) over the range of interaction radii
considered. Consequently,
\[
\ell^\ast(R)=\frac{2\pi R}{q^\ast(R)}
\approx C R,
\]
where \(C=2\pi/q^\ast\) is an \(\mathcal O(1)\) constant determined by
the dominant mode. For the disk-kernel multiplier \(M(q)\), the first
maximum occurs near \(q_0\approx 3.05\). Since \(q^\ast(R)\) remains close
to \(q_0\) in the aggregation-dominated regime, we obtain
$\ell^\ast(R)\approx 2.06R$,
which resembles the 2R-conjecture in the agent-based modeling \cite{BlondelTsitsiklis2007}.

\begin{figure}[h!]
	\centering
	\includegraphics[width=0.48\textwidth]{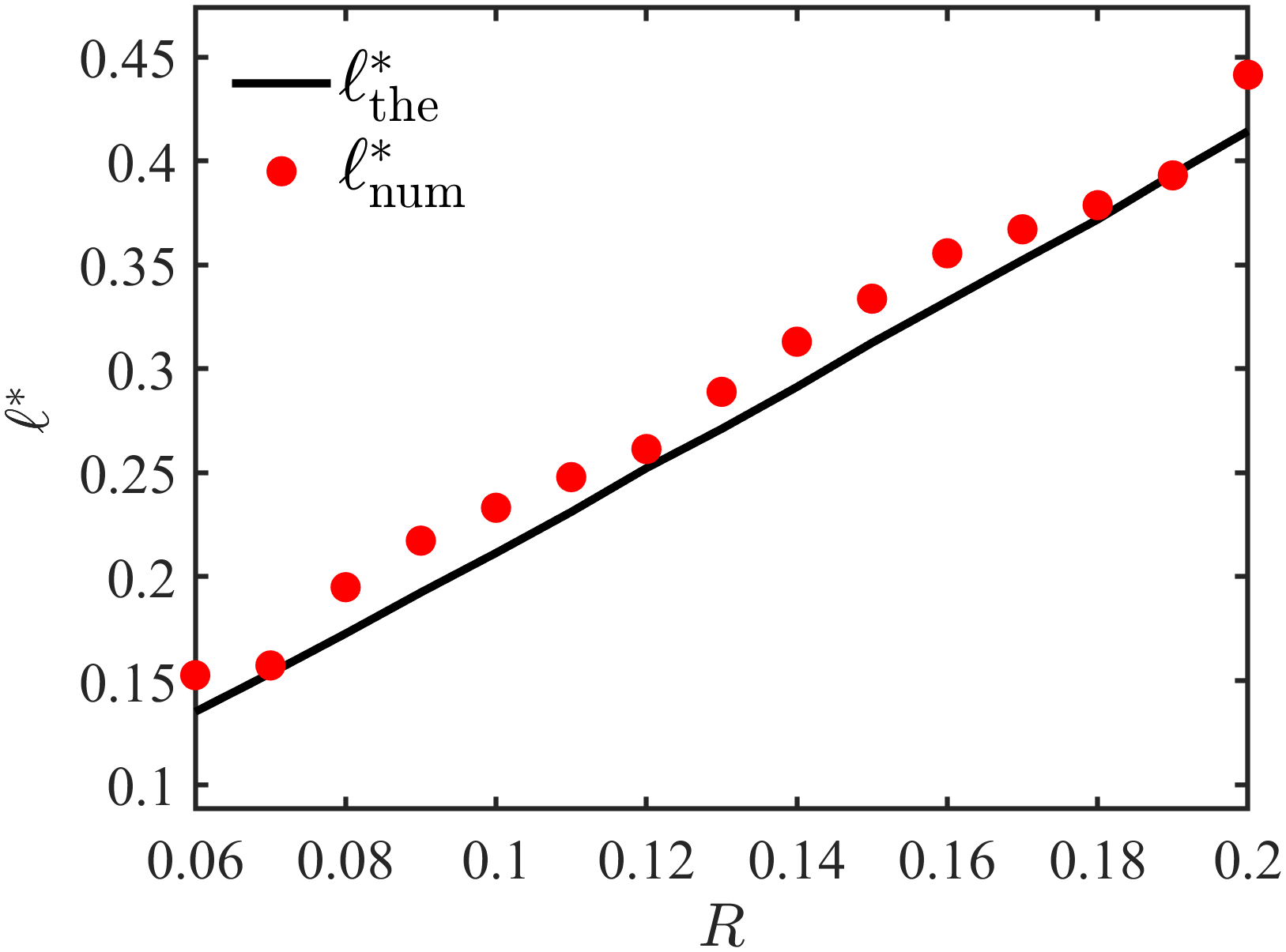}
	\includegraphics[width=0.48\textwidth]{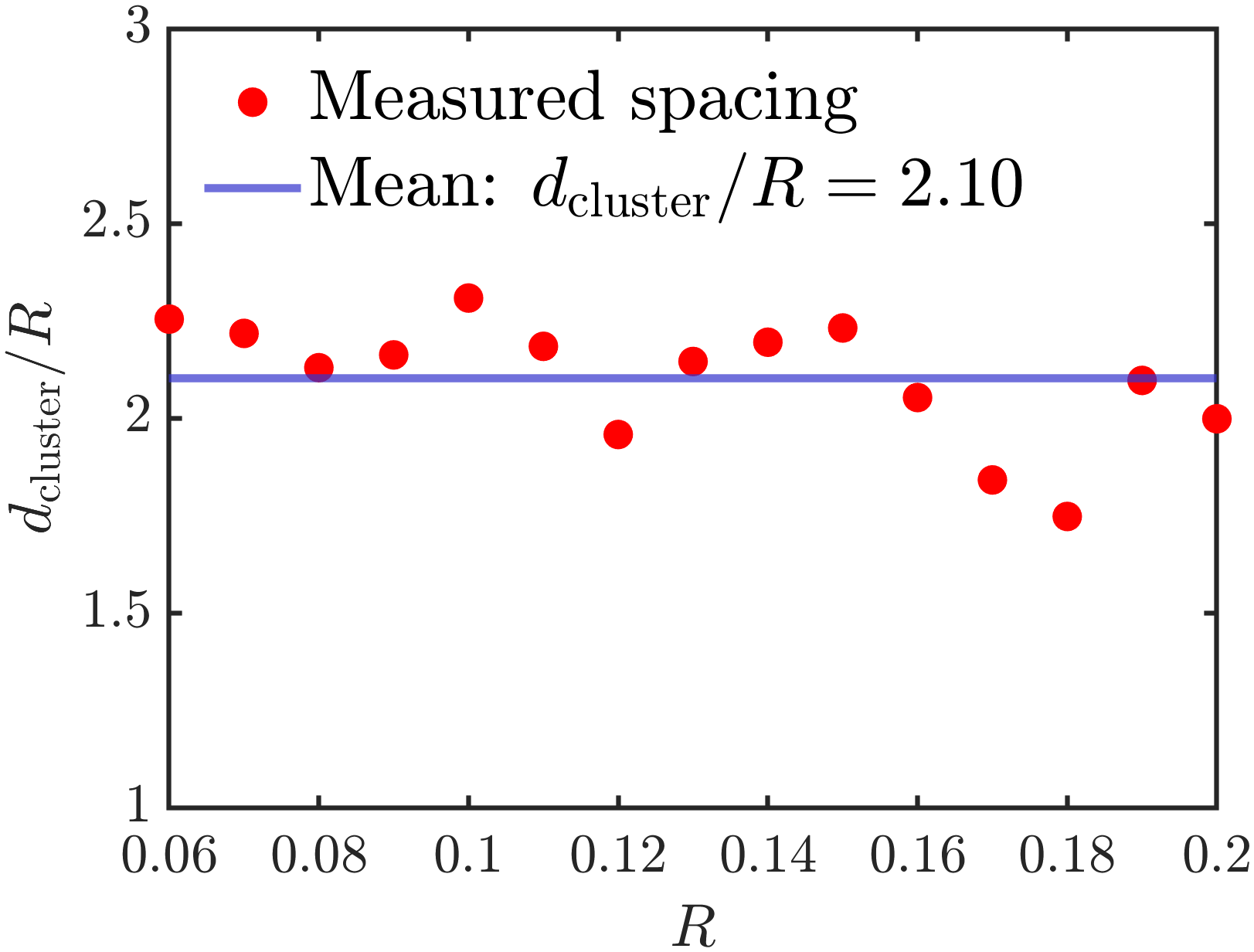}
	\caption{ Wavelength selection and cluster spacing for
		\(D_\rho=10^{-4}\).
		Left: Comparison between the theoretical prediction of the dominant wavelength $\ell^\ast(R)$ (black curve) and its numerical estimate $\ell^\ast_{\mathrm{num}}$ (red dots), showing good consistency and an approximately linear dependence on \(R\).
		Right: Cluster spacing $d_{\mathrm{cluster}}/R$. The ratio remains
		close to a constant, with mean value about \(2.10\), indicating that
		the  cluster spacing is close to \(2R\).
	}
	\label{fig:cluster_L_star}
	\vspace{-10pt}
\end{figure}

To further validate this prediction, we compare the theoretical wavelength $\ell^\ast(R)$ with numerical results extracted from the early-stage dynamics. 
As shown in Fig.~\ref{fig:cluster_L_star} (left), the numerical estimate $\ell^\ast_{\mathrm{num}}$ exhibits a nearly linear dependence on $R$ as the theoretical prediction. 
In particular, the ratio $\ell^\ast_{\mathrm{num}}/R$ remains nearly constant over the considered range of interaction radii, indicating that the dominant wavelength scales linearly with $R$, say, $\ell^\ast(R)=\mathcal{O}(R)$.

Next, we define the cluster spacing $d_{\mathrm{cluster}}$ as the average nearest-neighbour distance between local density maxima at late times. 
Fig.~\ref{fig:cluster_L_star} (right) shows the normalized spacing $d_{\mathrm{cluster}}/R$, which remains approximately constant across the same parameter range, with mean
value about \(2.10\). 
This indicates that the cluster spacing
also scales linearly with \(R\), say, $d_{\mathrm{cluster}}=\mathcal{O}(R)$.

%While $\ell^\ast(R)$ arises from linear stability analysis and reflects the early-stage growth of perturbations, $d_{\mathrm{cluster}}$ is measured from fully developed nonlinear patterns. 
%The observed agreement between these two scales indicates that the characteristic pattern size in the nonlinear regime is closely linked to the dominant wavelength selected by linear instability.
%Taken together, these observations might indicate a robust linear scaling behavior across both regimes:
%\[
%\ell^\ast(R)=\mathcal{O}(R), 
%\quad 
%d_{\mathrm{cluster}}=\mathcal{O}(R),
%\]
%with both the linear wavelength and the nonlinear cluster spacing close
%to the \(2R\) scale in the weak-diffusion aggregation regime.

This \(2R\) scaling resembles the classical 2R-conjecture in
bounded-confidence opinion dynamics, which states that
the H-K particle model with confidence radius \(R\)
typically forms final clusters separated by a distance of order \(2R\)
\cite{BlondelTsitsiklis2007,WangLiEChazelle2017}. In contrast
to the particle setting, the \(2R\) scaling in the present continuum model
is selected by the fastest-growing mode of the dispersion relation and
is further examined through the cluster spacing measured from nonlinear
simulations.

%
%\iffalse
%\begin{figure}[h!]
%	\centering
%	\vspace{-7pt}
%	\begin{subfigure}{0.4\textwidth}
	%		\centering
	%		\includegraphics[width=\linewidth]{figures/R025_1e-4.png}
	%		\caption{$R=0.25,D_\rho = 10^{-4}$}
	%	\end{subfigure}
%	\begin{subfigure}{0.4\textwidth}
	%		\centering
	%		\includegraphics[width=\linewidth]{figures/R025_4e-4.png}
	%				\caption{$R=0.25,D_\rho = 4\times10^{-4}$}
	%	\end{subfigure}
%		\begin{subfigure}{0.4\textwidth}
	%		\centering
	%		\includegraphics[width=\linewidth]{figures/R025_6e-4.png}
	%				\caption{$R=0.25,D_\rho = 6\times10^{-4}$}
	%	\end{subfigure}
%	\begin{subfigure}{0.4\textwidth}
	%		\centering
	%		\includegraphics[width=\linewidth]{figures/R025_1e-3.png}
	%				\caption{$R=0.25,D_\rho = 10^{-3}$}
	%	\end{subfigure}
%		\begin{subfigure}{0.4\textwidth}
	%		\centering
	%		\includegraphics[width=\linewidth]{figures/d_vs_D.png}
	%		\caption{$d_{avg}$ vs $D_\rho$}
	%	\end{subfigure}
%	\vspace{-3pt}
%	\caption{xxx}
%	\label{fig:xxxx}
%	\vspace{-10pt}
%\end{figure}
%\fi

\subsection{Macroscopic quantification of clustering}
\label{subsubsec:influence-S}

The linear stability analysis characterizes the onset of aggregation in
mathematical terms. To relate this phenomenon to observable patterns in
collective opinion dynamics, we next introduce several macroscopic
quantities that describe how opinion activity is distributed across the
domain. In particular, we compare the NODAR model with the
corresponding reduced model without attention feedback
(i.e., $\nabla A = 0$).
%in which the attention-dependent part of the
%cross-diffusion term vanishes and the dynamics reduce to the scalar
%nonlocal Fokker-Planck equation~\eqref{eq:HK_nonlocal}. 
Both models are solved using the same
numerical scheme.

To quantify the degree of spatial heterogeneity, we first introduce the
variance of the opinion density,
\[
\mathrm{Var}(\rho(t))
=
\frac{1}{|\Omega|}
\int_{\Omega} (\rho(x,t)-\bar{\rho}(t))^2\,dx,
\quad
\bar{\rho}(t)
=
\frac{1}{|\Omega|}\int_{\Omega}\rho(x,t)\,dx.
\]
This quantity provides a global measure of clustering: It is small
when the solution is close to a homogeneous state and increases as
spatial inhomogeneities develop.

\begin{figure}[h!]
	\centering
	\includegraphics[width=0.65\textwidth]{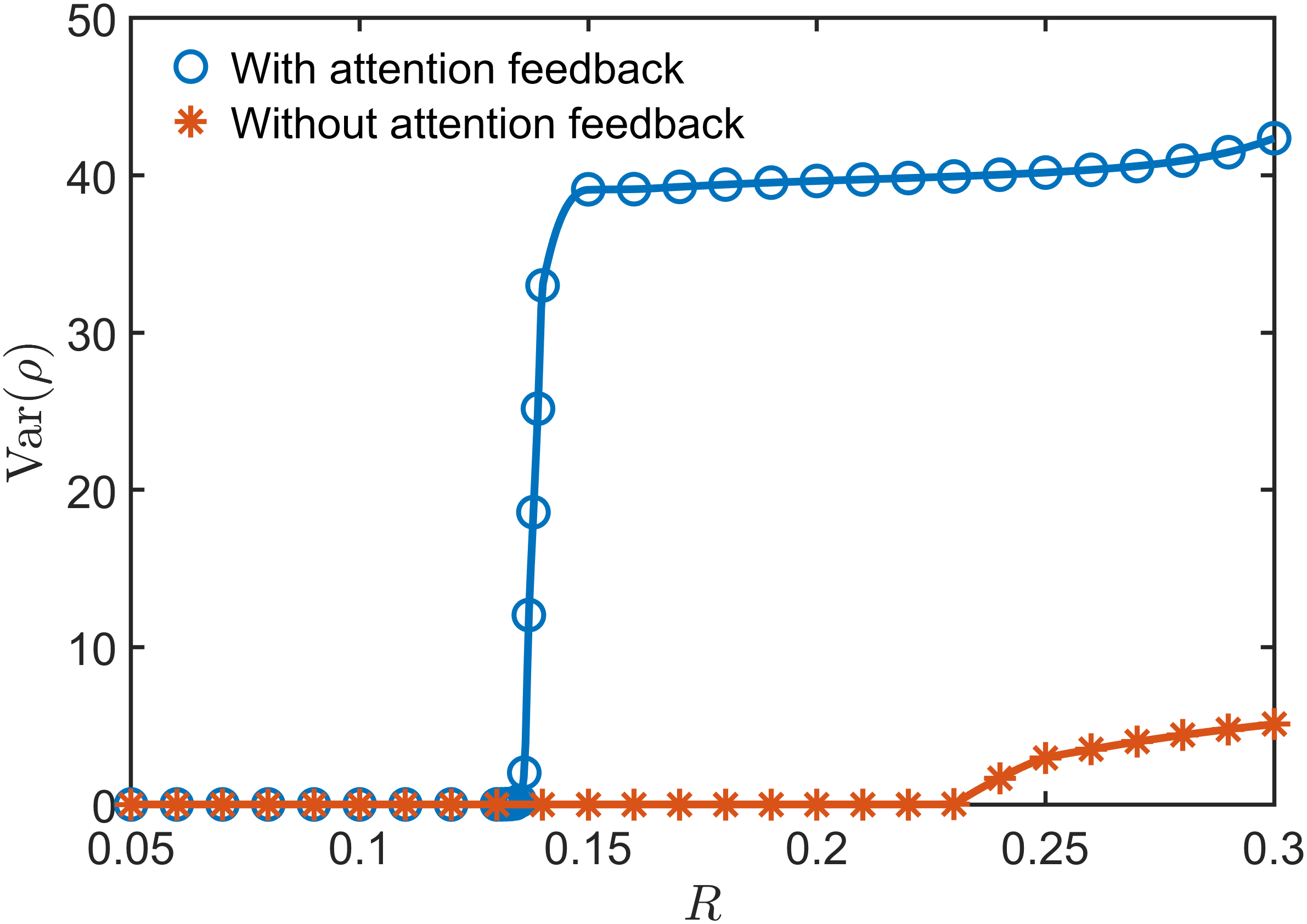}
	\caption{
		Variance $\mathrm{Var}(\rho)$ versus the interaction radius $R$ at $D_\rho =0.01$ for
		the NODAR model and the reduced model without attention feedback.
	}
	\label{fig:Var_rho_vs_R}
\end{figure}

Fig.~\ref{fig:Var_rho_vs_R} shows that, in both models, increasing $R$
drives a transition from a nearly homogeneous state to a clustered
regime, reflected by the rapid growth of $\mathrm{Var}(\rho)$.
However, the onset of clustering is observed at a smaller value of $R$, indicating an earlier onset of clustering in the
presence of attention feedback.

While the variance detects the emergence of spatial heterogeneity, it
does not describe how clusters are organized. To quantify the
internal structure of these patterns and to relate the macroscopic
indicators to the microscopic behavioral interpretation, we introduce
threshold-based macroscopic quantities.

In opinion adjustment, individual-level responses may be broadly described by three mechanisms: Keep, adopt, and compromise
\cite{ChacomaZanette2015,BalsaBarreiro2022SocialSpace}. These responses are not introduced as
separate state variables in the continuum model. Instead, their collective effect is reflected through the spatial redistribution of the opinion density $\rho(x,t)$. In particular, regions where $\rho$ becomes significantly larger than the background level may be interpreted as
macroscopic high-activity regions, where adoption effects are dominant.

Let $\sigma(t):=\sqrt{\mathrm{Var}(\rho(t))}$
denote its standard deviation. For a prescribed threshold parameter
$c>0$, we define the high-density region
\[
E_c(t):=
\{x\in\Omega:\rho(x,t)>\bar\rho(t)+c\,\sigma(t)\}.
\]
The set $E_c(t)$ extracts the high-activity core of the opinion
distribution (adopt). The complement $\Omega\setminus E_c(t)$ represents the
lower-activity background (keep or compromise).
%and is not further decomposed into keep and
%compromise responses.
Define the mass fraction and area fraction of this
high-activity core by
\[
M_c(t):=
\frac{\int_{E_c(t)} \rho(x,t)\,dx}
{\int_{\Omega} \rho(x,t)\,dx},
\quad
A_c(t):=
\frac{|E_c(t)|}{|\Omega|}.
\]
Here $M_c(t)$ measures how much of the total opinion mass is in
the high-density region, while $A_c(t)$ measures how much of the spatial
domain this region occupies.

\begin{figure}[h!]
	\centering
	\begin{subfigure}{0.48\textwidth}
		\centering
		\includegraphics[width=\linewidth]{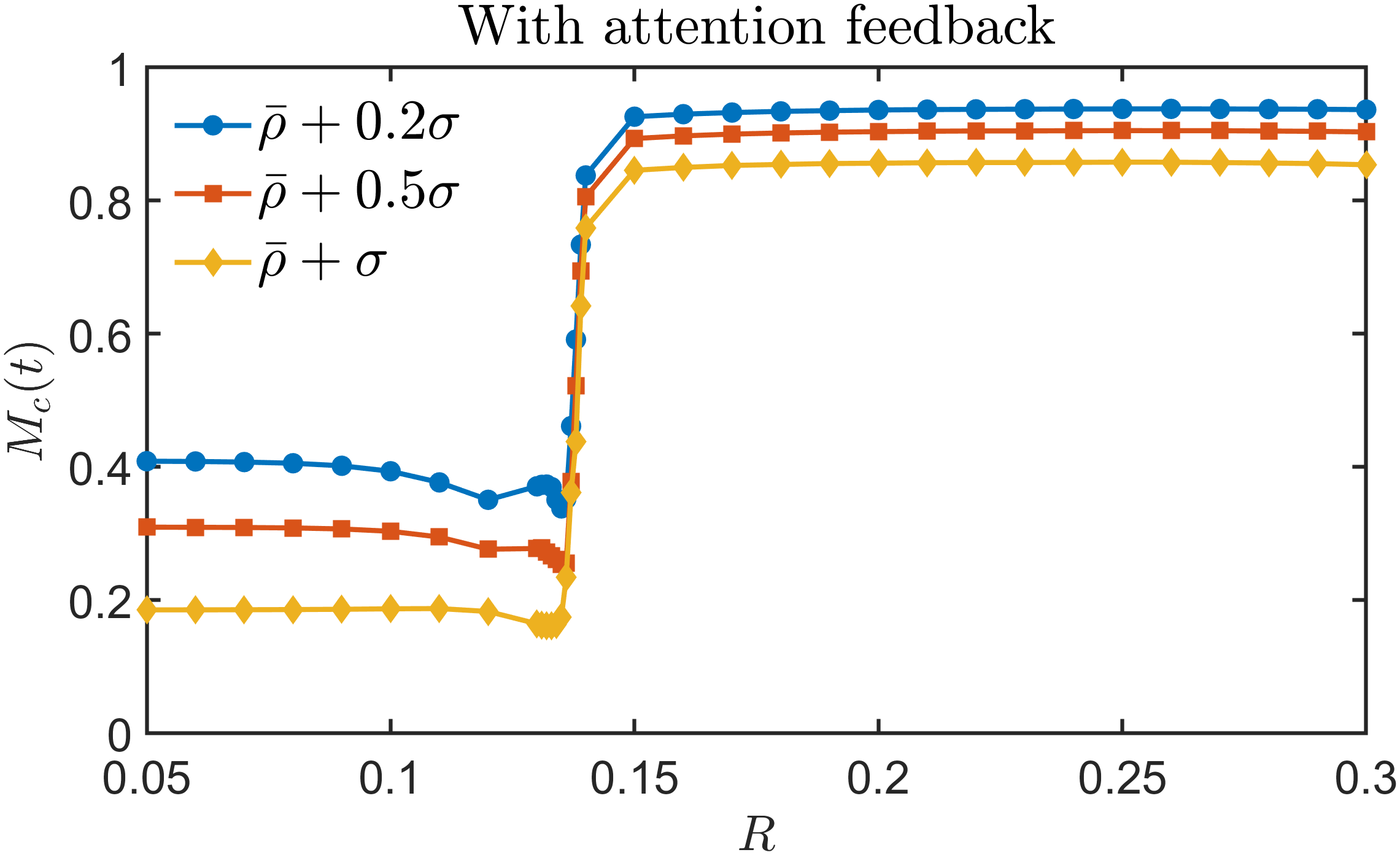}
	\end{subfigure}
	\hfill
	\begin{subfigure}{0.48\textwidth}
		\centering
		\includegraphics[width=\linewidth]{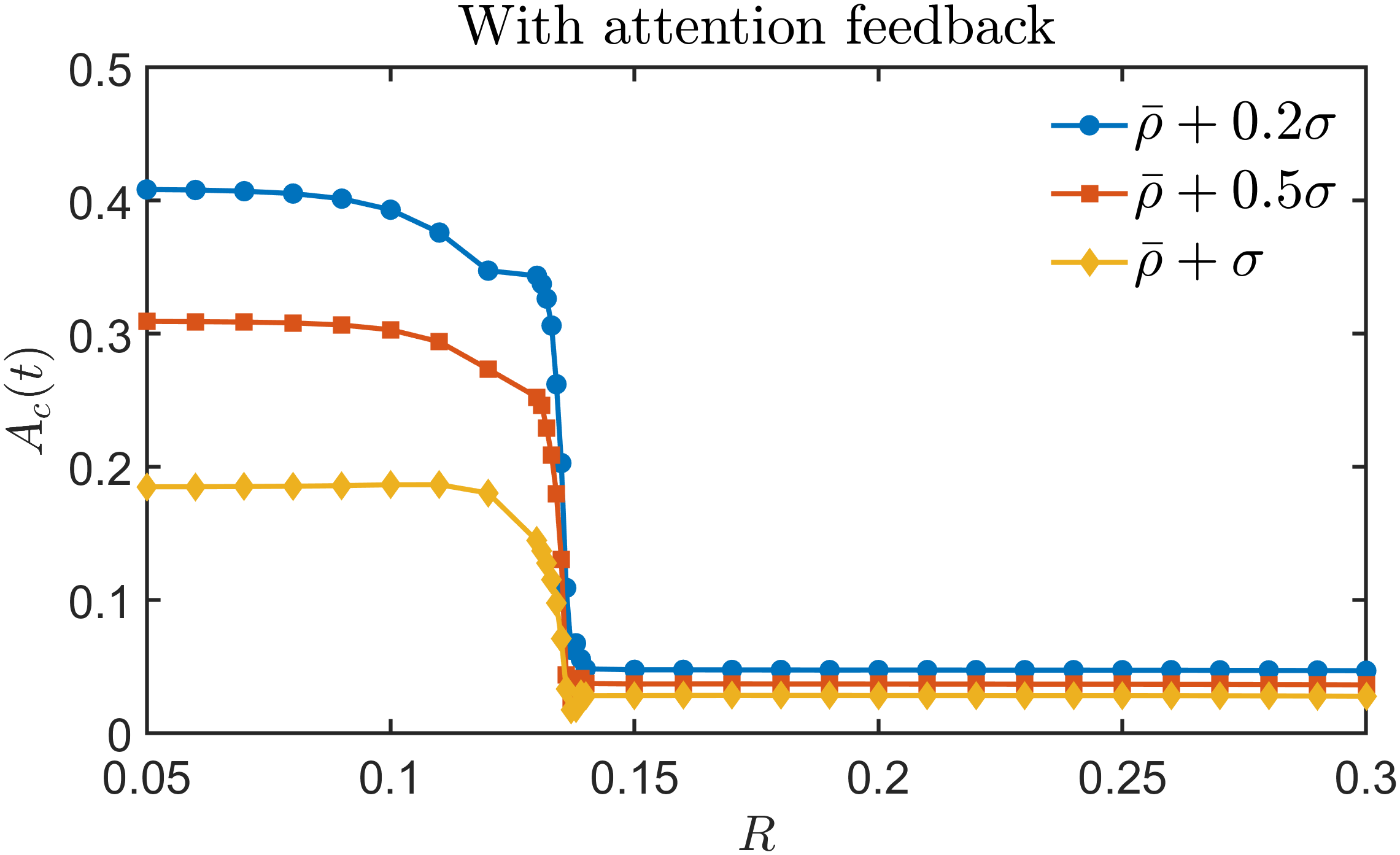}
	\end{subfigure}
	\begin{subfigure}{0.48\textwidth}
		\centering
		\includegraphics[width=\linewidth]{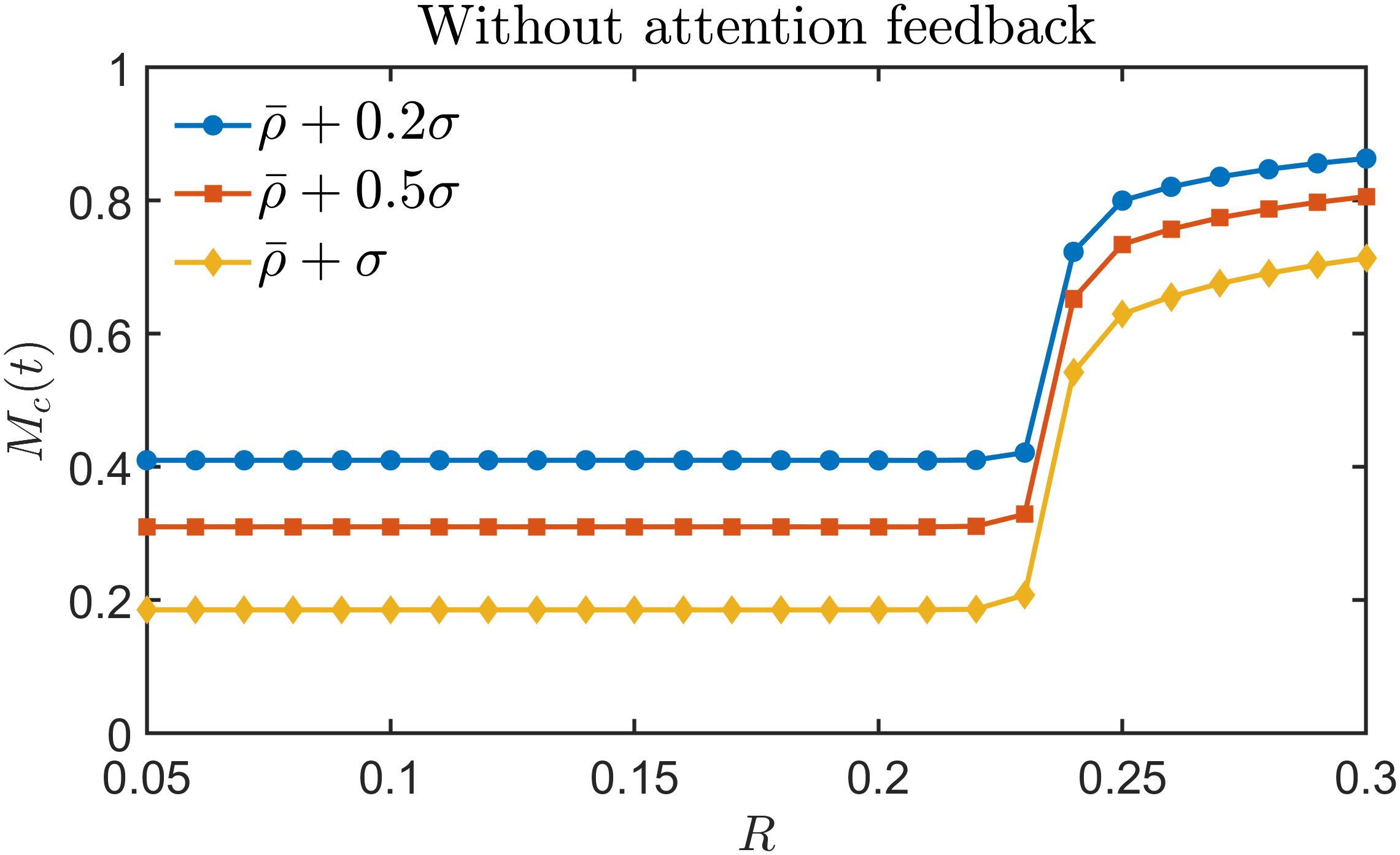}
	\end{subfigure}
	\hfill
	\begin{subfigure}{0.48\textwidth}
		\centering
		\includegraphics[width=\linewidth]{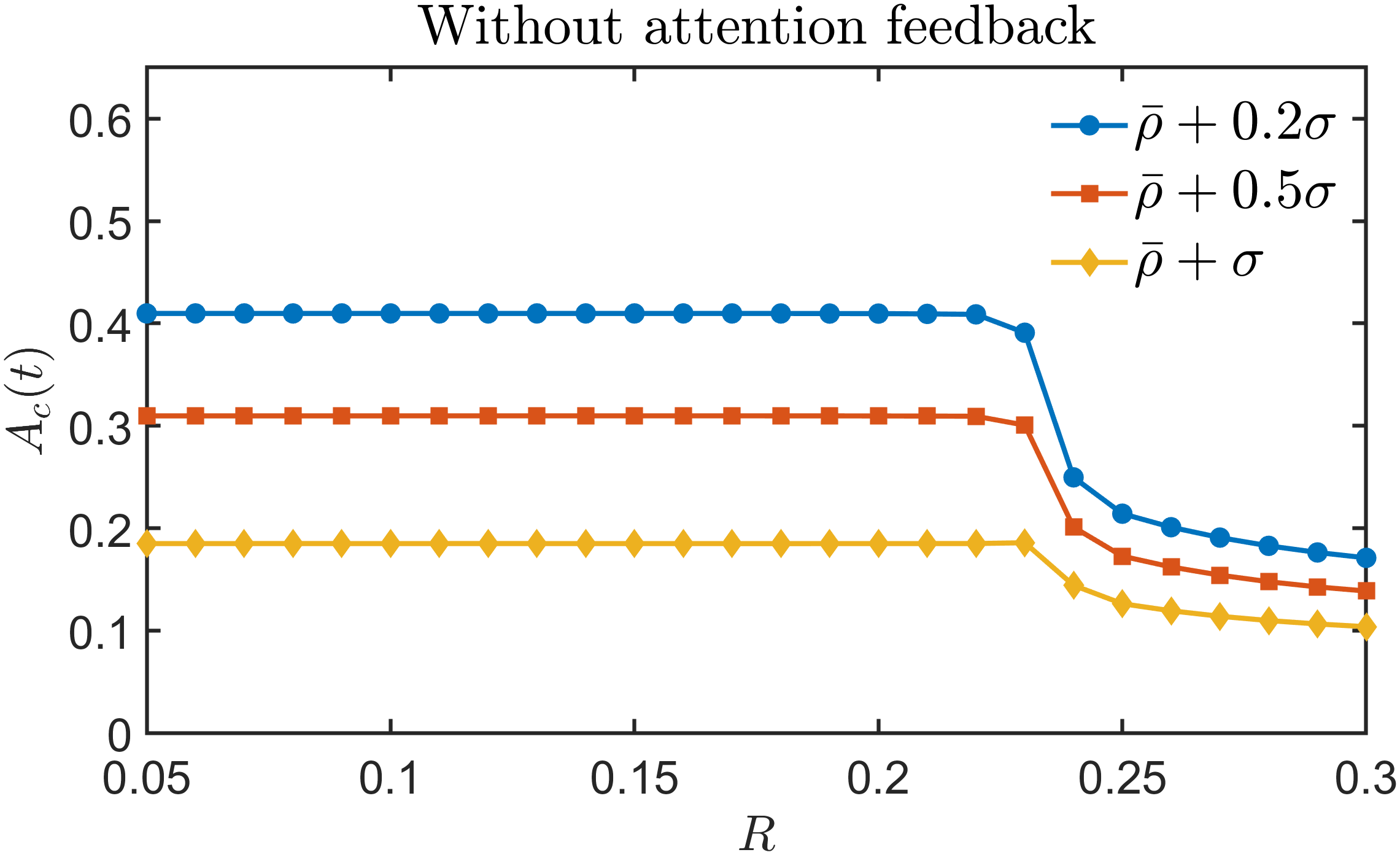}
	\end{subfigure}
	\vspace{-2pt}
	\caption{
		Macroscopic indicators of the high-activity core at $T=100$ as
		functions of the interaction radius $R$ (left: Mass fraction $M_c(T)$, right: Area fraction
		$A_c(T)$). 
		%The left column shows the
		%mass fraction $M_c(T)$, while the right column shows the area fraction
		%$A_c(T)$. 
		The top row corresponds to the full NODAR model with
		attention feedback, and the bottom row corresponds to the reduced
		model without attention feedback. A larger $M_c(T)$ together with a
		smaller $A_c(T)$ indicates that a substantial fraction of opinion
		activity is concentrated in a spatially confined high-activity core.
	}
	\label{fig:macro_indicators_compare}
\end{figure}
These quantities provide a macroscopic signature of non-consensus. In a
consensus-like homogeneous state, the opinion density is broadly
distributed, and no small region carries a disproportionate fraction of
the total mass. By contrast, a large value of $M_c(t)$ together with a
small value of $A_c(t)$ indicates that a substantial share of opinion
activity is concentrated within a relatively small portion of the domain.
This corresponds to the coexistence of a localized high-activity core and
a lower-activity background, and therefore reflects a non-consensus
pattern at the macroscopic level.

Fig.~\ref{fig:macro_indicators_compare} shows how the high-activity is organized. As $R$
increases, the NODAR model exhibits a clear increase in the mass
fraction $M_c(T)$ and a simultaneous decrease in the area fraction
$A_c(T)$, and the sharp changes suggest that 
there exists a phase transition from consensus to non-consensus. 

%This means that opinion activity is no longer broadly
%distributed over the domain, but becomes concentrated in localized
%high-density regions. 
%In terms of the threshold set $E_c(T)$, this behavior indicates that a
%larger proportion of the total opinion mass is captured by a smaller
%spatial region. Thus, the regime with large $M_c(T)$ and small $A_c(T)$
%provides a macroscopic signature of non-consensus: a localized
%high-activity core coexists with a lower-activity background, rather
%than the system approaching a spatially homogeneous state.
%Under the microscopic behavioral interpretation, the high-activity core
%$E_c(T)$ may be viewed as the macroscopic region where adoption is dominant. The remaining
%background $\Omega\setminus E_c(T)$ contains the effects of
%keeping and compromising.

The comparison between the NODAR model and the reduced model without
attention feedback highlights the role of attention-mediated reinforcement.
Although conformity interactions alone can already produce clustering,
feedback through the attention field amplifies the concentration of opinion
activity and makes the high-activity core more spatially confined. This
provides a macroscopic manifestation of the echo chamber effect.

\section{Conclusion}\label{sec:conclusion}

In this work, we propose a nonlocal model of opinion dynamics coupled
with a dynamically evolving attention field. Starting from a
agent-based inflow-outflow mechanism, we derive a nonlocal
advection-cross-diffusion system as the continuum limit, incorporating
two key mechanisms: Social conformity through nonlocal alignment and
attention-mediated feedback associated with echo chamber
reinforcement. These mechanisms give rise to non-consensus states characterized
by persistent spatial heterogeneity.

The linear stability analysis reveals that the interplay between these
mechanisms determines the stability of the homogeneous equilibrium and
the onset of spatial pattern formation. In absence of attention
feedback, clustering arises from the competition between nonlocal
alignment and diffusion. When attention feedback is present, the coupling
between opinion density and attention acts as an amplification mechanism.
It modifies the diffusion threshold and enlarges the parameter regime in
which spatial aggregation can occur. In particular, attention
reinforcement raises the diffusion level required to suppress instability,
thereby promoting the emergence of clustered patterns.
Numerical simulations further show that attention feedback concentrates
opinion activity into localized regions and reduces the spatial extent of
high-activity areas. These results
highlight how the feedback between social conformity and attention
reinforcement can generate persistent non-consensus patterns in spatially
extended information environments, as well as providing a
macroscopic signature of echo chamber reinforcement.

%This reflects a redistribution of collective
%attention toward a few dominant regions of opinion activity, 

%Overall,  Instead of converging to a spatially
%homogeneous macroscopic consensus state, the system may develop multiple
%localized clusters of opinion activity that coexist over time. Strong
%local reinforcement stabilizes such structures even in the presence of
%diffusion, whereas stronger spatial mixing of attention or opinion
%activity can delay or suppress the onset of clustering. 

In future work, we will consider the model parameter inference in the collective opinion dynamics, say, using spatially resolved opinion or activity
data to infer the interaction radius, diffusion
levels, and attention reinforcement strength, from observed aggregation
patterns \cite{ChuLiPorter2024KernelInference,Peralta2025OpinionDynamics}. Such a data-driven approach could help assess whether the
predicted instability thresholds, localized high-activity regions, and the echo chamber effect are consistent with real-world  scenarios.

\section*{Acknowledgment}
This research was supported by the National Natural Science Foundation of China (Nos. 42450275, 12231003, 11871105, 12571413).

\appendix
\section{Proof of Theorem~\ref{thm:basic-properties}}
\label{app:proof-positivity}
This appendix proves the preservation of non-negativity for $\rho$ and
$S$, together with the non-degeneracy of the effective attention field
$A=A^0+S$. We first prove the
estimates in Lemma \ref{lem:coeff-estimates}.

\begin{proof}[Proof of Lemma \ref{lem:coeff-estimates}]	
	Throughout the proof, $C_0, C_1, \ldots$, $C_5$ denote positive constants depending only on $\Omega$, $K_R$, $G_R$, $\varepsilon$, $a_0$, $A^0$, $M$.
	
	Since the kernels in Eq.~\eqref{eq:intro_V} are used periodically on $\mathbb{T}^2$, Young's inequality gives
	\[
	\|\mathcal{K}[u] - \varepsilon\|_{H^3(\Omega)}
	\le \|K_R\|_{L^1(\Omega)}\|u\|_{H^3(\Omega)}, \qquad
	\|\mathcal{N}[u]\|_{H^3(\Omega;\mathbb{R}^2)}
	\le \|G_R\|_{L^1(\Omega;\mathbb{R}^2)}\|u\|_{H^3(\Omega)}.
	\]
	then for $u \ge 0$ and $\|u\|_{H^3} \le M$, it has that
	\[
	\mathcal{K}[u](x) \ge \varepsilon,\quad
	\|\mathcal{N}[u]\|_{H^3(\Omega;\mathbb{R}^2)} + \|\mathcal{K}[u]\|_{H^3(\Omega)} \le C_0 \|u\|_{H^3} \le C_0 M.
	\]
	By the Sobolev composition estimate \cite{BrezisMironescu2001Composition}, we obtain
	\[
	\left\|\frac{1}{\mathcal{K}[u]}\right\|_{H^3(\Omega)} \le C_1.
	\]
	Since $H^3(\Omega)$ is a Banach algebra in two dimensions, it follows that
	\[
	\|V[u]\|_{H^3(\Omega;\mathbb{R}^2)}
	= \left\|\mathcal{N}[u]\,\frac{1}{\mathcal{K}[u]}\right\|_{H^3(\Omega;\mathbb{R}^2)}
	\le C_0 C_1 M.
	\]
	
	We next prove the Lipschitz estimate for $V$. It starts from
	\[
	V[u] - V[v]
	= \frac{\mathcal{N}[u] - \mathcal{N}[v]}{\mathcal{K}[u]}
	+ \mathcal{N}[v]\left(\frac{1}{\mathcal{K}[u]} - \frac{1}{\mathcal{K}[v]}\right).
	\]
	For the first term, Young's inequality and the product estimate in $H^3(\Omega)$ give
	\[
	\left\|\frac{\mathcal{N}[u] - \mathcal{N}[v]}{\mathcal{K}[u]}\right\|_{H^3(\Omega;\mathbb{R}^2)}
	\le C_1 \|G_R\|_{L^1} \|u - v\|_{H^3(\Omega)}.
	\]
	For the second term, since
	\[
	\frac{1}{\mathcal{K}[u]} - \frac{1}{\mathcal{K}[v]}
	= \frac{\mathcal{K}[v] - \mathcal{K}[u]}{\mathcal{K}[u]\,\mathcal{K}[v]},\quad \mathcal{K}[v] - \mathcal{K}[u] = \int_{\Omega}K_R(y-\cdot)(v(y)-u(y))dy.
	\]
	Using Young's inequality, the algebra property of $H^3(\Omega)$, and the lower bounds $\mathcal{K}[u], \mathcal{K}[v] \ge \varepsilon$ yield
	\[
	\left\|\frac{1}{\mathcal{K}[u]} - \frac{1}{\mathcal{K}[v]}\right\|_{H^3(\Omega)}
	\le C_1^2\Vert K_R\Vert_{L^1} \|u - v\|_{H^3(\Omega)}.
	\]
	Together with $\|\mathcal{N}[v]\|_{H^3(\Omega;\mathbb{R}^2)} \le C_0$, we obtain
	\[
	\|V[u] - V[v]\|_{H^3(\Omega;\mathbb{R}^2)} \le C_2\|u - v\|_{H^3(\Omega)},
	\]
	where $C_2 = C_1\|G_R\|_{L^1} + C_0 C_1^2 M \Vert K_R\Vert_{L^1}$.
	
	We now estimate the drift field $\nabla A/A$. Since $A_i = A^0 + S_i$, we have
	\[
	\|A_i\|_{H^4(\Omega)} \le \|A^0\|_{H^4(\Omega)} + \|S_i\|_{H^4(\Omega)} \le C_3.
	\]
	The lower bound $A_i \ge a_0 > 0$ and the Sobolev composition estimate \cite{BrezisMironescu2001Composition} imply
	\[
	\left\|\frac{1}{A_i}\right\|_{H^4(\Omega)} \le C_4.
	\]
	Therefore, by the product estimate in $H^3(\Omega)$,
	\[
	\left\|\frac{\nabla A_i}{A_i}\right\|_{H^3(\Omega)}
	\le \|\nabla A_i\|_{H^3(\Omega)} \left\|\frac{1}{A_i}\right\|_{H^3(\Omega)}
	\le C_3 C_4.
	\]
	
	For the Lipschitz estimate, we use the identity
	\[
	\frac{\nabla A_1}{A_1} - \frac{\nabla A_2}{A_2}
	= \frac{\nabla(A_1 - A_2)}{A_1}
	+ \nabla A_2 \left(\frac{1}{A_1} - \frac{1}{A_2}\right).
	\]
	Since $A_1 - A_2 = S_1 - S_2$, the first term satisfies
	\[
	\left\|\frac{\nabla(A_1 - A_2)}{A_1}\right\|_{H^3(\Omega)}
	\le C_4 \|S_1 - S_2\|_{H^4(\Omega)}.
	\]
	For the second term,
	\[
	\left\|\frac{1}{A_1} - \frac{1}{A_2}\right\|_{H^3(\Omega)}
	= \left\|\frac{S_2 - S_1}{A_1 A_2}\right\|_{H^3(\Omega)}
	\le C_4^2 \|S_1 - S_2\|_{H^4(\Omega)},
	\]
	and together with the product estimate and $\|\nabla A_2\|_{H^3} \le C_3$,
	\[
	\left\|\frac{\nabla A_1}{A_1} - \frac{\nabla A_2}{A_2}\right\|_{H^3(\Omega)}
	\le C_5 \|S_1 - S_2\|_{H^4(\Omega)},
	\]
	where $C_5 = C_4 + C_3 C_4^2$.
	
	Taking $C = \max\{C_0 C_1 M,\, C_2,\, C_3 C_4,\, C_5\}$, we obtain all the estimates stated in Lemma~\ref{lem:coeff-estimates}. This completes the proof.
\end{proof}

Now we prove Theorem~\ref{thm:basic-properties} by a Picard iteration.
The estimates in Lemma~\ref{lem:coeff-estimates} will be used to control
the drift field in the iteration.

\begin{proof}[Proof of Theorem~\ref{thm:basic-properties}]
	Starting from
	$(\rho^{(0)},S^{(0)})=(\rho_0,S_0)$,
	we define the Picard sequence recursively. 
	%Suppose that
	%\((\rho^{(n)},S^{(n)})\in\mathcal B_T\) has been given. 
	Set
	\[
	A^{(n)}:=A^0+S^{(n)},
	\qquad
	B^{(n)}
	:=
	V[\rho^{(n)}]
	+
	\frac{2D_\rho}{A^{(n)}}\nabla A^{(n)} .
	\]
	Since \(\rho^{(n)},S^{(n)}\ge0\), we have
	\[
	A^{(n)}\ge A^0\ge a_0>0,
	\qquad
	\mathcal K[\rho^{(n)}]\ge \varepsilon>0 .
	\]
	By Lemma~\ref{lem:coeff-estimates},
	$\|B^{(n)}\|_{L^\infty(0,T;H^3)}\le \tilde{C}_1 $.
	Since \(H^3(\Omega)\hookrightarrow W^{1,\infty}(\Omega)\) in two
	dimensions, we also have
	\[
	B^{(n)}\in L^\infty(0,T;W^{1,\infty}(\Omega)).
	\]

	We then define \((\rho^{(n+1)},S^{(n+1)})\) as the solution of
	\begin{equation}
		\label{eq:rho-picard}
		\partial_t\rho^{(n+1)}
		=
		D_\rho\Delta\rho^{(n+1)}
		-
		\nabla\cdot\bigl(\rho^{(n+1)}B^{(n)}\bigr),
		\qquad
		\rho^{(n+1)}(\cdot,0)=\rho_0,
	\end{equation}
	and
	\begin{equation}
		\label{eq:S-picard}
		\partial_tS^{(n+1)}
		=
		D_S\Delta S^{(n+1)}
		-\omega S^{(n+1)}
		+\theta\rho^{(n+1)},
		\qquad
		S^{(n+1)}(\cdot,0)=S_0 .
	\end{equation}
	This construction defines a Picard map \(\Phi\)
	\[
	(\rho^{(n+1)},S^{(n+1)})
	=
	\Phi(\rho^{(n)},S^{(n)}).
	\]

	In the following steps, we first prove that the above linear problems preserve
	non-negativity. For \(T>0\), set
	\[
	X_T:=C([0,T];H^3(\Omega))\times C([0,T];H^4(\Omega)) .
	\]
	For a constant \(\widehat M>0\), define
	\[
	\mathcal B_T
	:=
	\left\{
	(\rho,S)\in X_T:
	\rho,S\ge0\ \text{a.e. in }\Omega\times[0,T],
	~~
	\sup_{0\le t\le T}
	\left(
	\|\rho(t)\|_{H^3}^2+\|S(t)\|_{H^4}^2
	\right)
	\le \widehat M
	\right\}.
	\]
	We then derive uniform estimates to show that, after choosing a suitable
	\(\widehat M\) and taking \(T>0\) sufficiently small,
	\(\Phi(\mathcal B_T)\subset \mathcal B_T\). Finally, we prove that \(\Phi\)
	is a contraction on \(\mathcal B_T\) with respect to the norm
	\begin{equation}\label{eq:norm}
		\|(\rho,S)\|_{X_T} := \left[ \sup_{0\le t\le T} \left( \|\rho(t)\|_{H^3}^2 + \|S(t)\|_{H^4}^2 \right) \right]^{1/2}.
	\end{equation}
	%We proceed in the following steps.

	\smallskip
	\noindent\textbf{Step 1:}
	We first prove the preservation of non-negativity.
	Let \((\rho^{(n+1)})^-:=\max\{-\rho^{(n+1)},0\}\). Testing
	\eqref{eq:rho-picard} with \(-(\rho^{(n+1)})^-\), using periodic
	boundary conditions, gives
	\[
	\frac12\frac{d}{dt}\|(\rho^{(n+1)})^-\|_{L^2}^2
	+
	D_\rho\|\nabla(\rho^{(n+1)})^-\|_{L^2}^2
	=
	-\frac12\int_\Omega
	(\nabla\cdot B^{(n)})
	\bigl((\rho^{(n+1)})^-\bigr)^2\,dx .
	\]
	Since \(B^{(n)}\in L^\infty(0,T;W^{1,\infty})\),
	\[
	\frac{d}{dt}\|(\rho^{(n+1)})^-\|_{L^2}^2
	\le
	\|\nabla\cdot B^{(n)}\|_{L^\infty}
	\|(\rho^{(n+1)})^-\|_{L^2}^2 .
	\]
	Because \((\rho^{(n+1)}(\cdot,0))^-=0\), Gronwall's inequality implies
	\((\rho^{(n+1)})^-\equiv0\). Thus
	\[
	\rho^{(n+1)}\ge0
	\quad\text{a.e. in } \Omega\times(0,T).
	\]
	
	Similarly, let \((S^{(n+1)})^-:=\max\{-S^{(n+1)},0\}\). Testing
	\eqref{eq:S-picard} with \(-(S^{(n+1)})^-\) gives
	\[
	\frac12\frac{d}{dt}\|(S^{(n+1)})^-\|_{L^2}^2
	+
	D_S\|\nabla(S^{(n+1)})^-\|_{L^2}^2
	+
	\omega\|(S^{(n+1)})^-\|_{L^2}^2
	=
	-\theta\int_\Omega \rho^{(n+1)}(S^{(n+1)})^-\,dx .
	\]
	Since \(\theta>0\) and \(\rho^{(n+1)}\ge0\), the right-hand side is
	non-positive. Together with \((S^{(n+1)}(\cdot,0))^-=0\), this yields
	\[
	S^{(n+1)}\ge0
	\quad\text{a.e. in } \Omega\times(0,T).
	\]
	Hence the Picard iteration preserves non-negativity.

	\smallskip
	\noindent\textbf{Step 2:} We prove that, for suitable \(\widehat M\) and sufficiently small \(T>0\),
	the Picard map satisfies \(\Phi(\mathcal B_T)\subset\mathcal B_T\).
	By Step 1, \(\rho^{(n+1)},S^{(n+1)}\ge0\). It remains to prove the uniform
	\(H^3\times H^4\) estimate.
	
	Applying the standard \(H^3\)-energy estimate to
	\eqref{eq:rho-picard}, the drift term is controlled by
	Young's inequality:
	\begin{equation*}
		\left|
		\langle
		\nabla\cdot(\rho^{(n+1)}B^{(n)}),
		\rho^{(n+1)}
		\rangle_{H^3}
		\right|\le
		\frac{D_\rho}{2}\|\nabla\rho^{(n+1)}\|_{H^3}^2
		+
		{C}_1\|\rho^{(n+1)}\|_{H^3}^2 ,
	\end{equation*} 
	where \(C_1>0\) is independent of \(T\) and \(n\).
	Then using $B^{(n)}\in L^\infty(0,T;W^{1,\infty})$, we obtain
	\begin{equation}\label{eq:rho_uniform_est}
		\frac{d}{dt}\|\rho^{(n+1)}\|_{H^3}^2
		+
		D_\rho\|\nabla\rho^{(n+1)}\|_{H^3}^2
		\le
		{C}_1\|\rho^{(n+1)}\|_{H^3}^2 .
	\end{equation}

	For \(S^{(n+1)}\), applying the \(H^4\)-energy estimate to
	\eqref{eq:S-picard} gives
	\begin{equation*}
		\frac{d}{dt}\|S^{(n+1)}\|_{H^4}^2
		+
		D_S\|\nabla S^{(n+1)}\|_{H^4}^2
		\le
		{C}_2\|S^{(n+1)}\|_{H^4}^2
		+
		{C}_3\|\rho^{(n+1)}\|_{H^4}^2 .
	\end{equation*}
	Using $\|\rho^{(n+1)}\|_{H^4}^2
	\le
	\|\rho^{(n+1)}\|_{H^3}^2
	+
	\|\nabla \rho^{(n+1)}\|_{H^3}^2
	$,
	we obtain
	\begin{equation}\label{eq:S_uniform_est}
		\begin{aligned}
			\frac{d}{dt}\|S^{(n+1)}\|_{H^4}^2
			+
			D_S\|\nabla S^{(n+1)}\|_{H^4}^2
			&\le
			{C}_2\|S^{(n+1)}\|_{H^4}^2
			+
			{C}_3\|\rho^{(n+1)}\|_{H^3}^2\\
			&+
			{C}_3\|\nabla \rho^{(n+1)}\|_{H^3}^2 .
		\end{aligned}
	\end{equation}
	where \(C_2,C_3>0\) are independent of \(T\) and \(n\).
	
	We can choose sufficiently large \(\lambda>0\)  such that
	$\lambda D_\rho \ge C_3+1 $.
	Multiplying \eqref{eq:rho_uniform_est} by \(\lambda\) and adding \eqref{eq:S_uniform_est}, we obtain
	\begin{equation*}
		\begin{aligned}
			\frac{d}{dt}
			\left(
			\lambda\|\rho^{(n+1)}\|_{H^3}^2
			+
			\|S^{(n+1)}\|_{H^4}^2
			\right)
			&+ c_0
			\left(
			\|\nabla \rho^{(n+1)}\|_{H^3}^2
			+
			\|\nabla S^{(n+1)}\|_{H^4}^2
			\right)
			\\
			&\le
			\tilde{C}
			\left(
			\lambda\|\rho^{(n+1)}\|_{H^3}^2
			+
			\|S^{(n+1)}\|_{H^4}^2
			\right).
		\end{aligned}
	\end{equation*}
	where
	$ \tilde{C}
	=
	\max\left\{
	C_1+\frac{C_3}{\lambda},
	C_2
	\right\}$, $c_0= \min\{1,D_S\}$. Set $E_0:=\|\rho_0\|_{H^3}^2+\|S_0\|_{H^4}^2 $.
	By Gronwall's inequality, there exists $C_\ast$ such that 
	\begin{equation*}
		\begin{aligned}
			&\sup_{0\le t\le T}
			\left(
			\|\rho^{(n+1)}(t)\|_{H^3}^2
			+
			\|S^{(n+1)}(t)\|_{H^4}^2
			\right)
			\\
			&+
			\int_0^T
			\left(
			\|\nabla \rho^{(n+1)}\|_{H^3}^2
			+
			\|\nabla S^{(n+1)}\|_{H^4}^2
			\right)\,dt
			\le
			C_*E_0e^{\tilde{C} T},
		\end{aligned}
	\end{equation*}
	where \(C_*\) is independent of \(n\).
	
	It suffices to choose \(\widehat M>0\) such that
	$\widehat M\ge \max\{E_0,2C_*E_0\}$, and  \(T>0\) sufficiently small so that \(e^{\widetilde C T}\le2\).
	Then
	\[
	\sup_{0\le t\le T}
	\left(
	\|\rho^{(n+1)}(t)\|_{H^3}^2
	+
	\|S^{(n+1)}(t)\|_{H^4}^2
	\right)
	\le \widehat M.
	\]
	Thus \((\rho^{(n+1)},S^{(n+1)})\in\mathcal B_T\), and hence
	\(\Phi(\mathcal B_T)\subset\mathcal B_T\). Moreover, each pair $(\rho^{(n)}, S^{(n)})$ satisfies
	\[
	\|\rho^{(n)}\|_{L^\infty(0,T;H^3)}
	+
	\|\nabla\rho^{(n)}\|_{L^2(0,T;H^3)}
	+
	\|S^{(n)}\|_{L^\infty(0,T;H^4)}
	+
	\|\nabla S^{(n)}\|_{L^2(0,T;H^4)}
	\le C .
	\]
	
	\smallskip
	\noindent\textbf{Step 3:}
	We prove that \(\Phi\) is a contraction in \(\mathcal B_T\). 
	%Take two
	%arbitrary inputs in \(\mathcal B_T\), denoted by
	%$(\rho_i^{(n)},S_i^{(n)})\in\mathcal B_T$, $i=1,2$,
	%and let
	%\[
	%(\rho_i^{(n+1)},S_i^{(n+1)})
	%=
	%\Phi(\rho_i^{(n)},S_i^{(n)}),
	%\qquad i=1,2.
	%\]
	Define
	\begin{equation*}
		\delta\rho^{(n+1)}:=\rho^{(n+1)}-\rho^{(n)},\quad
		\delta S^{(n+1)}:=S^{(n+1)}-S^{(n)}.
	\end{equation*}
	Subtracting the corresponding equations gives
	\[
	\partial_t\delta\rho^{(n+1)}
	=
	D_\rho\Delta\delta\rho^{(n+1)}
	-
	\nabla\cdot\bigl(\delta\rho^{(n+1)}B^{(n)}\bigr)
	-
	\nabla\cdot
	\left(
	\rho^{(n)}(B^{(n)}-B^{(n-1)})
	\right),
	\]
	and
	\[
	\partial_t\delta S^{(n+1)}
	=
	D_S\Delta\delta S^{(n+1)}
	-\omega\delta S^{(n+1)}
	+\theta\delta\rho^{(n+1)},
	\]
	with zero initial data. By Lemma~\ref{lem:coeff-estimates} and the uniform
	bounds in \(\mathcal B_T\),
	\[
	\|B^{(n)}-B^{(n-1)}\|_{H^3}
	\le
	C_4
	\left(
	\|\delta\rho^{(n)}\|_{H^3}
	+
	\|\delta S^{(n)}\|_{H^4}
	\right),
	\]
	where \(C_4>0\) is independent of \(T\) and \(n\).
	
	By applying same \(H^3\)-energy estimate as in \eqref{eq:rho_uniform_est}
	to \(\delta\rho^{(n+1)}\), it yields 
	\begin{equation}\label{eq:delta_rho}
		\begin{aligned}
			\frac{d}{dt}
			\|\delta\rho^{(n+1)}\|_{H^3}^2
			+
			D_\rho
			\|\nabla\delta\rho^{(n+1)}\|_{H^3}^2
			&\le
			C_1\|\delta\rho^{(n+1)}\|_{H^3}^2
			\\
			&\quad+
			C_4
			\left(
			\|\delta\rho^{(n)}\|_{H^3}^2
			+
			\|\delta S^{(n)}\|_{H^4}^2
			\right).
		\end{aligned}
	\end{equation}

	Similarly, applying the $H^4$ -energy estimate leading to \eqref{eq:S_uniform_est} to
	\(\delta S^{(n+1)}\) gives
	\[
	\frac{d}{dt}
	\|\delta S^{(n+1)}\|_{H^4}^2
	+
	D_S\|\nabla\delta S^{(n+1)}\|_{H^4}^2
	\le
	C_2\|\delta S^{(n+1)}\|_{H^4}^2
	+
	C_3\|\delta\rho^{(n+1)}\|_{H^4}^2 .
	\]
	Using
	$\|\delta\rho^{(n+1)}\|_{H^4}^2
	\le
	\|\delta\rho^{(n+1)}\|_{H^3}^2
	+
	\|\nabla\delta\rho^{(n+1)}\|_{H^3}^2$,
	we obtain
	\begin{equation}\label{eq:delta_S}
		\begin{aligned}
			\frac{d}{dt}
			\|\delta S^{(n+1)}\|_{H^4}^2
			+
			D_S\|\nabla\delta S^{(n+1)}\|_{H^4}^2
			&\le
			C_2\|\delta S^{(n+1)}\|_{H^4}^2
			+
			C_3\|\delta\rho^{(n+1)}\|_{H^3}^2
			\\
			&\quad+
			C_3\|\nabla\delta\rho^{(n+1)}\|_{H^3}^2 .
		\end{aligned}
	\end{equation}
	Choose \(\Lambda\ge1\) sufficiently large such that
	\(\Lambda D_\rho\ge C_3+1\). Multiplying the estimate for
	\eqref{eq:delta_rho} by \(\Lambda\) and adding \eqref{eq:delta_S}, we get
	\[
	\begin{aligned}
		\frac{d}{dt}
		\left(
		\Lambda \|\delta\rho^{(n+1)}\|_{H^3}^2
		+
		\|\delta S^{(n+1)}\|_{H^4}^2
		\right)
		&\le
		C_5
		\left(
		\Lambda \| \delta\rho^{(n+1)}\|_{H^3}^2
		+
		\|\delta S^{(n+1)}\|_{H^4}^2
		\right)
		\\
		&\quad+
		\Lambda C_4
		\left(
		\|\delta\rho^{(n)}\|_{H^3}^2
		+
		\|\delta S^{(n)}\|_{H^4}^2
		\right),
	\end{aligned}
	\]
	where $C_5:=\max\left\{ C_1+\frac{C_3}{\Lambda},\,C_2\right\}$ is independent of \(T\) and \(n\).

	Since
	\(\delta\rho^{(n+1)}(\cdot,0)=0\) and
	\(\delta S^{(n+1)}(\cdot,0)=0\), Gronwall's inequality gives
	\[
	\begin{aligned}
		&\sup_{0\le t\le T}
		\left(
		\|\delta\rho^{(n+1)}(t)\|_{H^3}^2
		+
		\|\delta S^{(n+1)}(t)\|_{H^4}^2
		\right)
		\\
		&\qquad\le
		\Lambda C_4Te^{C_5T}
		\sup_{0\le t\le T}
		\left(
		\|\delta\rho^{(n)}(t)\|_{H^3}^2
		+
		\|\delta S^{(n)}(t)\|_{H^4}^2
		\right).
	\end{aligned}
	\]
	It suffices to choose small  \(T\) such that
	\((\Lambda C_4Te^{C_5T})^{1/2}<1\). Hence \(\Phi\) is a contraction on \(\mathcal B_T\)
	with respect to the norm \eqref{eq:norm}.

	\smallskip
	\noindent\textbf{Step 4:}
	By Steps 2 and 3, for the above choice of \(\widehat M\) and sufficiently
	small \(T>0\), we have
	$\Phi(\mathcal B_T)\subset \mathcal B_T$,
	and \(\Phi\) is a contraction on \(\mathcal B_T\) with respect to the norm
	\eqref{eq:norm}. Since
	\(\mathcal B_T\) is a closed subset of the Banach space \(X_T\), the Banach
	fixed point theorem yields a unique fixed point \((\rho,S)\in\mathcal B_T\).
	That is, \((\rho,S)=\Phi(\rho,S)\). Therefore \((\rho,S)\) solves the
	original nonlinear system on \([0,T]\) with initial data
	\(\rho(\cdot,0)=\rho_0\) and \(S(\cdot,0)=S_0\).
	
	Since \((\rho,S)\in\mathcal B_T\), we have
	\[
	\rho\in C([0,T];H^3(\Omega)),
	\quad
	S\in C([0,T];H^4(\Omega)).
	\]
	Moreover, applying the uniform estimate in Step 2 to the fixed point yields
	\[
	\rho\in L^2(0,T;H^4(\Omega)),
	\quad
	S\in L^2(0,T;H^5(\Omega)).
	\]
	Hence
	\[
	\rho\in C([0,T];H^3(\Omega))\cap L^2(0,T;H^4(\Omega)),
	\quad
	S\in C([0,T];H^4(\Omega))\cap L^2(0,T;H^5(\Omega)).
	\]
	Moreover, the positivity-preserving property in Step 1 is inherited by the
	fixed point, so that
	\[
	\rho\ge0,\qquad S\ge0,
	\qquad
	A:=A^0+S\ge A^0\ge a_0>0
	\quad\text{a.e. in } \Omega\times(0,T).
	\]
	
	Finally, integrating equation \eqref{eq:HK-S-scaled-rho} over \(\Omega\) yields the mass conservation.
	
\end{proof}

\section{Numerical validation}\label{app:verification}
This section provides a numerical verification of the proposed scheme. 
We first verify the accuracy through convergence tests, and then demonstrate its ability to preserve structural properties, including non-negativity and mass conservation.

\subsection{Convergence test}

To facilitate a rigorous error analysis, we define the relative error in the $L^p$-norm ($p=2, \infty$) for a generic variable $u$ as:
\begin{equation*}
	\|e_u\|_{L^p} = \frac{\|u_{\text{comp}}(T) - u_{\text{ref}}(T)\|_{L^p}}{\|u_{\text{ref}}(T)\|_{L^p}},
\end{equation*}
where $u_{\text{comp}}$ and $u_{\text{ref}}$ denote the computed and reference solutions at final time $T$, respectively. 
We consider the following smooth periodic initial conditions:
\begin{equation}\label{init_rho_S}
	\rho_0(x,y)  = S_0(x,y) = 0.2 + 0.1\cos(2\pi x)\cos(2\pi y) + 0.05\sin(4\pi x)\sin(2\pi y), 
\end{equation}
for $(x,y) \in [0,1]^2$. The parameters are chosen as
$R = 0.2$, $A^0=1/30$, $\varepsilon = 10^{-4}$,
$D_\rho = 0.1$, $D_S = 0.1$, $\omega = 0.1$, $\theta = 0.1$, and $I=0.1$.

To assess spatial accuracy, we perform simulations up to $T=0.2$ with a fixed time step $\Delta t=10^{-5}$.
Reference solutions are computed on the finest grid $N_f=256$.
As shown in Table~\ref{space_order_ACCD_rho}, 
the errors for $\rho$ and $S$ decay rapidly with grid refinement and reach $\mathcal{O}(10^{-13})$ already at $N=32$, 
indicating spectral convergence.

To assess temporal accuracy, we refine the time step $\Delta t$ while fixing the spatial resolution at $N=128$.
The reference solution is computed with $\Delta t_{\mathrm{ref}}=10^{-5}$.
As shown in Table~\ref{time_order_ACCD_rho}, both $\rho$ and $S$ exhibit third-order convergence.
%, being consistent with the order of the IMEX-ARS(2,3,3) scheme.

\begin{table}[!h]
	\centering
	\caption{Spatial convergence of $\rho$ and $S$, with initial data \eqref{init_rho_S}.}
	\label{space_order_ACCD_rho}
	\renewcommand{\arraystretch}{1.2}
	\begin{tabular}{c|c|c|c}
		\hline\hline
		$N$ &  $\|e_\rho\|_{L^2}$ &  $\|e_\rho\|_{L^\infty}$ & Convergence Plot \\ \hline
		4  & $2.323 \times 10^{-2}$ & $ 3.121 \times 10^{-2}$ & \multirow{5}{*}{\includegraphics[scale=0.25]{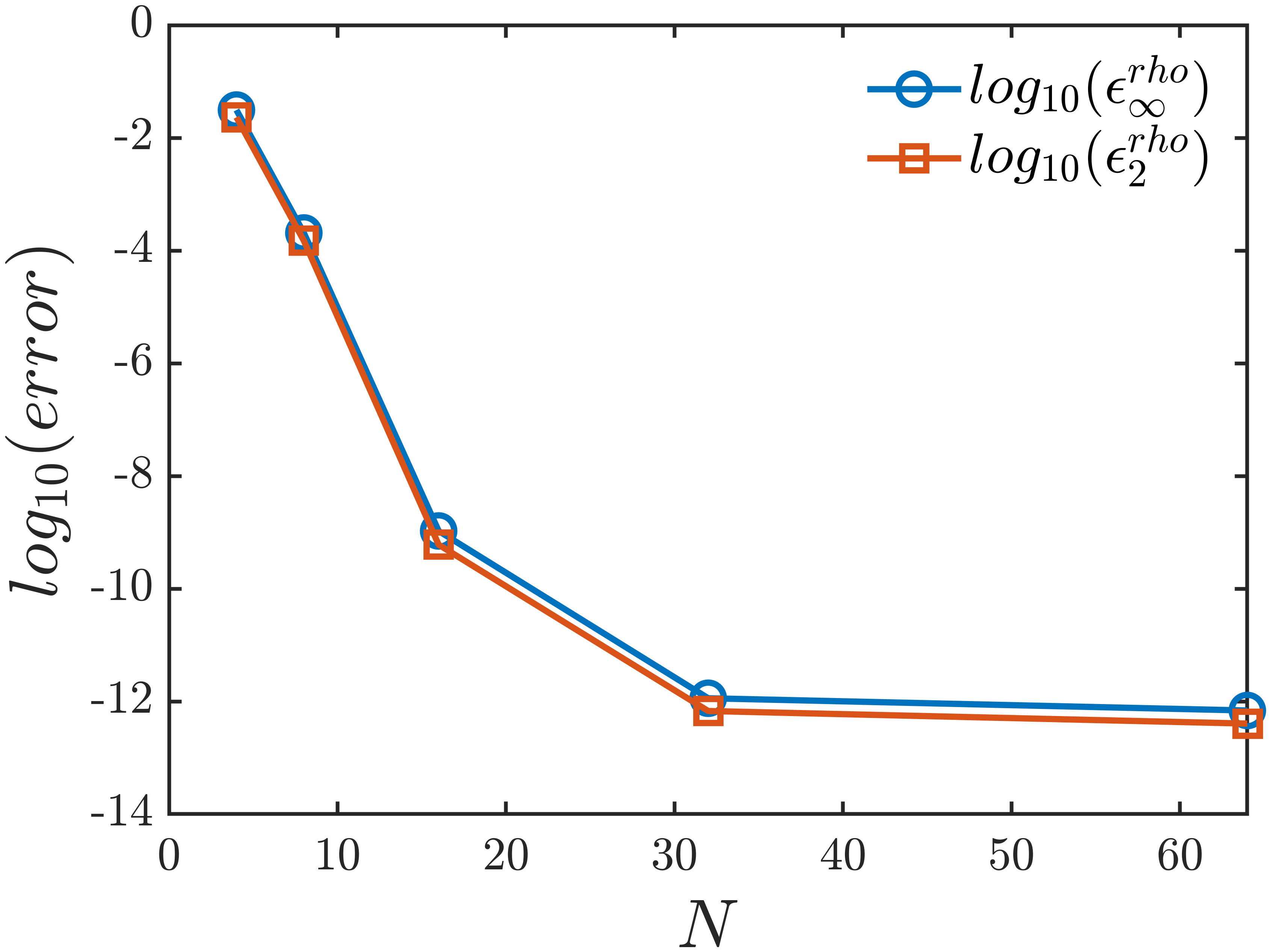}} \\
		8  & $1.499 \times 10^{-4}$ & $ 2.056\times 10^{-4}$ & \\
		16 & $6.041 \times 10^{-10}$ & $1.059 \times 10^{-9}$ & \\
		32 & $6.757 \times 10^{-13}$ & $1.138 \times 10^{-12}$ & \\
		64 & $4.041 \times 10^{-13}$ & $6.897 \times 10^{-13}$ & \\ 
		\hline
		$N$ &  $\|e_S\|_{L^2}$ & $\|e_S\|_{L^\infty}$ & Convergence Plot \\ \hline
		4  & $6.444 \times 10^{-4}$ & $ 1.189 \times 10^{-3}$ & \multirow{5}{*}{\includegraphics[scale=0.25]{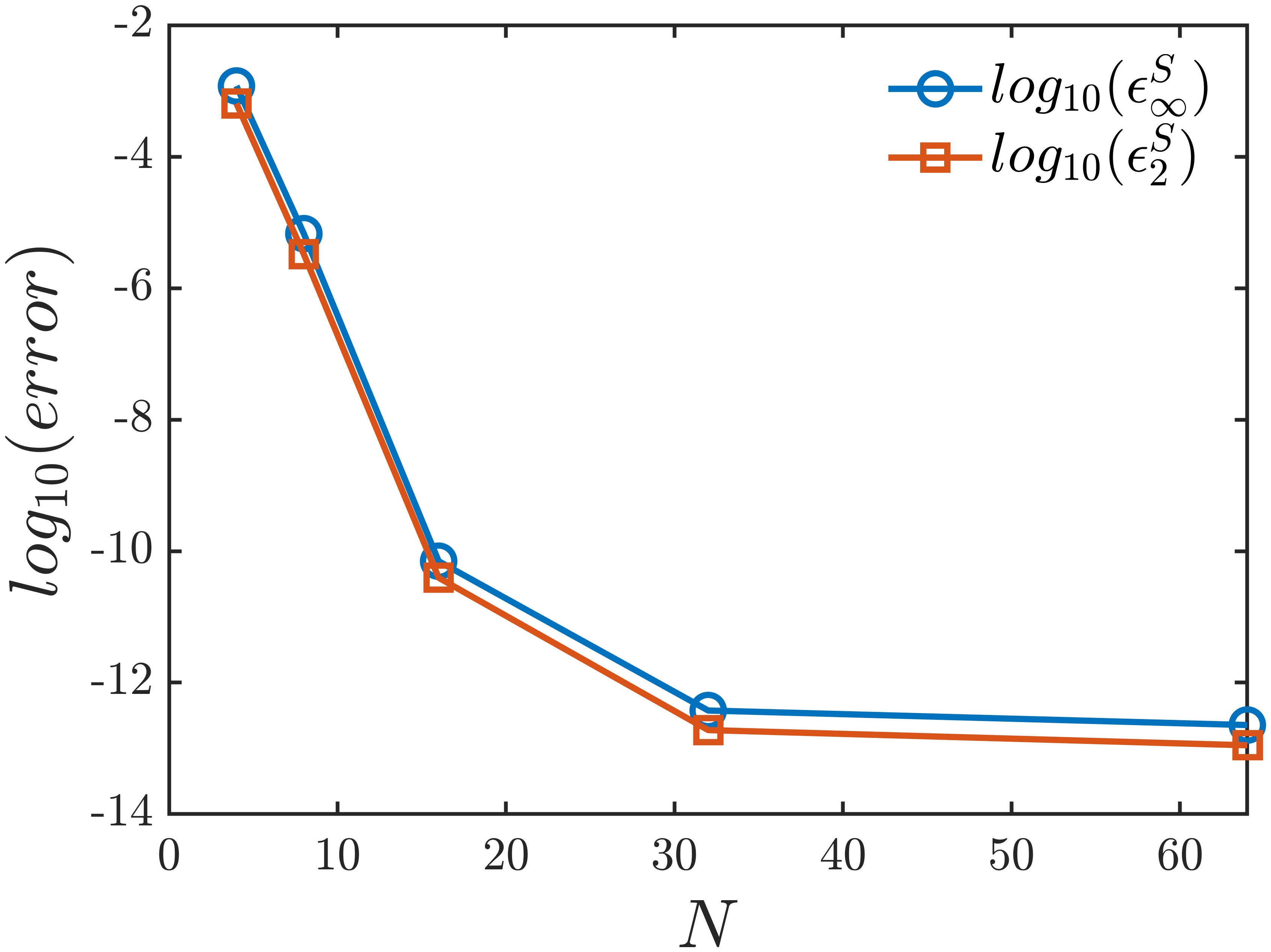}} \\
		8  & $3.275 \times 10^{-6}$ & $ 6.741\times 10^{-6}$ & \\
		16 & $3.942 \times 10^{-11}$ & $ 7.029 \times 10^{-11}$ & \\
		32 & $1.897 \times 10^{-13}$ & $ 3.766 \times 10^{-13}$ & \\
		64 & $1.114 \times 10^{-13}$ & $ 2.264 \times 10^{-13}$ & \\ \hline\hline
	\end{tabular}
\end{table}

%\begin{table}[!h]
%	\centering
%	\caption{Spatial convergence of $S$ in NODAR system, with initial data \eqref{init_rho_S}.}
%	\label{space_order_ACCD_S}
%	\renewcommand{\arraystretch}{1.2}
%	\begin{tabular}{c|c|c|c}
	%		\hline\hline
	%		$N$ &  $\|e_S\|_{L^2}$ & $\|e_S\|_{L^\infty}$ & Convergence Plot \\ \hline
	%		4  & $6.444 \times 10^{-4}$ & $ 1.189 \times 10^{-3}$ & \multirow{5}{*}{\includegraphics[scale=0.25]{figures/HKS_S_space_order.png}} \\
	%		8  & $3.275 \times 10^{-6}$ & $ 6.741\times 10^{-6}$ & \\
	%		16 & $3.942 \times 10^{-11}$ & $ 7.029 \times 10^{-11}$ & \\
	%		32 & $1.897 \times 10^{-13}$ & $ 3.766 \times 10^{-13}$ & \\
	%		64 & $1.114 \times 10^{-13}$ & $ 2.264 \times 10^{-13}$ & \\ \hline\hline
	%	\end{tabular}
%\end{table}

\begin{table}[h!]
	\centering
	\caption{Temporal convergence of $\rho$ and $S$, with initial data \eqref{init_rho_S}.}
	\label{time_order_ACCD_rho}
	\renewcommand{\arraystretch}{1.25}
	\begin{tabular}{c|c|c|c|c|c}
		\hline\hline
		$\Delta t$ & $\|e_\rho\|_{L^2}$ & Order & $\|e_\rho\|_{L^\infty}$ & Order & Convergence Plot \\ \hline
		$1 \times 10^{-3}$ & $5.741 \times 10^{-8}$ & - & $9.470 \times 10^{-8}$ & - & \multirow{4.3}{*}{\includegraphics[scale=0.2]{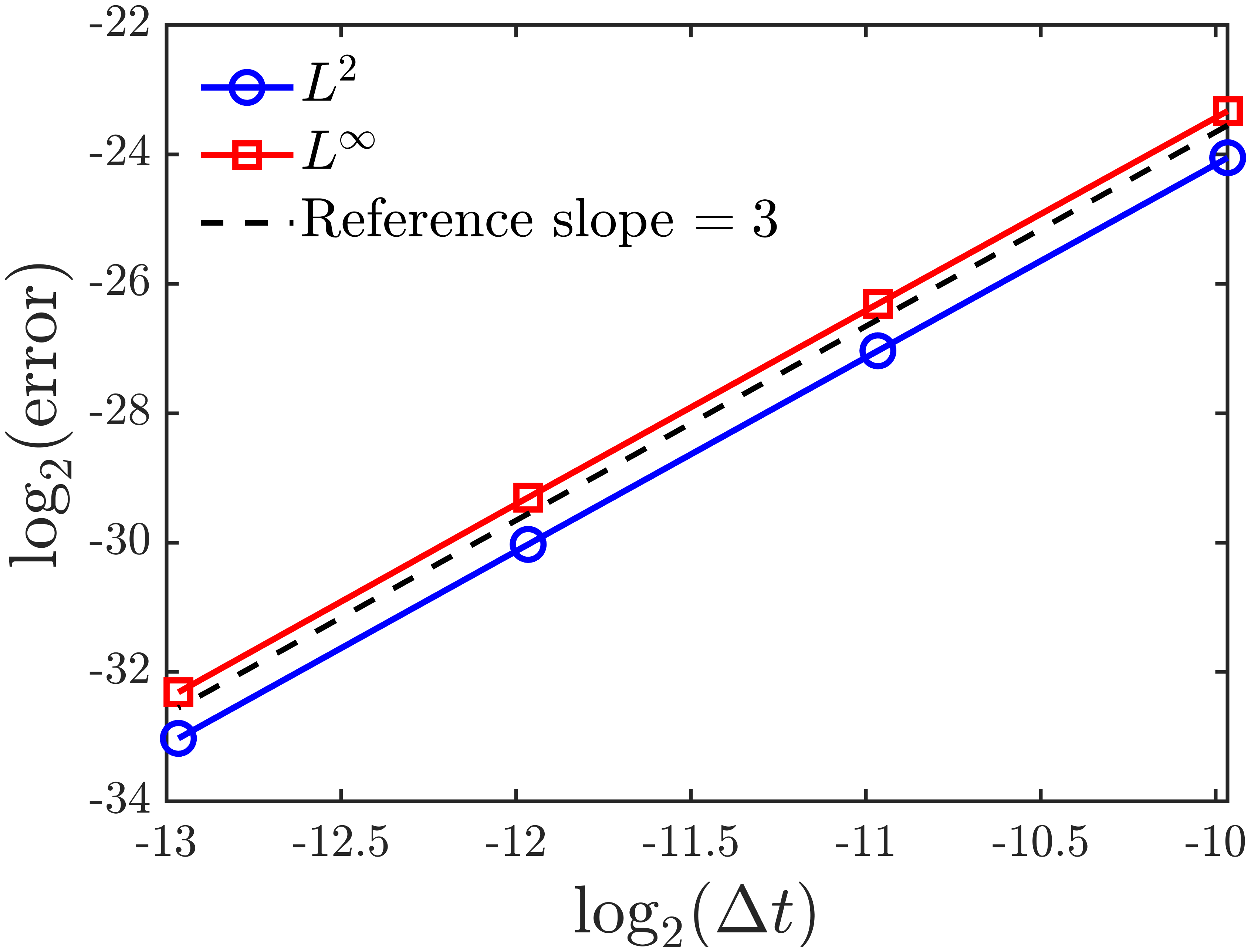}} \\
		$5 \times 10^{-4}$ & $ 7.262 \times 10^{-9}$ & 2.983 & $1.201 \times 10^{-8}$ & 2.980 & \\
		$2.5 \times 10^{-4}$ & $9.131 \times 10^{-10}$ & 2.992 & $1.150 \times 10^{-9}$ & 2.991 & \\
		$1.25 \times 10^{-4}$ & $ 1.144 \times 10^{-10}$ & 2.997 & $1.874 \times 10^{-10}$ & 3.010 & \\ 
		\hline
		$\Delta t$ & $\|e_S\|_{L^2}$ & Order & $\|e_S\|_{L^\infty}$ & Order & Convergence Plot \\ \hline
		$1 \times 10^{-3}$ & $ 8.536 \times 10^{-9}$ & - & $1.524 \times 10^{-8}$ & - & \multirow{4.3}{*}{\includegraphics[scale=0.2]{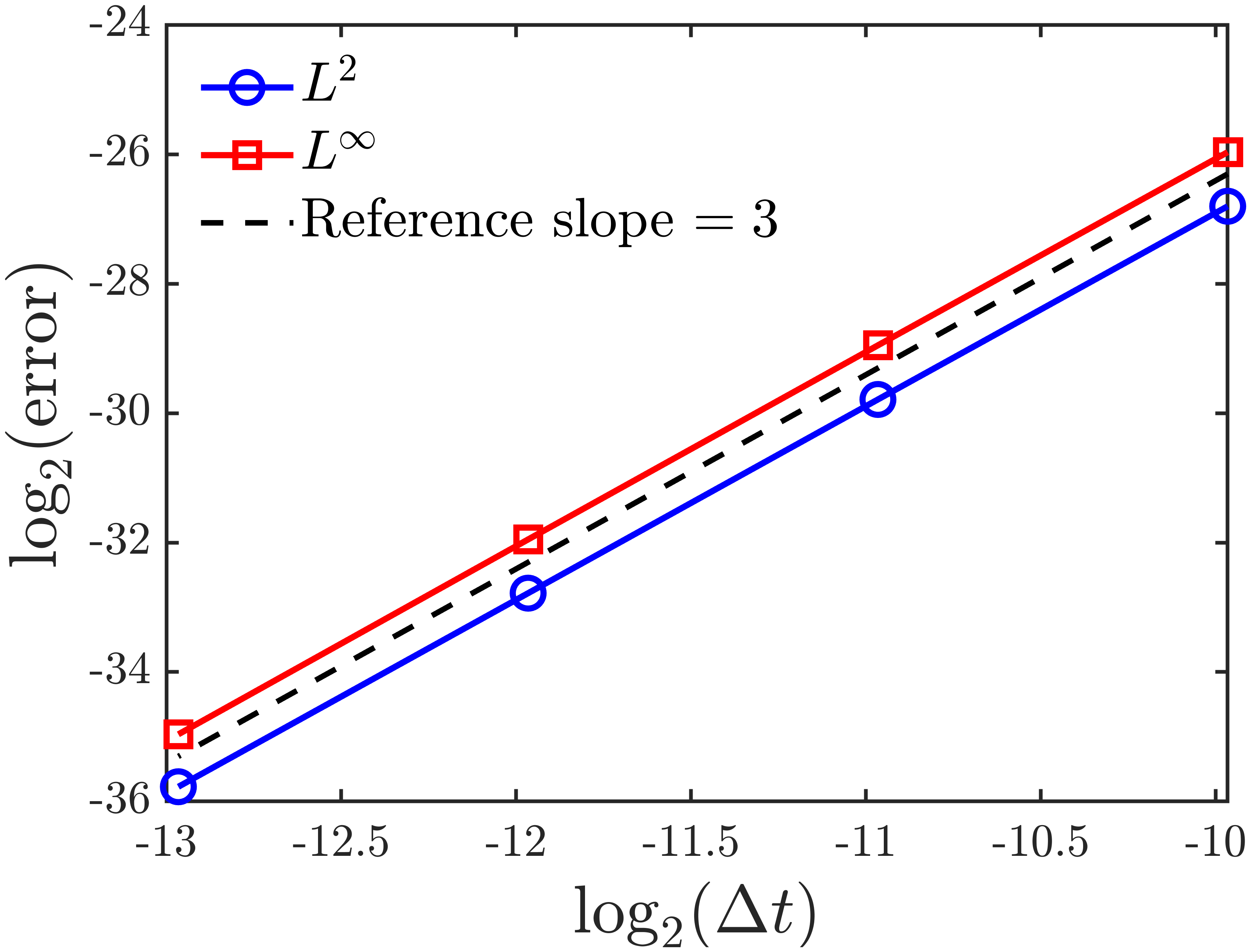}} \\
		$5 \times 10^{-4}$ & $1.078 \times 10^{-10}$ & 2.985 & $1.925 \times 10^{-9}$ & 2.985 & \\
		$2.5 \times 10^{-4}$ & $1.354 \times 10^{-10}$ & 2.993 & $2.408 \times 10^{-10}$ & 2.999 & \\
		$1.25 \times 10^{-4}$ & $1.703 \times 10^{-11}$ & 2.991 & $2.987 \times 10^{-11}$ & 3.011 & \\ \hline\hline
	\end{tabular}
\end{table}

%\begin{table}[h!]
%	\centering
%	\caption{Temporal convergence of $S$ in NODAR system, with initial data \eqref{init_rho_S}.}
%	\label{time_order_ACCD_S}
%	\renewcommand{\arraystretch}{1.25}
%	\begin{tabular}{c|c|c|c|c|c}
	%		\hline\hline
	%		$\Delta t$ & $\|e_S\|_{L^2}$ & Order & $\|e_S\|_{L^\infty}$ & Order & Convergence Plot \\ \hline
	%		$1 \times 10^{-3}$ & $ 8.536 \times 10^{-9}$ & - & $1.524 \times 10^{-8}$ & - & \multirow{4.3}{*}{\includegraphics[scale=0.2]{figures/HKS_S_time_order.png}} \\
	%		$5 \times 10^{-4}$ & $1.078 \times 10^{-10}$ & 2.985 & $1.925 \times 10^{-9}$ & 2.985 & \\
	%		$2.5 \times 10^{-4}$ & $1.354 \times 10^{-10}$ & 2.993 & $2.408 \times 10^{-10}$ & 2.999 & \\
	%		$1.25 \times 10^{-4}$ & $1.703 \times 10^{-11}$ & 2.991 & $2.987 \times 10^{-11}$ & 3.011 & \\ \hline\hline
	%	\end{tabular}
%\end{table}

\subsection{Structure-preserving properties}

We demonstrate that the proposed scheme preserves the key structural properties of the NODAR system, namely non-negativity and mass conservation. 
At each time step, a Lagrange multiplier projection is applied to enforce $\rho \ge 0$, $S \ge 0$, and the conservation of the total mass of $\rho$.

The initial data are taken as small random perturbations of the homogeneous equilibrium \eqref{eq:equilibrium},
\[
\rho(x,y,0)=\rho_0+\delta(x,y), \quad
S(x,y,0)=S_0+\delta(x,y),
\]
where $\delta(x,y)$ is constructed as a superposition of $30$ Gaussian functions with random centres, amplitudes, and widths. 

\begin{figure}[h!]
	\centering
	\vspace{-5pt}
	
	% ================= Row 1 =================
	\begin{subfigure}{0.28\textwidth}
		\centering
		\includegraphics[width=\linewidth]{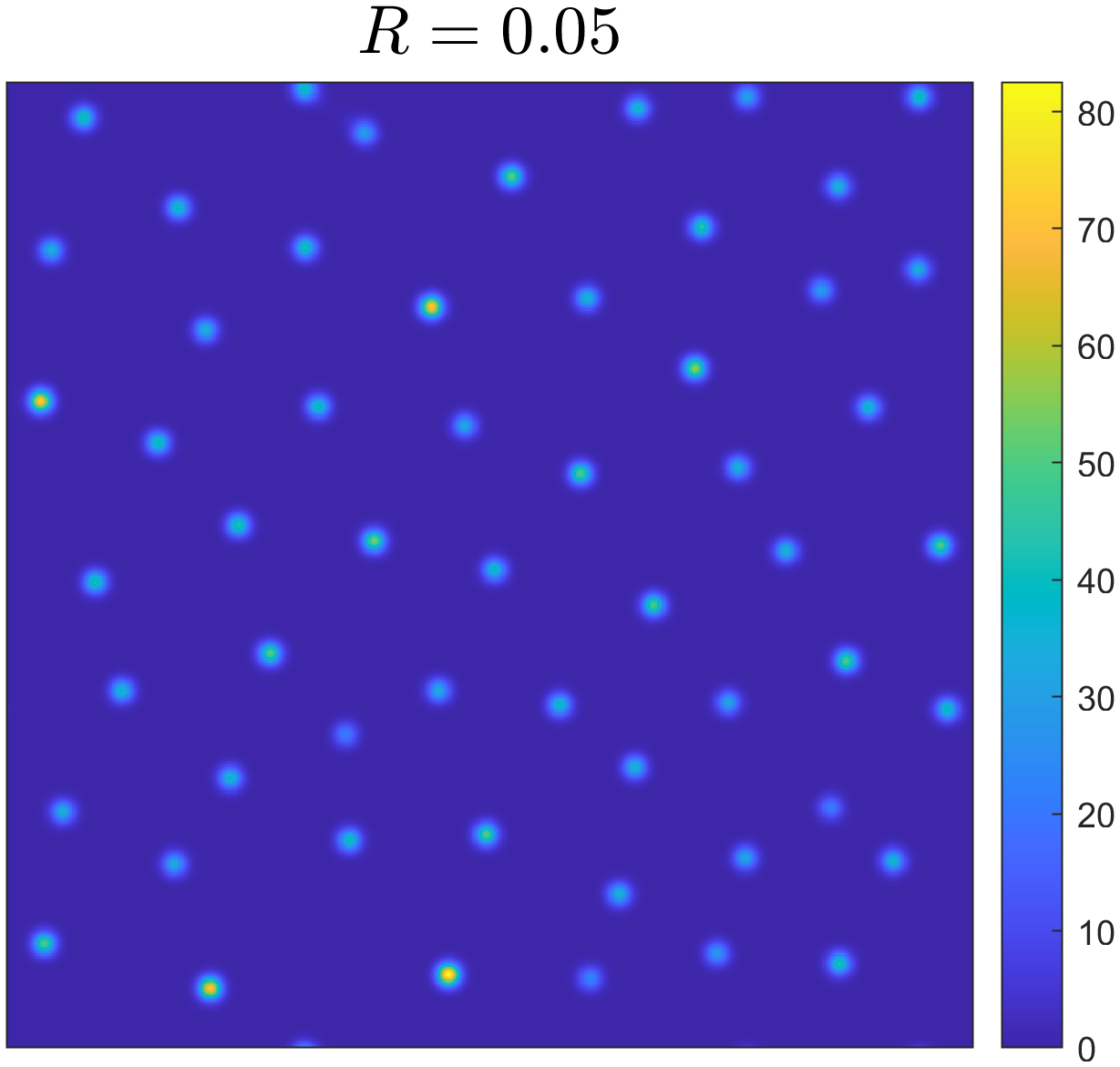}
	\end{subfigure}
	\hfill
	\begin{subfigure}{0.32\textwidth}
		\centering
		\includegraphics[width=\linewidth]{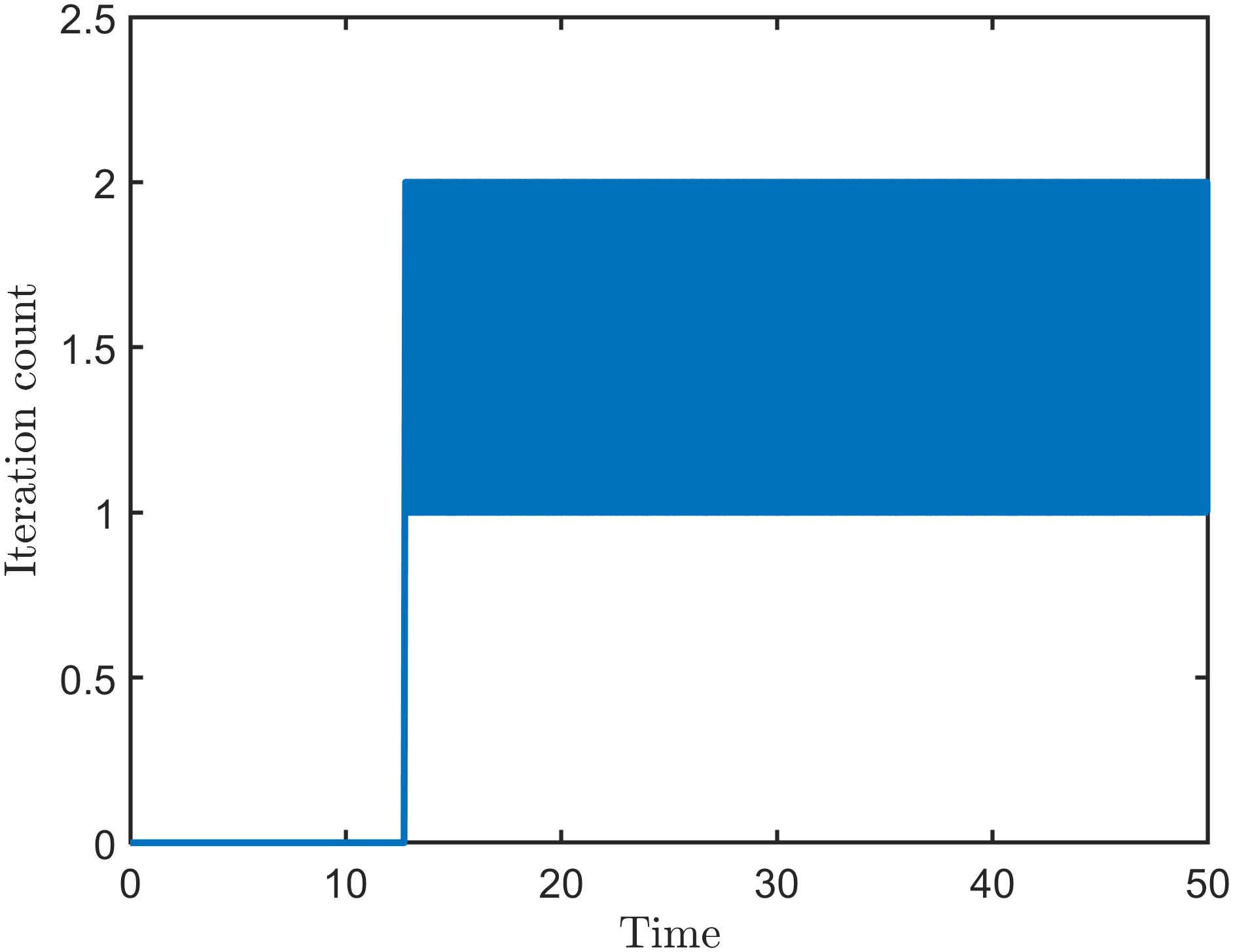}
	\end{subfigure}
	\hfill
	\begin{subfigure}{0.34\textwidth}
		\centering
		\includegraphics[width=\linewidth]{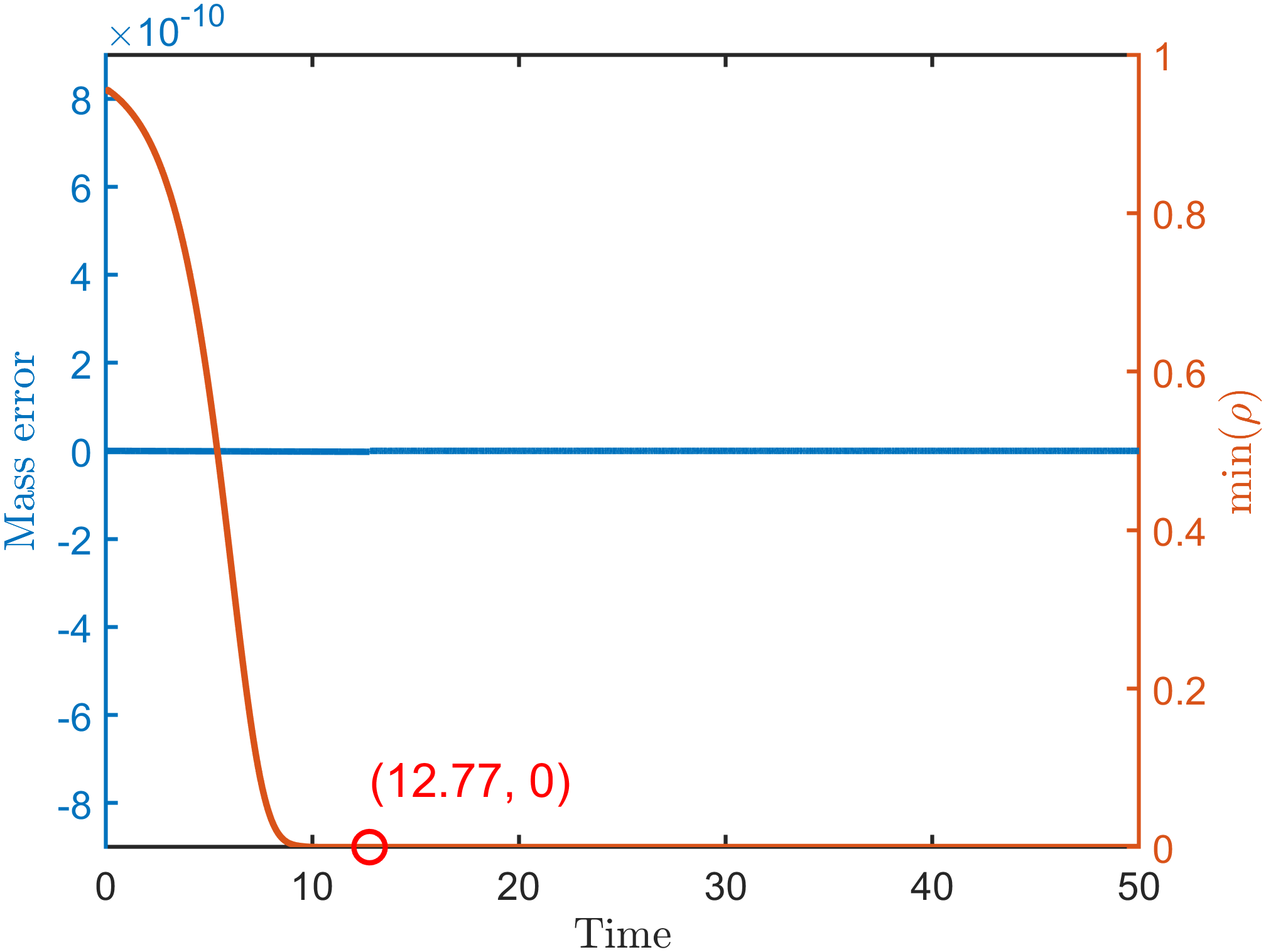}
	\end{subfigure}
	
	\par\vspace{5pt}
	
	% ================= Row 2 =================
	\begin{subfigure}{0.28\textwidth}
		\centering
		\includegraphics[width=\linewidth]{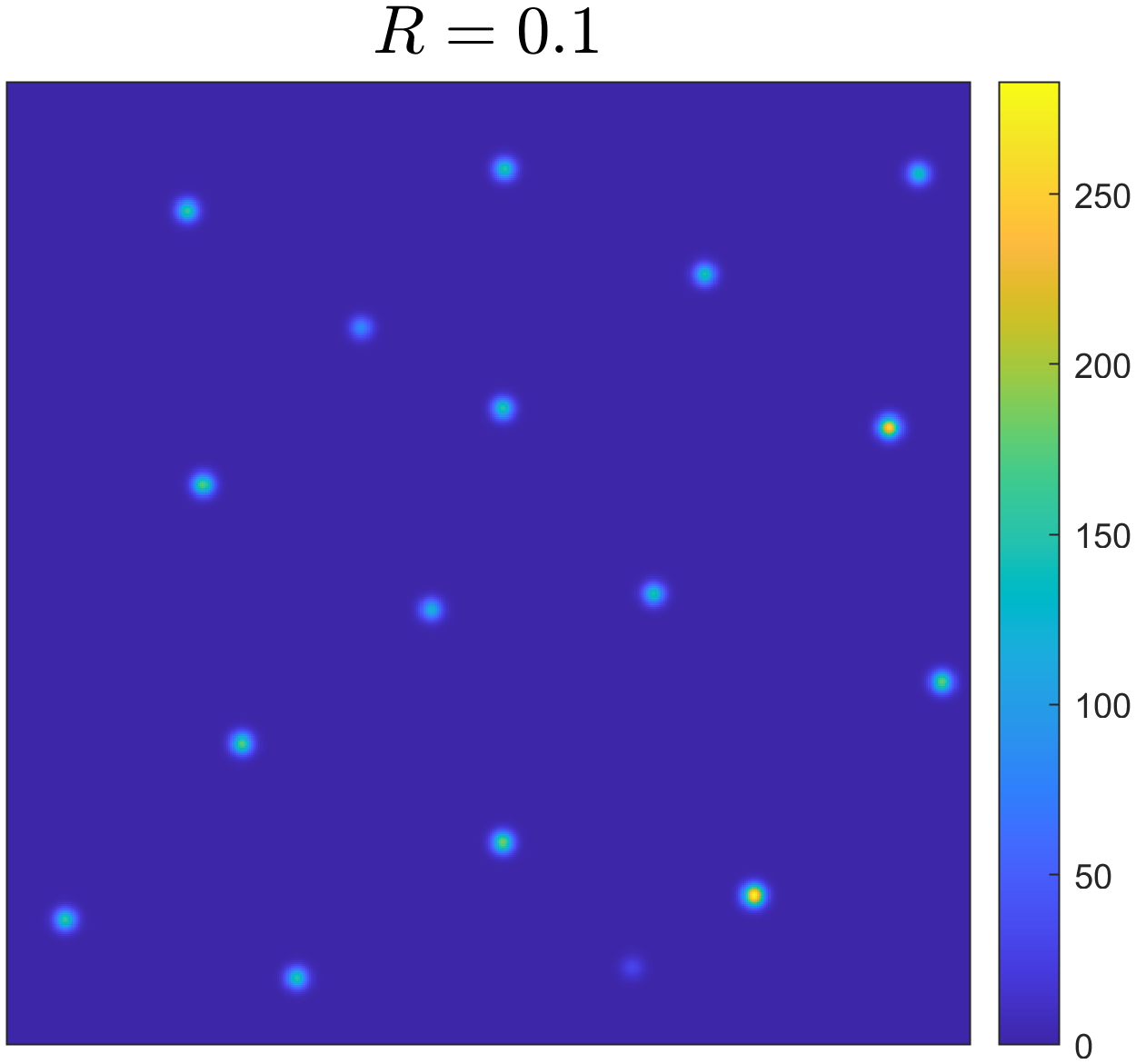}
	\end{subfigure}
	\hfill
	\begin{subfigure}{0.32\textwidth}
		\centering
		\includegraphics[width=\linewidth]{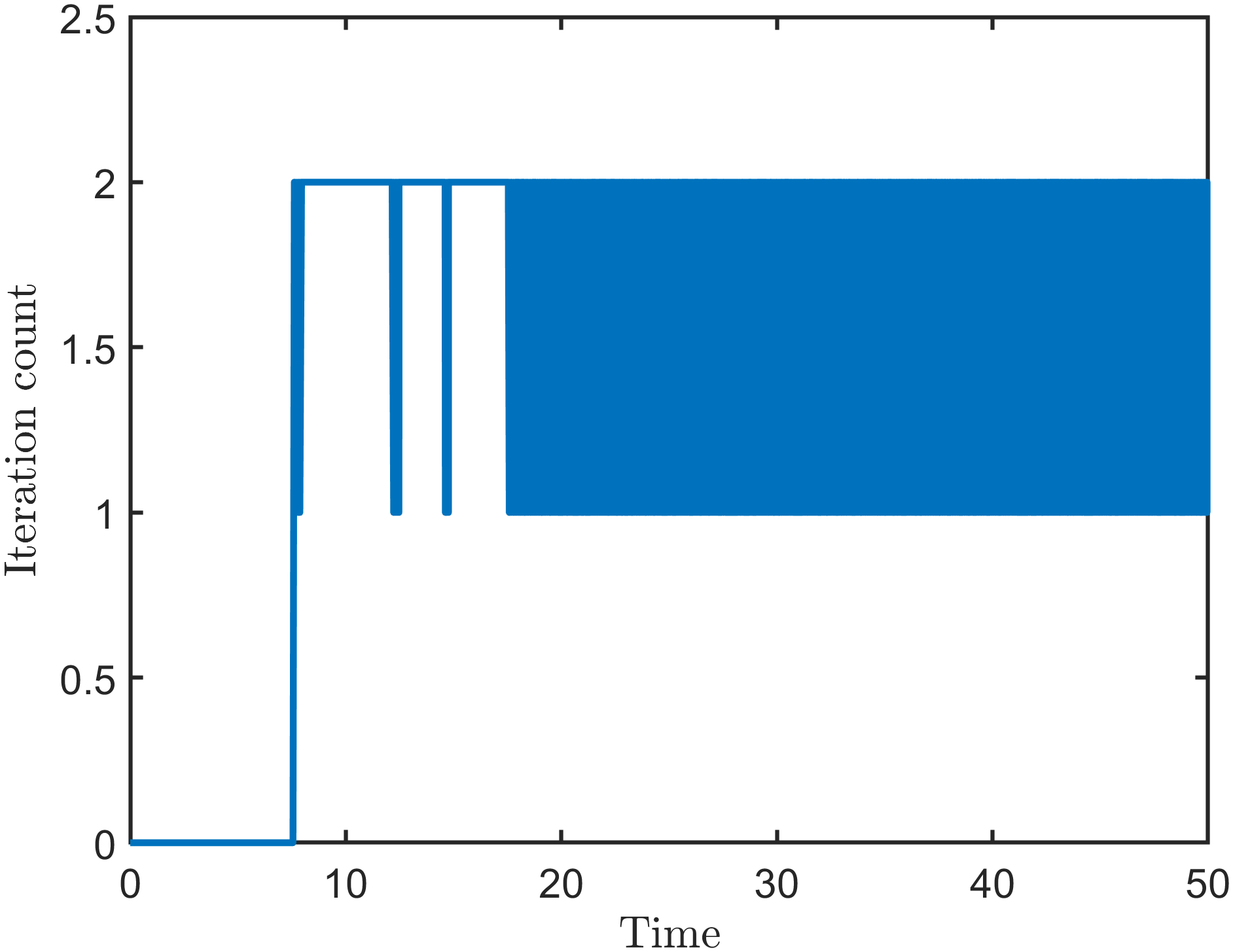}
	\end{subfigure}
	\hfill
	\begin{subfigure}{0.34\textwidth}
		\centering
		\includegraphics[width=\linewidth]{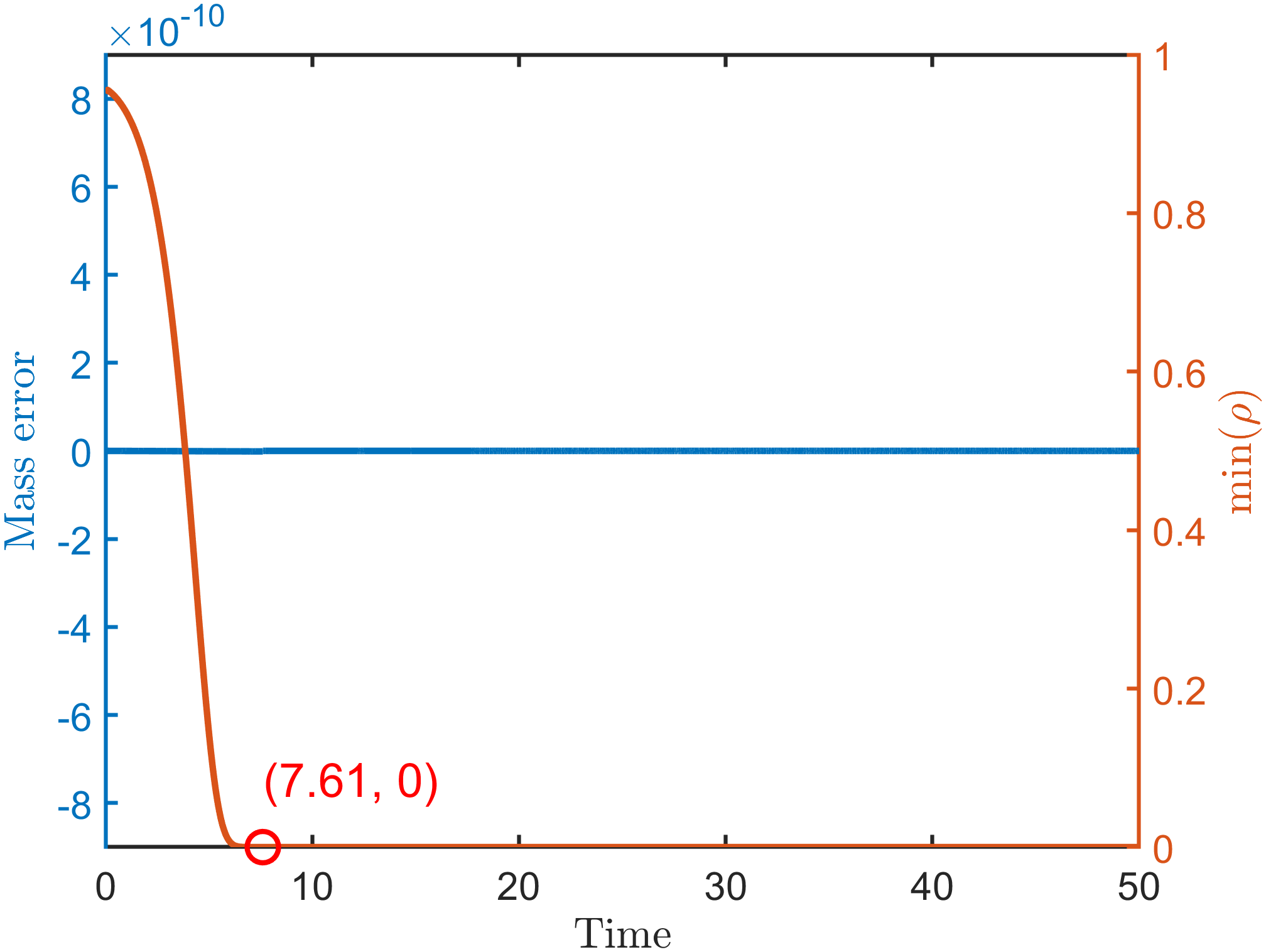}
	\end{subfigure}
	
	\caption{
		Numerical results for the structure-preserving scheme for $D_\rho=10^{-4}$.
		%Top row: results at $R=0.05$; bottom row: results at $R=0.10$.
		Left: snapshot of the opinion density $\rho(x,y)$ at $T=50$.
		Middle: number of iterations used in the Lagrange multiplier projection at each time step of $\rho$.
		Right: evolution of the mass error and the minimum value of $\rho$.
	}
	\label{fig:non-negativity_and_mass_conservation}
	\vspace{-10pt}
\end{figure}

Representative results for $R=0.05$ and $R=0.10$ are shown in Fig.~\ref{fig:non-negativity_and_mass_conservation}. 
The left panels show $\rho(x,y)$ at $T=50$, exhibiting well-resolved clustered patterns. 
The middle panels display the number of projection iterations, indicating computational efficiency. 
The right panels show the evolution of the mass error and $\min(\rho)$. 
The mass error remains at machine precision and $\rho \ge 0$ is preserved at all times, confirming that the scheme maintains both mass conservation and non-negativity.

\end{document}